\documentclass[10pt]{amsart}
\pdfoutput=1
\usepackage{amsmath, amsthm, amssymb,slashed,bm}
\usepackage[pdftex]{graphicx,hyperref}
\usepackage{tikz}
\usetikzlibrary{matrix,arrows,calc}

\usepackage{stmaryrd}
\usepackage[all,arc]{xy}

\usepackage{graphicx}
\usepackage{mathrsfs}
\usepackage[nameinlink,capitalise,noabbrev]{cleveref}

\usepackage{todonotes}

 \theoremstyle{plain}
 \newtheorem{thm}{Theorem}[section]
 \newtheorem{cor}[thm]{Corollary}
 \newtheorem{lem}[thm]{Lemma}
 \newtheorem{prop}[thm]{Proposition}
 \newtheorem*{thm*}{Theorem}
 \newtheorem*{thma}{Theorem A}
 \newtheorem*{thmb}{Theorem B}
 \newtheorem*{thmc}{Theorem C}

 \theoremstyle{definition}
 \newtheorem{defn}[thm]{Definition}
 \newtheorem{notation}[thm]{Notation}
 \newtheorem{ex}[thm]{Example}
 
 \theoremstyle{remark}
 \newtheorem{rmk}[thm]{Remark}
 \newtheorem*{rmk*}{Remark}

\def\beq{\begin{eqnarray}}
\def\eeq{\end{eqnarray}}
\newcommand{\bp}{\begin{proof}[Proof]}
\newcommand{\ep}{\end{proof}}

\newcommand{\sq}{\mathord{/\!\!/}}

\newcommand{\lra}[1]{\overset{#1}{\longrightarrow}}

\newcommand{\Prod}[1]{\underset{#1}{\prod}}

\newcommand{\Coprod}[1]{\underset{#1}{\coprod}}

\def\uL{{\mathbb{L}}} 

\def\R{{\mathbb{R}}}
\def\CG{{G^{(2)}}}

\def\F{{\mathbb{F}}}
\def\N{{\mathbb{N}}}
\def\id{{{\rm id}}}

\def\SL{{\rm SL}}
\def\Spin{ {{\rm Spin}}}
\def\SO{ {\rm SO}}

\def\vol{{\rm vol}}
\def\Ell{\mathcal{E}\ell\ell}

\def\pt{{\rm pt}}
\def\ev{{\rm ev}}
\def\cl{{\rm cl}}
\def\odd{{\rm odd}}
\def\cL{{\mathcal{L}}}

\def\Lat{{\sf {Lat}}}
\def\K{{\rm {K}}}
\def\MU{{\rm {MU}}}
\def\C{{\mathbb{C}}}

\def\P{{\mathbb{P}}}
\def\Q{{\mathbb{Q}}}

\def\Z{{\mathbb{Z}}}
\def\E{{\rm {E}}} 
\def\bE{{\mathbb {E}}} 
\def\M{{\mathbb{M}}}
\def\Mc{{\Bord_{\rm c}^{d|\delta}}}
\def\Mcc{{\Bord_{\rm cc}^{d|\delta}}}
\def\X{{\mathcal{X}}}
\def\Y{{\mathcal{Y}}}
\def\Fun{{\sf Fun}}
\def\Map{{\sf Map}}

\def\Map{{\sf Map}}
\def\Vect{{\sf Vect}} 
\def\Sym{{\sf Sym}} 
 
\def\End{{\sf End}}
\def\Hom{\mathop{\sf Hom}}

\def\Iso{ {{\rm Iso}}}

\def\Bord{\hbox{{\sf Bord}}}

\def\G{{\mathbb{G}}}
\def\L{{\mathbb{L}}}
\def\W{{\mathbb{W}}}

\DeclareMathOperator{\Sub}{\sf Sub}
\DeclareMathOperator{\Cl}{\sf Cl}
\DeclareMathOperator{\Aut}{\sf Aut}

\DeclareMathOperator{\Stack}{\sf Stack}
\DeclareMathOperator{\al}{\alpha}

\DeclareMathOperator{\can}{can}

\DeclareMathOperator{\im}{\rm im}

\def\cX{{\X}}
\def\cY{{\mathcal{Y}}}

\def\coP{\reflectbox{$\mathbb{P}$}}

\def\twocommute{\ensuremath{\rotatebox[origin=c]{30}{$\Rightarrow$}}}

\begin{document}

\title{Power operations in the Stolz--Teichner program}

\author{Tobias Barthel}
\address{Max Planck Institute for Mathematics\\ Bonn, Germany}
\email{tbarthel@mpim-bonn.mpg.de}
\author{Daniel Berwick-Evans}
\address{University of Illinois at Urbana-Champaign \\ Urbana, Illinois, USA}
\email{danbe@illinois.edu}
\author{Nathaniel Stapleton}
\address{University of Kentucky\\ Lexington, Kentucky, USA}
\email{nat.j.stapleton@uky.edu}


\date{\today}
\begin{abstract}
The Stolz--Teichner program proposes a deep connection between geometric field theories and certain cohomology theories. In this paper, we extend this connection by developing a theory of geometric power operations for geometric field theories restricted to closed bordisms. These operations satisfy relations analogous to the ones exhibited by their homotopical counterparts. We also provide computational tools to identify the geometrically defined operations with the usual power operations on complexified equivariant $\K$-theory. Further, we use the geometric approach to construct power operations for complexified equivariant elliptic cohomology. 
\end{abstract}
\maketitle 
\setcounter{tocdepth}{1}

\section{Introduction}

Multiplicative cohomology theories often carry intricate additional structure. When a cohomology theory is built from geometric cocycles, this additional structure is typically inherited from geometrically defined operations on the representing cocycles. For example, symmetric and exterior powers of vector bundles induce operations on topological $\K$-theory, and similar constructions for bordisms give rise to operations on the complex cobordism spectrum~$\MU$. 
The Stolz--Teichner program indicates a deep but mysterious relationship between $d$-dimensional field theories and height~$d$ cohomology theories, with conjectured cocycle models for K-theory and elliptic cohomology from field theories of dimension~1 and~2, respectively~\cite{ST11}. It is natural to ask whether the geometry of field theories fosters interesting operations on these proposed cocycles. 

The first goal of this paper is to introduce geometrically defined power operations on a large class of field theories. These operations are a consequence of power cooperations that exist on the level of categories of closed bordisms, inspired by constructions in the physics literature~\cite{DMVV}. Furthermore, we indicate how to extend these cooperations to arbitrary bordisms. The second goal is to give explicit formulas for the effect of power operations when restricted to a subcategory of tori. The third goal is to compare these power operations with power operations on complexified equivariant $\K$-theory and to deduce a formula for power operations on complexified equivariant elliptic cohomology. Finally, we exhibit a strong analogy between power operations for $d$-dimensional field theories and those for Borel equivariant height~$d$ Morava $\E$-theory. This is noteworthy as the power operations for Morava $\E$-theory are a consequence of the arithmetic geometry of a universal deformation formal group whereas the power operations studied here come from differential geometry.

Informally, our main contributions can be summarized as follows:

\begin{thm*} 
The restriction of any geometric field theory to closed bordisms admits a consistent theory of power operations. In dimensions $1$ and~$2$ these determine power operations on complexified equivariant $\K$-theory and complexified equivariant elliptic cohomology, respectively.
\end{thm*}

These results deepen the proposed analogy relating field theories and cohomology theories. They also tie in with operations that have been studied in the physics literature, e.g., the work of Dijkgraaf--Moore--Verlinde--Verlinde computing the elliptic genus of a symmetric product~\cite{DMVV}. Physically-inspired approaches to power operations have made prior contact with chromatic homotopy theory in the work of Baker~\cite{Baker}, Ando~\cite{andoisogenies,Andopower}, Tamanoi~\cite{Tamanoi2,Tamanoi1}, and Ganter~\cite{Ganterpower,GanterHecke}. The operations studied below are anchored in these ideas, streamlining prior constructions while also tying in with the Stolz--Teichner program.

The term \emph{consistent theory} in the above theorem refers to a collection of compatibility relations satisfied by our geometric power operations that are analogous to the ones for homotopical power operations. An extension of the power cooperation to the full bordism category would then provide a no-go principle: any cohomology theory admitting a theory of geometric cocycles built out of field theories must have a theory of power operations in the sense of \cite{bmms}. Said differently: 

{\it If a category of field theories provides geometric cocycles for a cohomology theory $E^*$, then the representing spectrum $E$ must support the structure of an $H_{\infty}$-ring spectrum.} 

This principle constrains the conceivable zoo of cohomology theories that one might try to relate to field theories. The existence of an $H_{\infty}$-structure constitutes a substantial constraint on a cohomology theory. For example, the spectrum $\K/p$ representing mod $p$ topological $\K$-theory does not admit an $H_{\infty}$-ring structure, as shown by McClure in \cite[Proposition~IX.1.6]{bmms}.

\subsection*{Power cooperations on geometric bordism categories}\label{sec:introFT}

A field theory is a symmetric monoidal functor from a bordism category to the symmetric monoidal category of complex vector spaces. The symmetric monoidal structure on bordisms is disjoint union, whereas the symmetric monoidal structure on vector spaces is the tensor product. Early versions of this definition are due to Segal~\cite{Segal_cft} and Atiyah~\cite{atiyah_tqft}, though we have in mind the more modern approach of Stolz and Teichner~\cite{ST11} that incorporates smoothness, supersymmetry, a model geometry on bordisms (see~\Cref{appen:model}), and equips bordisms with maps to a smooth stack~$\X$~\cite{Stoffel}. The model geometry specifies the (super) dimension $d|\delta$ of the field theory via the dimension of the bordisms involved. We will typically be interested in global quotient stacks $\X=[X \sq G]$ for a finite group~$G$ acting on a compact manifold~$X$. 

In the case of super Euclidean model geometries (see \cref{ex:superEuc}) Stolz and Teichner have conjectures relating field theories over $[X\sq G]$ of super dimension~$1|1$ and~$2|1$ with cocycle models for $G$-equivariant K-theory and $G$-equivariant elliptic cohomology of~$X$, respectively. The geometry of super Euclidean field theories gives compelling evidence for these conjectures~\cite{ST04,florin_11,Han,PokmanPhD,HST,ST11,Stoffel,stoffel_eqchernchar} as beautifully summarized in~\cite[\S1]{ST11}. The compatibility between the power operations constructed below and the corresponding complexified cohomology theories substantiates these connections further still. 

In fact, the existence of our theory of geometric power operations follows from the more fundamental construction of geometric power cooperations on the level of bordism categories. Fixing a model geometry $\M$ and a smooth stack $\X$, let $\Bord^{d|\delta}(\X)$ denote Stolz and Teichner's category of $(d|\delta)$-dimensional bordisms with $\M$-structure over~$\X$ and let $\mathcal{V}$ be their category of topological vector spaces. The power cooperation is a symmetric monoidal functor
\beq
\xymatrix{\coP_n\colon \Bord^{d|\delta}(\X^{\times n}\sq \Sigma_n) \ar@{.>}[r] &  \Bord^{d|\delta}(\X),}
\label{eq:powercoopFT}
\eeq
that induces a power operation on field theories by precomposition
\beq
\xymatrix{\P_n:=\coP_n^*\colon \Fun^\otimes(\Bord^{d|\delta}(\X),\mathcal{V}) \ar@{.>}[r] & \Fun^\otimes(\Bord^{d|\delta}(\X^{\times n}\sq \Sigma_n),\mathcal{V}).}\label{eq:poweropFT} \eeq
In brief, the cooperation~\eqref{eq:powercoopFT} is defined by pulling back along the finite-sheeted cover $\X^{\times n}\times [n]\sq \Sigma_n\to \X^{\times n}\sq \Sigma_n$ and pushing forward along evaluation $\ev\colon \X^{\times n}\times [n]\sq \Sigma_n\to \X$; see~\cref{defn:coop}.

\begin{rmk*}
A complete construction of~\eqref{eq:powercoopFT} requires a technical modification to Stolz and Teichner's definition of a geometric bordism category. The modification uses the 2-fibered product of stacks to define composition  of bordisms rather than the strict fibered products used in~\cite[Defintions~2.13, 2.21, 2.46, 4.4]{ST11}; see~\S\ref{sec:bordext}. We will pursue this elsewhere. In the introduction, maps depending on this modified definition by dotted arrows. Our focus in this paper is on categories of closed bordisms for which such a modification is not necessary. 
\end{rmk*}

Let $\Mc(\X)\hookrightarrow \Bord^{d|\delta}(\X)$ denote the full subcategory of the bordism category consisting of \emph{closed} bordisms with $\M$-structure over $\X$. Equivalently, $\Mc(\X)$ is the subcategory of endomorphisms of the unit for the symmetric monoidal structure on $\Bord^{d|\delta}(\X)$. We construct a power cooperation $\coP_n\colon \Mc(\X^{\times n}\sq \Sigma_n)\to \Mc(\X)$; we anticipate that it arises as the restriction of~\eqref{eq:powercoopFT},
\beq
\begin{tikzpicture}[baseline=(basepoint)];
\node (A) at (0,0) {$\Bord^{d|\delta}(\X^{\times n}\sq \Sigma_n)$};
\node (B) at (6,0) {$\Bord^{d|\delta}(\X)$};
\node (C) at (0,-1.5) {$\Mc(\X^{\times n}\sq \Sigma_n)$};
\node (D) at (6,-1.5) {$\Mc(\X),$};
\draw[->,dotted] (A) to node [above] {$\coP_n$} (B);
\draw[->] (C) to  (A);
\draw[->] (D) to (B);
\draw[->] (C) to node  [above] {$\coP_n$} (D);
\path (0,-.75) coordinate (basepoint);
\end{tikzpicture}\label{FTrestrict}
\eeq 
but our construction does not depend on \eqref{eq:powercoopFT}. Our reason for believing in the extension to the bordism category is that the power cooperation is determined by constructions on the entire category of supermanifolds with $\M$-structure; see \cref{sec:bordext} for a further elaboration on this point. 

\begin{thma}[\Cref{sec:cc} and \cref{sec:consistentcoop}] \label{thma}
The power cooperation $\coP_n \colon \Mc(\X^{\times n}\sq \Sigma_n) \to \Mc(\X)$ is a symmetric monoidal map of stacks on the site of supermanifolds. The geometric power cooperations satisfies the identities dual to those satisfied by power operations \cite[VIII.1.1]{bmms}. 
\end{thma}

\subsection{Computing power cooperations on super tori}
We show that one can further restrict~\eqref{FTrestrict} to subcategories of $\Mc(\X)$ that are \emph{cover closed}, meaning all finite covers of objects are also in the given subcategory. For the field theories of interest in the Stolz--Teichner program, a particularly convenient subcategory is the one generated by (super) tori whose map to~$\X$ is essentially constant (in the stacky sense); this is a super-version of the iterated inertia stack or ghost loop stack of~$\X$ that we denote by $\mathcal{L}_{0}^{d|\delta}(\X) \subset \Mc(\X)$. Restricting a field theory to the subcategory $\mathcal{L}_{0}^{d|\delta}(\X)\subset \Bord^{d|\delta}(\X)$ is called \emph{dimensional reduction}. 

Closing $\mathcal{L}_{0}^{d|\delta}(\X)\subset \Bord^{d|\delta}(\X)$ under disjoint unions of super tori is equivalent to taking the free symmetric monoidal stack on $\mathcal{L}_{0}^{d|\delta}(\X)$, denoted $\Sym(\mathcal{L}_{0}^{d|\delta}(\X))\subset \Mc(\X)$. This is a cover closed substack because finite covers of tori are disjoint unions of tori. Consequently, the power cooperation may be restricted to substacks of the form $\Sym(\mathcal{L}_{0}^{d|\delta}(\X))$. 

Now assume that $\X$ is equivalent to a global quotient stack $[X \sq G]$, where $X$ is a manifold and $G$ is a finite group. We construct explicit atlases for the stacks $\mathcal{L}_{0}^{d|\delta}(X \sq G)$, which allow for explicit descriptions of the power cooperations. This allows us to give formulas for the power operations in cases of interest. 

\begin{thmb}[\Cref{thm:mainthm}] \label{thmbintro}
For any global quotient $[X \sq G]$, there is an explicit atlas $\mathcal{U}(X \sq G)$ of $\mathcal{L}^{d|\delta}_0(X\sq G)$ and a $2$-commutative diagram of stacks
\[
\xymatrix{\mathcal{U}(X^{\times n}\sq G\wr \Sigma_n) \ar[r]^-{\widetilde{\coP}_n} \ar[d] & \coprod_{i \in \mathbb{N}}(\mathcal{U}(X\sq G))^{\times i} \ar[d] \\ \mathcal{L}^{d|\delta}_0(X^{\times n}\sq G\wr \Sigma_n) \ar[r]^-{\coP_n} & \Sym(\mathcal{L}^{d|\delta}_0(X\sq G)),}
\]
where $\widetilde{\coP}_n$ admits an explicit description depending on a number of choices. 
\end{thmb}

Since a functor from a category to a symmetric monoidal category uniquely determines a symmetric monoidal functor out of the free symmetric monoidal category on the source, the functor $\coP_n$ in the diagram is equivalent to the data of the (symmetric monoidal) power cooperation
\[
\coP_n\colon \Sym(\mathcal{L}^{d|\delta}_0(X^{\times n}\sq G\wr \Sigma_n))\to \Sym(\mathcal{L}_{0}^{d|\delta}(X \sq G)).
\]
 
\subsection*{Comparing with power operations in cohomology}
Restricting a field theory along the inclusion of closed bordisms, $\Mc(\X)\hookrightarrow \Bord^{d|\delta}(\X)$, gives a map \beq\label{eq:FTres}
\xymatrix{{\rm res}\colon \Fun^\otimes(\Bord^{d|\delta}(\X),\mathcal{V})\ar[r] & \Fun^\otimes(\Mc(\X),\C)=: C^\infty_\otimes(\Mc(\X)).}
\eeq
By definition, $\Mc(\X)$ is the subcategory of endomorphisms of the monoidal unit in $\Bord(\X)$, so this restriction lands in endomorphisms of the unit in $\mathcal{V}$, i.e., automorphisms of the 1-dimensional vector space, $\End(\C)\cong \C$. Since functors from a stack to~$\C$ are usually called \emph{functions} on the stack, we take the shorthand~$C^\infty_\otimes(\Mc(\X))$ for this functor category. The subscript $\otimes$ indicates that these functions satisfy a condition: disjoint union in~$\Mc(\X)$ is compatible with multiplication in~$\C$. In particular, the restriction map
\beq
C^\infty_\otimes(\Mc(\X))\to C^\infty(\Mcc(\X))\label{eq:ccmap}
\eeq
is an isomorphism, where $\Mcc(\X)\subset \Mc(\X)$ is the subcategory of closed \emph{connected} bordisms. The image of a field theory in $C^\infty(\Mcc(\X))$ under~\eqref{eq:FTres} and~\eqref{eq:ccmap} is called the partition function. 

Combined with the observations above, \cref{thm:mainthm} determines a multiplicative (but not additive) map
\[
\xymatrix{\P_n = C^\infty(\coP_n)\colon C^\infty(\mathcal{L}_{0}^{d|\delta}(\X))\cong C^\infty_\otimes(\Sym(\mathcal{L}_{0}^{d|\delta}(\X))) \ar[r] &  C^\infty(\mathcal{L}_{0}^{d|\delta}(\X^{\times n}\sq \Sigma_n))}
\]
which we call the \emph{geometric power operation}. 

When $d|\delta=1|1$ and $2|1$, functions on $\mathcal{L}_{0}^{d|\delta}(X\sq G)$ can be identified with cocycle models for complexified equivariant K-theory and complex analytic equivariant elliptic cohomology, respectively~\cite{DBE_EquivTMF,stoffel_eqchernchar}. Hence, the power cooperation induces a power operation in each of these equivariant cohomology theories. Furthermore, Theorem B provides a formula for the power operation in terms of the pullback of functions along the map $\widetilde{\coP}_n$ between ordinary supermanifolds. 

\begin{thmc}
The power operation $\P_n$ specializes as follows:
\begin{enumerate}
    \item In dimension $1|1$, $\P_n$ is compatible with the $n^{\rm th}$ power operation on equivariant $\K$-theory via the equivariant Chern character, see \cref{thm:main11}.
    \item In dimension $2|1$, the operations $\P_n$ provide a theory of power operations on complexified equivariant elliptic cohomology (\cref{prop:21formula}) that specializes to the expected Adams operations and are closely related to the formula for the character of the power operation on height~2 Morava $\E$-theory, see  \cref{thm:TMFAdams}.
\end{enumerate}
\end{thmc}

There have been some exciting recent developments in the understanding of integral equivariant elliptic cohomology, see \cite{gepnermeier_eqell,lurie_elliptic3}. To date, these have not given an ultracommutative (or $G_{\infty}$ in the sense of \cite[Remark 5.1.16]{schwedeglobal}) structure on the equivariant spectrum. Hence, in contrast to the situation for $\K$-theory, there is no complete theory to which we can compare the formula for the total power operation in equivariant elliptic cohomology. 

Finally, in \cref{Sec:Ethy}, we explain that the formulas for power operations extracted from restriction to closed $d$-dimensional bordisms closely mirror those for power operations in $\E$-theory obtained by the first and third author~\cite{charpo}.

\subsection*{Notation and conventions} 

The main constructions in this paper take place in stacks on the site of supermanifolds; see Appendix~\ref{append:A} for a review. 
In short, a \emph{stack} is a presheaf of groupoids on the site of supermanifolds satisfying descent for all covers; covers are surjective submersions of supermanifolds. All diagrams involving stacks should be assumed to be 2-commutative unless stated otherwise; we remind that 2-commutativity is additional data, though often we do not make this explicit. Frequently our stacks will be presented as groupoid objects in supermanifolds. This uses the 2-functor from the 2-category of Lie groupoids, smooth functors, and smooth natural transformations to the 2-category of stacks. For the Lie groupoid $\mathcal{G}=\{\mathcal{G}_1\rightrightarrows \mathcal{G}_0\}$, we use the notation $[\mathcal{G}]$ to denote the value of this functor, i.e., the corresponding stack. We follow the usual convention where the same letter (e.g., $S\in {\sf SMfld}$) is used to denote a supermanifold~$S$, its associated Lie groupoid $\{S\rightrightarrows S\}$, and its underlying stack $[\{S\rightrightarrows S\}]$.

\subsection*{Acknowledgements}

The first author was partly supported by the DNRF92 and the European Unions Horizon 2020 research and innovation programme under the Marie Sklodowska-Curie grant agreement No.~751794. The third author was supported by NSF grant DMS-1906236.

We would like to thank Matt Ando, Nora Ganter, Tom Nevins, Charles Rezk, Chris Schommer-Pries, Stephan Stolz, and Peter Teichner for helpful discussions about the subject matter of this paper. 
Moreover, we would like to thank the Max Planck Institute for Mathematics and the SFB 1085 at the University of Regensburg for its hospitality. 

\tableofcontents

\section{Geometric power cooperations}\label{sec:geompowerops}

In this section we construct and investigate geometric power cooperations with an emphasis on the cooperations on moduli stacks of super tori. We establish compatibility relations between these cooperations and also sketch the extension of the geometric power cooperations to bordism categories.

\subsection{Constructing geometric power (co)operations}\label{sec:21}

Fix a model geometry $\M$ (see~\Cref{appen:model}). For a stack $\X$, let $\M(\X)$ be the stack on the site of supermanifolds (see~\Cref{rmk:itsastack}) whose value on a supermanifold $S$ is the groupoid $\M(\X)(S)$ with objects the set of correspondences 
\[
S\leftarrow T\to \X,
\]
where $T\to S$ is an $S$-family of supermanifolds with $\M$-structure and~$T\to \X$ is a map of stacks. Morphisms in $\M(\X)(S)$ consist of diagrams
\beq
\begin{tikzpicture}[baseline=(basepoint)];\label{eq:twotriangles}
\node (A) at (0,-.75) {$S$};
\node (B) at (3,0) {$T$};
\node (E) at (3,-1.5) {$T'$};
\node (G) at (6,-.75) {$\X$,};
\node (H) at (4,-.75) {$\twocommute$};
\draw[->] (B) to  (A);
\draw[->] (B) to (E);
\draw[->] (E) to (A);
\draw[->] (B) to  (G);
\draw[->] (E) to (G);
\path (0,-.75) coordinate (basepoint);
\end{tikzpicture}
\eeq
where $T\to T'$ is a fiberwise isometry of supermanifolds with $\M$-structure, the left hand triangle strictly commutes, and the right hand triangle $2$-commutes. Given $f\colon S \to S'$, the induced functor $\M(\X)(S') \to \M(\X)(S)$ is given by pulling back $T \to S'$ to $S$ along the given map $f$. This is well-defined since the condition on $T \to S'$ is a fiberwise condition.

\begin{ex} \label{ex:globalquotient}
We are primarily interested in global quotient stacks. These are stacks of the form $\X = [X \sq G]$, where $X$ is a smooth manifold equipped with an action of a finite group $G$. In this case, a map
\[
S \to \M([X\sq G])
\]
is equivalent to an $S$-family $T \to S$ of supermanifolds with $\M$-structure and the choice of a principal $G$-bundle $E \to T$ equipped with a $G$-equivariant map $E \to X$. Thus this data may be displayed as
\[
S \leftarrow T \leftarrow E \rightarrow X.
\]
Assume another map $S \to \M([X \sq G])$ is given by 
\[
S \leftarrow T' \leftarrow E' \rightarrow X.
\] 
An isomorphism in $\M([X \sq G])(S)$ corresponds to a fiberwise isometry $f \colon T \to T'$ of supermanifolds with $\M$-structure and a choice of isomorphism of principal $G$-bundles $E \cong f^*E'$ over $T$. This choice of isomorphism of principal $G$-bundles corresponds to the  $2$-commutative triangle in diagram \eqref{eq:twotriangles}. 
\end{ex}

The disjoint union of $S$-families of supermanifolds with $\M$-structure and the disjoint union of morphisms promotes $\M(\X)$ to a symmetric monoidal stack; see~\cite[Definition~7.21]{HKST} or \Cref{defn:symmonstack} for the definition of a symmetric monoidal stack. Furthermore, a map $f\colon \X\to \mathcal{Y}$ of stacks induces a morphism $f_*\colon \M(\X)\to \M(\mathcal{Y})$ of symmetric monoidal stacks by postcomposition, and a 2-morphism $f\Rightarrow g$ between morphisms of stacks $f,g\colon \X\to \mathcal{Y}$ gives a 2-morphism between~$f_*$ and~$g_*$, 
\beq
\begin{tikzpicture}
\node (C) at (2,-.75) {$S$};
\node (B) at (4,-.75) {$T$};
\node (G) at (6,-.75) {$\X$};
\node (A) at (9,-.75) {$\mathcal{Y}$,};
\node (AA) at (7.5,-.75) {$\Downarrow$};
\draw[->] (B) to  (G);
\draw[->,bend left] (G) to node [above] {$f$} (A);
\draw[->,bend right] (G) to node [below] {$g$}  (A);
\draw[->] (B) to (C);
\end{tikzpicture}\nonumber
\eeq
by composing the given 2-morphism $f\Rightarrow g$ with the identity 2-morphism on $T\to \X$. When referring to objects in $\M(\X)$ below, we will often drop the family parameter $S$ from the notation, writing $T\to \X$ for such an object, where $T$ is an $S$-family of supermanifolds with $\M$-structure.

The geometric power cooperation will be defined as a push-pull construction for the correspondence 
\[
\X^{\times n}\sq \Sigma_n\stackrel{\pi}{\leftarrow} (\X^{\times n}\times \underline{n})\sq \Sigma_n\stackrel{\ev}{\to}\X,
\]
where $\pi$ is induced by the projection $\underline{n}\sq \Sigma_n\to \pt\sq \Sigma_n$, $\ev$ is the evaluation map, and we recall that the quotients by $\Sigma_ n$ are taken in the category of stacks; see \Cref{def:symstack}. Define a map of stacks
\[
\pi^!\colon \M(\X^{\times n}\sq \Sigma_n) \to \M((\X^{\times n}\times \underline{n})\sq \Sigma_n) 
\]
that sends an object $T\to \X^{\times n}\sq \Sigma_n$ to the object $\widetilde{T}\to (\X^{\times n}\times \underline{n})\sq \Sigma_n$ characterized by the diagram
\beq
\begin{tikzpicture}
\node (A) at (3,0) {$\widetilde{T}$};
\node (B) at (3,-1.5) {$T$};
\node (C) at (6,0) {$(\X^{\times n}\times \underline{n})\sq \Sigma_n$};
\node (D) at (6,-1.5) {$\X^{\times n}\sq \Sigma_n$};
\node (E) at (9,0) {$\underline{n}\sq \Sigma_n$};
\node (F) at (9,-1.5) {$\pt\sq \Sigma_n$.};
\draw[->] (A) to  (B);
\draw[->] (C) to node [left] {$\pi$} (D);
\draw[->] (E) to (F);
\draw[->] (A) to (C);
\draw[->] (B) to (D);
\draw[->] (C) to (E);
\draw[->] (D) to (F);
\end{tikzpicture}\nonumber
\eeq
Both squares are 2-pullback squares and the projection $\underline{n}\sq \Sigma_n\to \pt\sq \Sigma_n$ is the universal $n$-sheeted cover, where we identify $\Sigma_n$ with $\Aut(\underline{n})$. Hence, $\widetilde{T}$ is the 2-pullback of the $n$-sheeted cover $\X^{\times n}\times \underline{n}\sq\Sigma_n\to \X^{\times n}\sq \Sigma_n$ along $\pi$. The value of $\pi^!$ on morphisms is similarly defined via pullback.

\begin{defn} \label{defn:coop}
The \emph{$n^{\rm th}$ geometric power cooperation} on $\M(\X)$ is the functor
\[
\coP_n\colon \M(\X^{\times n}\sq \Sigma_n)\to \M(\X)
\]
defined by $\coP_n:=\ev_*\circ \pi^!$, the composition of the pullback along $\pi$ and the pushforward along the evaluation map $\ev$.
\end{defn}

The above defines $\coP_n$ as a functor between stacks up to unique isomorphism, since a 2-pullback is part of the construction of $\coP_n$; it also respects the monoidal structure. 

\begin{lem} \label{lem:copsym} The geometric power cooperation $\coP_n$ is symmetric monoidal. 
\end{lem}
\bp
Suppose an isomorphism~$T\cong T_1\coprod T_2$ witnesses $T$ as a disjoint union. Then this isomorphism pulls back to an isomorphism between covering spaces $\widetilde{T}\cong \widetilde{T}_1\coprod \widetilde{T}_2$. Similarly, isomorphisms between disjoint unions pullback to isomorphisms between their covering spaces. Further, every cover of the empty set is the empty set and the empty set is the unit of the symmetric monoidal structure.
\ep

\subsection{Cover-closed substacks and moduli stacks of super tori}\label{sec:cc}

We will write $\M(-)$ when viewing $\M$ as a 2-functor from the $2$-category of stacks to the $2$-category of stacks on the site of supermanifolds as described in the previous section. One can consider subfunctors of $\M(-)$ that are compatible with the geometric power cooperation.

\begin{defn} We say that a sub 2-functor $\N(-) \subset \M(-)$ is \emph{cover closed} if, for all stacks~$\X$, $\N(\X)$ is a symmetric monoidal substack of $\M(\X)$ and, for every supermanifold $S$, any finite cover of an object in $\N(\X)(S)$ is also in $\N(\X)(S)$. i.e., if $S \leftarrow T\to \X$ is an object of $\N(\X)(S)$, then, for any finite-sheeted covering space~$\widetilde{T}\to T$, we also have that $S\leftarrow \widetilde{T}\to \X$ is an object in $\N(\X)(S)$, where the maps to $\X$ and $S$ come from the compositions $\widetilde{T}\to T\to \X$ and $\widetilde{T}\to T\to S$. 
\end{defn} 

We adopt this definition for the following reason. 

\begin{lem} \label{lem:powerrestr}
Geometric power cooperations $\coP_n$ restrict to any cover closed $\N(-) \subset \M(-)$, i.e., the following diagram 2-commutes
\[
\xymatrix{\N(\X^{\times n}\sq \Sigma_n) \ar[r]^-{\coP_n} \ar[d] & \N(\X) \ar[d] \\
\M(\X^{\times n}\sq \Sigma_n) \ar[r]_-{\coP_n} & \M(\X).}
\]
\end{lem}

\begin{rmk} The subcategory $\Mc(\X)\subset \M(\X)$ of closed bordisms with $\M$-structure is an example of a cover closed subcategory. With $\N(\X)=\Mc(\X)$ this gives the first statement in Theorem~A. 
\end{rmk}

From this point on, we shall assume that we are given a model geometry~$\M$ whose underlying model space is the supermanifold~$\R^{d|\delta}$ and whose group of isometries contains the standard translation group~$\bE^d$ that acts on $\R^{d|\delta}$ through the canonical inclusion $\bE^d\cong \R^d\subset \R^{d|\delta}$. Let $\uL = \Z^d$. 

\begin{defn}
An \emph{$S$-family of based lattices} is an inclusion $\Lambda \colon S \times \uL \hookrightarrow S \times \bE^d$ of abelian group objects in the category of supermanifolds over $S$ with the property that the induced map $\Lambda \otimes \R \colon S \times \R^d \to S \times \bE^d$ is an isomorphism.
\end{defn}

\begin{defn}\label{defn:supertori} For a model geometry as described above, an \emph{$S$-family of super tori} is a quotient of the form ($S\times \R^{d|\delta})/\uL$, where $\uL$ acts on $S\times \R^{d|\delta}$ through an $S$-family of based lattices $\Lambda \colon S \times \uL \hookrightarrow S\times \bE^d\subset S\times \bE^{d|\delta}$. Let $T^{d|\delta}_\Lambda \to S$ denote the $S$-family of super tori associated with an $S$-family of lattices. \end{defn}

\cref{ex:supertori} explains why an $S$-family of super tori is an $S$-family of supermanifolds with $\M$-structure.

\begin{defn}\label{defn:stori} Let $\mathcal{L}^{d|\delta}(\X)\subset \M(\X)$ denote the full substack of super tori over $\X$. Thus an object in $\mathcal{L}^{d|\delta}(\X)(S)$ locally has the structure of an $S$-family of super tori with a map to $\X$.
\end{defn}

Because $\uL$ is required to act freely on $S \times \R^{d|\delta}$ in \cref{defn:supertori}, the stack $[(S \times \R^{d|\delta}) \sq \uL]$ is representable by the supermanifold $(S \times \R^{d|\delta})/\uL$. Specifically, there is a canonical equivalence 
\begin{equation}\label{eq:representablequotient}
{\Stack}([(S \times \R^{d|\delta})\sq \uL],\X)\stackrel{\sim}{\leftarrow} {\Stack}((S \times \R^{d|\delta})/\uL,\X)  
\end{equation}
induced by the quotient map of Lie groupoids
\[
(S \times \R^{d|\delta})\sq \uL \rightarrow (S \times \R^{d|\delta}) / \uL.
\]
There is also a zig-zag of Lie groupoids
\[
(S \times \R^{d|\delta})/ \uL\stackrel{\sim}{\leftarrow} (S \times \R^{d|\delta}) \sq \uL \to (S \times \R^{0|\delta}) \sq \uL
\]
determining a map of stacks
\beq
(S\times \R^{d|\delta})/\uL\to [(S \times \R^{0|\delta}) \sq \uL]\label{eq:factorthis}
\eeq
There is an important substack of $\mathcal{L}^{d|\delta}(\X)$ consisting of families of super tori equipped with certain degenerate maps $T^{d|\delta}_\Lambda \to \X$. 

\begin{defn}\label{defn:edinvariantmaps}
Let $\mathcal{L}^{d|\delta}_0(\X)\subset \mathcal{L}^{d|\delta}(\X)$ denote the full substack of super tori over $\X$ that are locally isomorphic to maps that factor through~\eqref{eq:factorthis}, namely as $(S \times \R^{d|\delta})/ \uL \to [(S \times \R^{0|\delta}) \sq \uL]\to \X$.
\end{defn}

\begin{rmk}\label{rmk:edinvariantmaps}
We observe that the map
\[
[(S \times \R^{d|\delta}) \sq \uL] \to [(S \times \R^{0|\delta}) \sq \uL]
\]
is induced by the quotient by the (free) $\bE^d$-action on $\R^{d|\delta}$. In this sense, $\mathcal{L}^{d|\delta}_0(\X)$ is the substack of $\bE^d$ invariant maps. Indeed, the notation comes from viewing $\mathcal{L}^{d|\delta}(\X)$ as a kind of $d$-fold free loop space with geometry and $\mathcal{L}^{d|\delta}_0(\X)$ as the subspace of constant $d$-fold super loops. 
\end{rmk}

Let $\Sym(\mathcal{L}^{d|\delta}(\X))$ and $\Sym(\mathcal{L}_{0}^{d|\delta}(\X))$ denote the free symmetric monoidal stacks on $\mathcal{L}^{d|\delta}(\X)$ and on $\mathcal{L}^{d|\delta}_0(\X)$, respectively; see \Cref{superstacks} for the construction of symmetric powers and the free symmetric monoidal stack on a given stack. Following \Cref{superstacks}, let $\Sym^n(\X) = \X^{\times n} \sq \Sigma_n$ and let $\Sym^{\le n}(\mathcal{\X}):=\coprod_{i\le n} \mathcal{\X}^{\times i}\sq \Sigma_i$.

\begin{prop}
The stack $\Sym(\mathcal{L}^{d|\delta}(\X))$ is equivalent to the full substack of~$\M(\X)$ with objects locally isomorphic to disjoint unions of $S$-families of super tori over~$\X$. 
\end{prop}

\begin{proof}
Since $\mathcal{L}^{d|\delta}(\X) \subseteq \M(\X)$ is a substack and $\M(\X)$ is a symmetric monoidal stack, the free-forgetful adjunction between stacks and symmetric monoidal stacks give a canonical map
\begin{equation} \label{eq:freemap}
\Sym(\mathcal{L}^{d|\delta}(\X))\to \M(\X).
\end{equation}

First we will prove that the essential image of \eqref{eq:freemap} consists of objects locally isomorphic to disjoint unions of $S$-families of super tori over $\X$. The gluing data for such an object consists of a permutation of components followed by a coproduct of isometries of super tori. These are precisely the local isomorphisms in the source of \eqref{eq:freemap}.

To show that \eqref{eq:freemap} is fully faithful, note that an isomorphism between objects locally isomorphic to disjoint unions of $S$-families of super tori over $\X$ is a collection of local isomorphisms. Again, these are determined by a permutation of the connected components followed by an isometry over $S$. These are precisely the local isomorphisms in the source of $\eqref{eq:freemap}$.
\end{proof}

\begin{rmk}
In the remainder of the paper, we will identify $\Sym(\mathcal{L}^{d|\delta}(\X))$ with its essential image in $\M(\X)$ along the map in $\eqref{eq:freemap}$. This results in an equivalent category, but this equivalence is not an equality: objects in $\Sym(\mathcal{L}^{d|\delta}(\X))$ over $S$ come with a (local) decomposition into a disjoint union of super tori, whereas the essential image in $\M(\X)$ consists of families of super manifolds with $\M$-structure for which there exists such a local description as a disjoint union. For example, the target in \cref{lem:factor} is an object in $\M(\X)$ that is in the essential image of $\Sym(\mathcal{L}^{d|\delta}(\X))$ but not in the image.  
\end{rmk}

\begin{cor}\label{lem:coverclosed} 
For each stack $\X$, the full substacks $\Sym(\mathcal{L}^{d|\delta}(\X))$ and $\Sym(\mathcal{L}^{d|\delta}_0(\X))$ of $\M(\X)$ are cover closed. 
\end{cor}
\bp
It suffices to check that the covering spaces of a connected super torus are disjoint unions of super tori. Up to isomorphism, such covers are given by $\widetilde{T}_\Lambda^{d|\delta}\cong \coprod_k (S\times \R^{d|\delta})/\uL_k$ for $S$-families of sublattices $\uL_k \subset \uL$. 
\ep

As a consequence of \Cref{lem:powerrestr} and~\Cref{lem:coverclosed}, geometric power cooperations restrict to give functors on moduli spaces of super tori:

\begin{cor}\label{cor:cooprestriction}
There is a commutative diagram 
\beq
\xymatrix{\Sym(\mathcal{L}_{0}^{d|\delta}(\X^{\times n}\sq \Sigma_n)) \ar@{-->}[r]^-{\coP_n} \ar[d] &  \Sym(\mathcal{L}_{0}^{d|\delta}(\X)) \ar[d] \\
\M(\X^{\times n}\sq \Sigma_n) \ar[r]_-{\coP_n} & \M(\X).} \label{eq:toripower}
\eeq
and similarly for $\mathcal{L}^{d|\delta}(\X^{\times n} \sq \Sigma_n)$.
\end{cor}

Since $\coP_n$ is symmetric monoidal, it is determined by its value on \emph{connected} tori, i.e., the substack $\mathcal{L}_{0}^{d|\delta}(\X^{\times n}\sq \Sigma_n)\subset \Sym(\mathcal{L}_{0}^{d|\delta}(\X^{\times n}\sq \Sigma_n))$. For later purposes, we record the following result:

\begin{lem}\label{lem:cooprestriction}
The $n^{\rm th}$ geometric power cooperation restricts restricts to a functor
\[
\coP_n\colon\mathcal{L}_{0}^{d|\delta}(\X^{\times n} \sq \Sigma_n) \to \Sym^{\le n}(\mathcal{L}_{0}^{d|\delta}(\X)),
\]
and similarly for $\mathcal{L}^{d|\delta}(\X^{\times n} \sq \Sigma_n)$.
\end{lem}
\begin{proof}
By construction, the value of the $n^{\rm th}$ geometric power cooperation on an object in $\mathcal{L}_{0}^{d|\delta}(\X^{\times n}\sq \Sigma_n)(S)$ has, locally in $S$, at most $n$ components. Therefore, it takes value in $\Sym^{\le n}(\mathcal{L}_{0}^{d|\delta}(\X))(S)$.
\end{proof}

\subsection{Relations satisfied by the geometric power cooperations}\label{sec:consistentcoop}
Let $\N(-) \subseteq \M(-)$ be a cover closed. In this section we show that the geometric power cooperations satisfy the dual relations to the relations satisfied by the classical power operations on an $H_{\infty}$-ring spectrum. After taking functions, we will be able to conclude that the geometric power operations satisfy the same relations as classical power operations.

To describe these relations we will make use of a number of canonical maps between symmetric powers of stacks, which are analogous to the maps between symmetric powers of spectra introduced in \cite[Section~I.2]{bmms}. Let $j,k \ge 0$ and let
\begin{equation}\label{eq:symalpha}
    \xymatrix{\Sym^j(\X) \times \Sym^k(\X) \ar[rr]^-{\alpha_{j,k}} \ar@{=}[d] & & \Sym^{j+k}(\X) \ar@{=}[d] \\
    \X^{\times j}\sq\Sigma_j \times \X^{\times k}\sq\Sigma_k \ar[r]_-{\sim} &  \X^{\times j+k}\sq (\Sigma_j \times \Sigma_k) \ar[r] & \X^{\times j+k}\sq \Sigma_{j+k}
    }
\end{equation}
be the map of stacks that is induced by the inclusion $\Sigma_j \times \Sigma_k \subseteq \Sigma_{j+k}$.

The inclusion $\Sigma_j \wr \Sigma_k \subseteq \Sigma_{jk}$ induces a map of stacks 
\begin{equation}\label{eq:symbeta}
    \xymatrix{\Sym^j(\Sym^k(\X)) \ar[rr]^-{\beta_{j,k} } \ar@{=}[d] & & \Sym^{jk}(\X) \ar@{=}[d] \\ 
    (\X^{\times k}\sq \Sigma_k)^{\times j}\sq \Sigma_{j} \ar[r]^-{\sim} & \X^{\times jk}\sq(\Sigma_j \wr \Sigma_k) \ar[r] & \X^{\times jk}\sq\Sigma_{jk}.}
\end{equation}

The diagonal inclusion $\Sigma_k \to \Sigma_k \times \Sigma_k$ gives a map 
\begin{equation}\label{eq:symdelta}
    \resizebox{\textwidth}{!}{
    \xymatrix{\Sym^k(\X \times \Y) \ar[rr]^-{\delta_k} \ar@<-5ex>@{=}[d] & & \Sym^k(\X) \times \Sym^k(\Y)\ar@<5ex>@{=}[d] \\
    (\X \times \Y)^{\times k}\sq \Sigma_k \simeq (\X^{\times k} \times \Y^{\times k})\sq \Sigma_k \ar[rr] & & (\X^{\times k} \times \Y^{\times k})\sq (\Sigma_k \times \Sigma_k) \simeq (\X^{\times k}\sq \Sigma_k) \times (\Y^{\times k}\sq \Sigma_k).
    }
    }
\end{equation}
To ease notation, we will occasionally omit the subscripts on the maps $\alpha$, $\beta$, or $\delta$.

The projection maps $\X\times\Y \to \X$ and $\X\times\Y \to \Y$ induce a map
\begin{equation}\label{eq:can}
    \can\colon \N(\X \times \Y) \to \N(\X) \times \N(\Y).
\end{equation}

Finally, there is a fold map 
\begin{equation}\label{eq:nabla}
\nabla\colon \N(\X) \times \N(\X) \to \N(\X)
\end{equation}
which sends a pair of $S$-points $S \leftarrow T \to \X$ and $S \leftarrow T' \to \X$ to $S \leftarrow (T \coprod T') \to \X$. This is the symmetric monoidal structure on $\N(\X)$

We will repeatedly use the next result. 

\begin{lem} \label{lem:basic}
Suppose $\N(-) \subseteq \M(-)$ is cover closed and consider a 2-pullback of stacks
\[
\xymatrix{
\Y' \ar[r]^-{g} \ar[d]_{\pi_1} & \X' \ar[d]^{\pi_0} \\
\cY \ar[r]_-{f} & \cX,
}
\]
where $\pi_0$ a finite cover (see \cref{defn:finitecover}). Then the following diagram 2-commutes:
\[
\xymatrix{\N(\Y) \ar[r]^-{f_*} \ar[d]_{\pi_1^!} & \N(\X) \ar[d]^{\pi_0^!} \\
\N(\Y') \ar[r]_{g_*} & \N(\X').}
\]
\end{lem}
\begin{proof}
We may test this on $S$-points. Consider the 2-commuting diagram, 
\[
\xymatrix{& \widetilde{T} \ar[r] \ar[d] & \Y' \ar[d]_{\pi_1} \ar[r]^{g} & \X' \ar[d]^{\pi_0} \\ S  & T \ar[l] \ar[r] & \Y \ar[r]^{f} & \X,}
\]
in which the two squares are 2-pullbacks. It follows that the outer rectangle is a 2-pullback as well. This provides a natural isomorphism between $\pi_{0}^{!} f_*(S)$ and $g_* \pi_{1}^{!}(S)$.
\end{proof}

Now we prove a sequence of lemmas describing how the geometric power cooperations interact with the maps $\alpha_{j,k}$, $\beta_{j,k}$, and $\delta_k$.

\begin{lem}\label{lem:cooprelation1}
Let $\X$ be a stack and let $j,k\ge 0$. The following diagram 2-commutes:
\[
\xymatrix{
\N(\Sym^j(\X) \times \Sym^k(\X)) \ar[d]_{\N(\alpha_{j,k})} \ar[rr]^-{\can} &  & \N(\Sym^j(\X)) \times \N(\Sym^k(\X)) \ar[d]^{\coP_j \times \coP_k} \\
\N(\Sym^{j+k}(\X)) \ar[r]_-{\coP_{j+k}} & \N(\X) & \N(\X) \times \N(X). \ar[l]^-{\nabla}
}
\]
\end{lem}
\begin{proof}
There is a 2-pullback
\[
\xymatrix{
(\X^{\times j+k}\times \underline{j+k})\sq (\Sigma_j\times\Sigma_k) \ar[r]^-{\widetilde{\alpha}} \ar[d]_{\widetilde{\pi}} & (\X^{\times j+k}\times \underline{j+k})\sq\Sigma_{j+k} \ar[d]^{\pi} \\
\Sym^j(\X) \times \Sym^k(\X) \ar[r]_-{\alpha} & \Sym^{j+k}(\X).
}
\]
The diagram of the lemma expands as follows, where \cref{lem:basic} establishes 2-commutativity of the top left square and the bottom left triangle commutes by inspection:
\begin{equation}\label{eq:proofcooprelation1}
\resizebox{\textwidth}{!}{
\xymatrix{
\N(\Sym^{j+k}(\X)) \ar[d]_{\pi^!} & \N(\Sym^j(\X) \times \Sym^k(\X)) \ar[l]_-{\N(\alpha)} \ar[d]^{\widetilde{\pi}^!} \ar[r]^{\can} & \N(\Sym^j(\X)) \times \N(\Sym^k(\X)) \ar[d]^{\pi^! \times \pi^!} \\
\N((\X^{\times j+k}\times\underline{j+k})\sq\Sigma_{j+k}) \ar[dr]_{\ev_*} & \N((\X^{\times j+k}\times \underline{j+k})\sq (\Sigma_j\times\Sigma_k)) \ar[l]_{\N(\widetilde{\alpha})} \ar[d]^{\ev_*}  & \N((\X^{\times j}\times\underline{j})\sq\Sigma_{j}) \times \N((\X^{\times k}\times\underline{k})\sq\Sigma_{k}) \ar[d]^{\ev_*\times\ev_*} \\
& \N(\X) & \N(\X) \times \N(\X). \ar[l]^-{\nabla}
}
}
\end{equation}
In order to show that the right rectangle 2-commutes as well, we will introduce some auxiliary constructions. The covering $\widetilde{\pi}$ decomposes into a disjoint union
\begin{equation}\label{eq:coverdecompositionrelation1}
\resizebox{\textwidth}{!}{
\xymatrix{
(\X^{\times j}\times\X^{\times k}\times\underline{j})\sq\Sigma_j\times\Sigma_k \coprod (\X^{\times j}\times\X^{\times k}\times\underline{k})\sq\Sigma_j\times\Sigma_k \ar[r]^-{i_1\coprod i_2}_-{\sim} \ar[d]_{\pi_1 \coprod \pi_2} & (\X^{\times j+k}\times \underline{j+k})\sq (\Sigma_j\times\Sigma_k) \ar[d]^{\widetilde{\pi}} \\
\Sym^{j}(\X) \times \Sym^k(\X) \ar[r]_-{=} & \Sym^{j}(\X) \times \Sym^k(\X), 
}
}
\end{equation}
whose components fit into two 2-pullbacks
\begin{equation}\label{eq:two2pullbacksrelation1}
\resizebox{\textwidth}{!}{
\xymatrix{
(\X^{\times j}\times\X^{\times k}\times\underline{j})\sq\Sigma_j\times\Sigma_k \ar[r]^-{p_1} \ar[d]_{\pi_1} & (\X^{\times j}\times\underline{j})\sq\Sigma_j \ar[d]^{\pi} & (\X^{\times j}\times\X^{\times k}\times\underline{k})\sq\Sigma_j\times\Sigma_k \ar[r]^-{p_2} \ar[d]_{\pi_2} & (\X^{\times k}\times\underline{k})\sq\Sigma_k \ar[d]^{\pi} \\
\Sym^{j}(\X) \times \Sym^k(\X) \ar[r] & \Sym^{j}(\X) & \Sym^{j}(\X) \times \Sym^{k}(\X) \ar[r] & \Sym^{k}(\X).
}
}
\end{equation}
The map $\pi_1$ is given by
\[
(\X^{\times j}\times\X^{\times k}\times\underline{j})\sq\Sigma_j\times\Sigma_k \simeq (\X^{\times j} \times \underline{j}) \sq \Sigma_j \times\X^{\times k} \sq \Sigma_k \to \Sym^{j}(\X) \times \Sym^k(\X).
\]
where the last map is induced by the $j$-fold cover, and similarly for $\pi_2$. The unlabelled bottom horizontal maps in \eqref{eq:two2pullbacksrelation1} are the canonical projections onto the corresponding factors. Write 
\[
\mathcal{Z}_1 = (\X^{\times j}\times\X^{\times k}\times\underline{j})\sq\Sigma_j\times\Sigma_k \qquad \text{and} \qquad \mathcal{Z}_2 = (\X^{\times j}\times\X^{\times k}\times\underline{k})\sq\Sigma_j\times\Sigma_k.
\]
Moreover, define $\widetilde{\nabla}$ to be the composite
\[
\widetilde{\nabla} = \nabla \circ ((i_1)_*\times(i_2)_*)\colon \N(\mathcal{Z}_1)\times\N(\mathcal{Z}_2)\to \N((\X^{\times j+k}\times \underline{j+k})\sq (\Sigma_j\times\Sigma_k)).
\]
With this preparation, we can thus expand the right rectangle of \eqref{eq:proofcooprelation1} into a larger diagram:
\[
\resizebox{\textwidth}{!}{
\xymatrix{
\N(\Sym^j(\X) \times \Sym^k(\X)) \ar[d]_{\widetilde{\pi}^!} \ar[rr]^{\can} \ar[rd]^{\pi_1^!\times\pi_2^!} & & \N(\Sym^j(\X)) \times \N(\Sym^k(\X)) \ar[dd]^{\pi^! \times \pi^!} \\
\N((\X^{\times j+k}\times \underline{j+k})\sq (\Sigma_j\times\Sigma_k)) \ar[d]_{\ev_*} &
\N(\mathcal{Z}_1) \times \N(\mathcal{Z}_2)
 \ar[l]^-{\widetilde{\nabla}} \ar[d]_{\ev_*\times\ev_*} \ar[rd]^-{(p_1)_*\times(p_2)_*} \\
\N(\X) & \N(\X) \times \N(\X) \ar[l]^-{\nabla} & \N((\X^{\times j}\times\underline{j})\sq\Sigma_{j}) \times \N((\X^{\times k}\times\underline{k})\sq\Sigma_{k}).  \ar[l]^-{\ev_*\times\ev_*}
}
}
\]
The 2-commutativity of the top right quadrilateral follows from the two 2-pullbacks of \eqref{eq:two2pullbacksrelation1}, while that of the top left triangle is a consequence of the decomposition in \eqref{eq:coverdecompositionrelation1}. Finally, the bottom left square and bottom right triangle 2-commute by direct inspection, thereby finishing the proof.
\end{proof}

\begin{lem}\label{lem:cooprelation2}
For any stack $\X$ and any $j,k\ge 0$, the following diagram 2-commutes:
\[
\xymatrix{
\N(\Sym^j(\Sym^k(\X))) \ar[r]^-{\coP_j} \ar[d]_{\N(\beta_{j,k})} & \N(\Sym^k(\X)) \ar[d]^{\coP_k} \\
\N(\Sym^{jk}\X) \ar[r]_-{\coP_{jk}} & \N(\X).
}
\]
\end{lem}
\begin{proof}
The proof proceeds along the same lines as the previous one: we will decompose the diagram in the statement according to the definition of the geometric power cooperation and then verify that each of the subdiagrams 2-commute. The expanded diagram is:
\begin{equation}\label{eq:proofcooprelation2}
\xymatrix{
\N(\Sym^j(\Sym^k(\X))) \ar[r]^-{\pi^!} \ar[dd]_{\N(\beta)} & \N((\Sym^k(\X)^{\times j}\times \underline{j})\sq\Sigma_j) \ar[d]^{\pi^!} \ar[r]^-{\ev_*} & \N(\Sym^k(\X)) \ar[d]^{\pi^!} \\
& \N((((\X^{\times k} \times \underline{k})\sq\Sigma_k)^{\times j} \times \underline{j})\sq \Sigma_j) \ar@{-->}[d]^{\N(\widetilde{\beta})} \ar[r]^-{\ev_*} & \N((\X^{\times k}\times \underline{k})\sq \Sigma_k) \ar[d]^{\ev_*} \\
\N(\Sym^{jk}(\X)) \ar[r]_-{\pi^!} & \N((\X^{\times jk}\times \underline{jk})\sq \Sigma_{jk}) \ar[r]_-{ev_*} & \N(\X),
}
\end{equation}
with the map $\widetilde{\beta}$ yet to be constructed. 

To this end, view $\underline{jk}$ as $\underline{k} \times \underline{j}$, let $\underline{jk} \to \underline{j}$ be the resulting projection, and let $\Sigma_k \wr \Sigma_j \to \Sigma_{jk}$ be the canonical inclusion. The wreath product $\Sigma_k \wr \Sigma_j$ acts on $\underline{j}$ through the projection $\Sigma_k \wr \Sigma_j \to \Sigma_j$ and the usual action of $\Sigma_j$ on $\underline{j}$. These maps and actions assemble in a diagram 
\[
\xymatrix{
((\underline{k}\sq \Sigma_k)^{\times j} \times \underline{j})\sq\Sigma_j \ar[r] \ar[d] & \underline{jk}\sq(\Sigma_k \wr \Sigma_j) \ar[r] \ar[d] & \underline{jk}\sq\Sigma_{jk} \ar[dd] \\
((\ast\sq\Sigma_k)^{\times j} \times \underline{j})\sq\Sigma_j \ar[r] \ar[d] & \underline{j}\sq(\Sigma_k \wr \Sigma_j) \ar[d] \\
(\ast\sq\Sigma_k)^{\times j}\sq\Sigma_j \ar[r] & \ast\sq(\Sigma_k \wr \Sigma_j) \ar[r] & \ast\sq\Sigma_{jk}
}
\]
in which every rectangle is a 2-pullback. By using the 2-for-3 property of 2-pullbacks, this diagram induces a 2-pullback diagram
\[
\xymatrix{
(((\X^{\times k} \times \underline{k})\sq\Sigma_k)^j \times \underline{j})\sq \Sigma_j \ar[r]^-{\widetilde{\beta}} \ar[d]_{\pi} & (\X^{\times jk}\times \underline{jk})\sq \Sigma_{jk} \ar[dd]^{\pi}\\
(\Sym^k(\X)^{\times j}\times\underline{j})\sq\Sigma_j \ar[d]_{\pi}  \\
\Sym^j(\Sym^k(\X)) \ar[r]_-{\beta} & \Sym^{jk}(\X).
}
\]
In particular, we have constructed the map $\widetilde{\beta}$ and shown that the left rectangle in \eqref{eq:proofcooprelation2} 2-commutes. Similarly, one can construct a 2-pullback
\[
\xymatrix{
(((\X^{\times k} \times \underline{k})\sq\Sigma_k)^{\times j} \times \underline{j})\sq \Sigma_j \ar[r]^-{\ev} \ar[d]_{\pi} & (\X^{\times k}\times \underline{k})\sq \Sigma_k \ar[d]^{\pi} \\
(\Sym^k(\X)^{\times j}\times \underline{j})\sq\Sigma_j \ar[r]_-{\ev} & \Sym^k(\X)
}
\]
that witnesses the 2-commutativity of the right upper square in \eqref{eq:proofcooprelation2}, while the proof that the right bottom square 2-commutes is left to the reader.
\end{proof}

\begin{lem}\label{lem:cooprelation3}
For stacks $\X,\Y$ and any $j,k\ge 0$, the following diagram 2-commutes:
\[
\xymatrix{
\N(\Sym^k(\X \times \Y)) \ar[rr]^-{\coP_k} \ar[d]_{\N(\delta_k)} & & \N(\X \times \Y) \ar[d] \\
\N(\Sym^k(\X) \times \Sym^k(\Y)) \ar[r]_-{\can} & \N(\Sym^k(\X)) \times \N(\Sym^k(\Y)) \ar[r]_-{\coP_k \times \coP_k} & \N(\X) \times \N(\Y).
}
\]
\end{lem}
\begin{proof}
The proof is similar to the arguments used in the previous two lemmas; we omit the details. The key observation is that there exists two 2-pullbacks
\[
\resizebox{\textwidth}{!}{
\xymatrix{
(\X^{\times k}\times\Y^{\times k}\times\underline{k})\sq\Sigma_k \ar[r]^-{\widetilde{p}} \ar[d]_{\pi} & (\X^{\times k}\times\underline{k})\sq\Sigma_k \ar[d]^{\pi} & (\X^{\times k}\times\Y^{\times k}\times\underline{k})\sq\Sigma_k \ar[r]^-{\widetilde{p}} \ar[d]_{\pi} & (\Y^{\times k}\times\underline{k})\sq\Sigma_k \ar[d]^{\pi} \\
\Sym^k(\X\times\Y) \ar[r]_p & \Sym^k(\X) & \Sym^k(\X\times\Y) \ar[r]_p & \Sym^k(\Y).
}
}
\]
This allows us to check that the following expanded diagram 2-commutes:
\[
\resizebox{\textwidth}{!}{
\xymatrix{
& \N(\Sym^k(\X\times\Y)) \ar[r]^-{\pi^!} \ar[dl]_{\N(\delta)} \ar[d]_{\N(p)\times \N(p)} & \N((\X^{\times}\times\Y^{\times k}\times\underline{k})\sq\Sigma_k) \ar[d]_{\N(\widetilde{p})\times\N(\widetilde{p})} \ar[r]^-{\ev_*} & \N(\X\times\Y) \ar[d]^{\can} \\
\N(\Sym^k(\X) \times \Sym^k(\Y)) \ar[r]_-{\can} & \N(\Sym^k(\X)) \times \N(\Sym^k(\Y)) \ar[r]_-{\pi^!\times\pi^!} & \N((\X^{\times k}\times\underline{k})\sq\Sigma_k) \times \N((\Y^{\times k}\times\underline{k})\sq\Sigma_k) \ar[r]_-{\ev_*\times\ev_*} & \N(\X)\times\N(\Y),
}
}
\]
using that the 2-commutativity of the subdiagrams can be established for each of the two factors individually.
\end{proof}

\begin{rmk}
Each of the last three lemmas also admits a proof using $S$-points. To prove these results using $S$-points, one needs an $S$-point description of each of the maps $\N(\alpha_{j,k})$, $\N(\beta_{j,k})$, and $\N(\delta_k)$. For instance, in the case of $\N(\alpha_{j,k})$, a pair
\[
S \leftarrow T \to \Sym^j(\X) \qquad S \leftarrow T' \to \Sym^k(\X)
\]
corresponding to a $j$-fold cover $\widetilde{T}$ of $T$ equipped with a $\Sigma_j$-equivariant map $\widetilde{T}\to \X^j$ and a $k$-fold cover $\widetilde{T'}$ of $T'$ equipped with a $\Sigma_j$-equivariant map $\widetilde{T'}\to \X^k$ is sent to the $j+k$-fold cover $\widetilde{T} \coprod \widetilde{T'}$ and the $\Sigma_{j+k}$-equivariant map $\widetilde{T} \coprod \widetilde{T'} \to \X^{j+k}$.
\end{rmk}

\subsection{Geometric power operations}

The geometric power operation is the operation given by applying multiplicative functions to the geometric power cooperation. In this subsection, we define the geometric power operation and describe its first properties.

\begin{defn} \label{defn:funs}A \emph{function} on a stack is a morphism to the sheaf~$C^\infty(-)$ of $\C$-valued smooth functions on the site of supermanifolds. A function on a symmetric monoidal stack is \emph{multiplicative} if this morphism is symmetric monoidal, where we take multiplication of functions as a symmetric monoidal structure on $C^\infty(-)$. 
\end{defn} 

Applying this to $\N(\X)$, for $\N(-) \subset \M(-)$ cover closed, a function $f\in C^\infty(\N(\X))$ is multiplicative if 
\[
f(T\to \X)=f(T_1\to \X)\cdot f(T_2\to \X)\in C^\infty(S),\qquad 
\]
for any $S$, where $T$ is related to $T_1$ and $T_2$ via an isomorphism in $\N(\X)(S)$ of the form 
\[
\begin{tikzpicture}
\node (A) at (0,-.75) {$S$};
\node (B) at (3,0) {$T_1 \coprod T_2$};
\node (E) at (3,-1.5) {$T$};
\node (G) at (6,-.75) {$\X$.};
\node (H) at (4,-.75) {$\twocommute$};
\draw[->] (B) to (A);
\draw[->] (E) to (A);
\draw[->] (B) to node [left] {$\cong$} (E);
\draw[->] (B) to  (G);
\draw[->] (E) to (G);
\path (0,-.75) coordinate (basepoint);
\end{tikzpicture}\nonumber
\]
We denote the ring of these multiplicative functions by $C^\infty_\otimes (\N(\X))\subset C^\infty(\N(\X))$. 

Since the pushforward along a smooth map $\X \to \Y$ gives a symmetric monoidal functor $\N(\X)\to \N(\Y)$, the induced map on functions preserves the subset of multiplicative functions. The geometric power cooperation $\coP_n$ is symmetric monoidal by \cref{lem:copsym}. Thus we may make the following definition:

\begin{defn}\label{defn:powerop} The \emph{$n^{\rm th}$ geometric power operation} $\P_n$ is the map induced by $\coP_n$ on multiplicative functions, 
\[
\P_n:=\coP_n^*\colon C^\infty_\otimes (\N(\X))\to C^\infty_\otimes (\N(\X^{\times n}\sq \Sigma_n)).
\]
\end{defn}

Next we define the concordance relation on functions. 

\begin{defn} \label{concordance2} Two functions $f_0,f_1\in C^\infty(\N(\X))$ are \emph{concordant} if there exists a function $f\in C^\infty(\N(\X\times \R))$ such that the restrictions along $i_0,i_1\colon \X\hookrightarrow \X\times \R$ to $0,1\in \R$ satisfy
\[
i_0^*f=f_0\qquad i_1^*f = f_1. 
\]
\end{defn}
Concordance defines an equivalence relation (e.g., see~\cite[\S1]{HKST}), and equivalence classes are called \emph{concordance classes} of functions. 

\begin{lem} \label{concordance} 
The $n^{\rm th}$ geometric power operation descends to a map on concordance classes. 
\end{lem}

\bp
Naturality of the construction of $\P_n$ gives the commuting square on the left
\beq
\begin{tikzpicture}
\node (A) at (0,0) {$C^\infty_\otimes (\N(\X\times \R))$};
\node (B) at (5,0) {$C^\infty_\otimes (\N((\X\times \R)^{\times n}\sq \Sigma_n))$};
\node (C) at (0,-1.5) {$C^\infty_\otimes (\N(\X))$};
\node (D) at (5,-1.5) {$C^\infty_\otimes (\N(\X^{\times n}\sq \Sigma_n))$};
\node (E) at (10,-.75) {$C^\infty_\otimes (\N(\X^{\times n}\sq \Sigma_n\times \R))$,};
\draw[->] (A) to node [above] {$\P_n$} (B);
\draw[->] (C) to node [below]  {$\P_n$} (D);
\draw[->] (B) to node [right] {$j_0^*,j_1^*$} (D);
\draw[->] (A) to node [left] {$i_0^*,i_1^*$} (C);
\draw[->] (B) to node [above] {$\Delta^*$} (E);
\draw[->] (E) to node [below] {$i_0^*,i_1^*$} (D);
\end{tikzpicture}\nonumber
\eeq
where $j_0,j_1\colon \X^{\times n}\sq \Sigma_n \to (\X\times \R)^{\times n}\sq \Sigma_n$ include at $(0,0,\dots, 0)\in \R^n$ and $(1,1,\dots,1)\in \R^n$, respectively. The triangle on the right commutes because $j_0$ and $j_1$ factor through the map $\Delta\colon (\X^{\times n}\sq \Sigma_n) \times \R\to (\X\times \R)^{\times n}\sq \Sigma_n$ induced by the diagonal $\R\to \R^n$. In particular, this diagram sends a concordance $f\in C^\infty_\otimes (\N(\X\times \R))$ between functions $f_0,f_1\in C^\infty_\otimes (\N(\X))$ to a concordance $\Delta^*\P_n(f)\in C^\infty_\otimes (\N(\X^{\times n}\sq \Sigma_n\times \R))$ between functions $\P_n(f_0),\P_n(f_1)\in C^\infty_\otimes (\N(\X^{\times n}\sq \Sigma_n))$. Therefore the geometric power operation preserves concordance, so is well-defined on concordance classes of functions. 
\ep

Our next goal is to show that the geometric power operation restricts to functions on moduli stacks of super tori. We start with the following lemma which is a consequence of the definition of a free symmetric monoidal stack, see \Cref{def:symstack}.

\begin{lem} \label{lem:multrestr}
For any stack $\X$, restriction of functions induces natural isomorphisms
\[
C^\infty_\otimes(\Sym(\mathcal{L}^{d|\delta}(\X))) \cong C^\infty(\mathcal{L}^{d|\delta}(\X)) \qquad \text{and} \qquad C^\infty_\otimes(\Sym(\mathcal{L}_0^{d|\delta}(\X))) \cong C^\infty(\mathcal{L}_0^{d|\delta}(\X)).
\]
\end{lem}

Applying $C^\infty_\otimes(-)$ to the restriction of the geometric power cooperation of \eqref{eq:toripower}, the $n^{\rm th}$ geometric power operation gives a map
\beq
C^\infty_\otimes(\Sym(\mathcal{L}^{d|\delta}(\X)))\to C^\infty_\otimes(\Sym(\mathcal{L}^{d|\delta}(\X^{\times n}\sq \Sigma_n))). \label{eq:toripowerfn}
\eeq
By \Cref{lem:multrestr}, the multiplicative functions on the source and on the target are determined by their restriction to $\mathcal{L}^{d|\delta}(\X)$. From \Cref{cor:cooprestriction}, we thus obtain:

\begin{prop}\label{prop:oprestriction}
For any stack $\cX$, the $n^{\rm th}$ geometric power operation restricts to functions on moduli spaces of super tori, i.e., the following diagram commutes:
\[
\xymatrix{C^\infty_\otimes (\M(\X))\ar[r]^-{\P_n} \ar[d] & C^\infty_\otimes (\M(\X^{\times n}\sq \Sigma_n)) \ar[d] \\
C^\infty(\mathcal{L}^{d|\delta}(\X))\ar@{-->}[r]_-{\P_n} & C^\infty(\mathcal{L}^{d|\delta}(\X^{\times n}\sq \Sigma_n)).
}
\]
\end{prop}

We emphasize that $\P_n$ usually does not respect addition of functions, so it is not a map of algebras. 

Finally, we will use the results of \cref{sec:consistentcoop} to deduce relations among geometric power operations analogous to the relations afforded by homotopical power operations. Applying $\N(-)$ to the functions of stacks \eqref{eq:symalpha}, \eqref{eq:symbeta}, and \eqref{eq:symdelta} induces maps on functions which we will denote by $\alpha^*$, $\beta^*$, and $\delta^*$, respectively. With this notation, we obtain the following relations for the geometric power operations.

\begin{prop}\label{prop:opconsistenttheory}
Let $j,k \ge 0$ be integers, $\X$ and $\Y$ stacks, and let $\N(-) \subseteq \M(-)$ be cover closed. For $x \in C^\infty_\otimes (\N(\X))$ and $y \in C^\infty_\otimes (\N(\Y))$, we have: 
    \begin{enumerate}
        \item $\alpha^*\P_{j+k}(x) = \P_j(x)\P_k(x) \in C^\infty_\otimes (\N(\Sym^j(\X) \times \Sym^k(\X)))$.
        \item $\beta^*\P_{jk}(x) = \P_j(\P_k(x)) \in C^\infty_\otimes (\N(\Sym^j(\Sym^k(\X))))$.
        \item $\delta^*(\P_k(x)\P_k(y)) = \P_k(xy) \in C^\infty_\otimes(\N(\Sym^k(\X \times \Y)))$.
    \end{enumerate}
\end{prop}
\begin{proof}
The relations follow immediately from \cref{lem:cooprelation1}, \cref{lem:cooprelation2}, and \cref{lem:cooprelation3} by passing to multiplicative functions. 
\end{proof}

In the situation of \cref{prop:opconsistenttheory}, we say that the collection $\{\P_k\}_{k \ge 0}$ is a \emph{consistent set of geometric power operations}. This is in analogy to the terminology used in \cite[Chapter VIII,\S1]{bmms}, where it is shown that the data of a consistent set of power operations for a ring spectrum $E$ is equivalent to the data of an $H_{\infty}^0$-structure on $E$, see Proposition 1.2 in \cite[Chapter VIII]{bmms}.

\subsection{Extension to bordism categories}\label{sec:bordext}

In this section, we sketch the construction of geometric power cooperations on bordism categories as in~\eqref{eq:powercoopFT}, following the discussion in~\cref{sec:introFT}.

Ignoring some important technical details, we recall that Stolz and Teichner define an $S$-family of geometric bordisms over~$\X$ to be a triple conisting of: (1) a proper $S$-family $B\to S$ of $d|\delta$-dimensional supermanifolds with $\M$-structure, (2) a map of stacks~$\phi\colon  B \to \X$, and (3) incoming and outgoing boundary data specified by $(d-1)|\delta$-dimensional super manifolds~$Y_{\rm in}\sqcup Y_{\rm out}= \partial B\hookrightarrow B$ over $S$. The $\M$-structure determines the super dimension $d|\delta$ of the bordisms~$B$. The \emph{geometric bordism category} is a category internal to symmetric monoidal stacks, $\Bord^{d|\delta}(\X)$, whose morphism stack is given by the above triples and whose object stack is given by $S$-families of supermanifolds~$Y\to S$ equipped with a collar; see~\cite[Defintions~2.13, 2.21, 2.46, 4.4]{ST11} in the case~$\X=\pt$. 

The construction of the geometric power cooperation 
\[
\xymatrix{\coP_n\colon \Bord^{d|\delta}(\X^{\times n}\sq \Sigma_n) \ar@{.>}[r] & \Bord^{d|\delta}(\X)}
\]
should proceed as follows. Given an $S$-family of bordisms over $\X^{\times n}\sq \Sigma_n$, pulling back the universal $n$-sheeted covering $(\X^{\times n}\times \underline{n})\sq \Sigma_n\to \X^{\times n}\sq \Sigma_n$ and post-composing with the evaluation map to $\X$ yields a bordism $\widetilde{\phi}\colon \widetilde{B}\to \X$ with incoming and outgoing boundary $\widetilde{Y}_{\rm in}\sqcup \widetilde{Y}_{\rm out} \hookrightarrow \widetilde{B}$, summarized by the 2-commuting diagram:
\beq
\begin{tikzpicture}[baseline=(basepoint)]
\node (A) at (0,0) {$\widetilde{B}$};
\node (B) at (3,0) {$(\X^{\times n} \times \underline{n})\sq \Sigma_n$};
\node (D) at (0,-1.5) {$B$};
\node (E) at (3,-1.5) {$\X^{\times n}\sq \Sigma_n$.};
\node (F) at (6,0) {$\X$};
\node (G) at (-2,-1.5) {$Y_{\rm in}\sqcup Y_{\rm out}$};
\node (H) at (-2,0) {$\widetilde{Y}_{\rm in}\sqcup \widetilde{Y}_{\rm out}$};
\draw (.25,-.75) -- (.25,-.5) -- (.5,-.5);
\draw (-1.5,-.75) -- (-1.5,-.5) -- (-1.25,-.5);
\draw[->] (A) to  (B);
\draw[->] (A) to (D);
\draw[->] (B) to (E);
\draw[->] (D) to node [below] {$\phi$} (E);
\draw[->] (B) to node [below] {$\ev$} (F);
\draw[->,right hook-latex] (G) to (D);
\draw[->] (H) to (G);
\draw[->,right hook-latex] (H) to (A);
\draw[->,bend left=20] (A) to node [above] {$\widetilde{\phi}$} (F);
\path (0,-.75) coordinate (basepoint);
\end{tikzpicture}\label{eq:poweropFTdiag}
\eeq 
The operation is well-defined on geometric bordisms: if $Y_{\rm in}\sqcup Y_{\rm out}$ comprise the boundary of~$B$, $\widetilde{Y}_{\rm in}\sqcup \widetilde{Y}_{\rm out}$ is the boundary of $\widetilde{B}$; and if $B$ is equipped with a model geometry, $\widetilde{B}$ has a uniquely determined model geometry (see~\Cref{rmk:covermodel}). Furthermore, the operation is symmetric monoidal: a cover of a disjoint union is canonically isomorphic to a disjoint union of covers. 

To extend the above sketch to the asserted power cooperation on geometric bordism categories $\xymatrix{\coP_n\colon \Bord^{d|\delta}(\X^{\times n}\sq \Sigma_n) \ar@{.>}[r] & \Bord^{d|\delta}(\X)}$, one encounters a problem: the structure maps (source, target, unit, and composition) in Stolz and Teichner's definition of geometric bordism category are required to satisfy certain strict conditions. For example, composition is defined on the strict fibered product of morphisms over objects~\cite[pg.~20]{ST11}, and a functor between bordism categories is required to be strictly compatible with source and target maps~\cite[Definition~2.18]{ST11}. Because the construction of~$\coP_n$ requires 2-pullbacks in stacks, composition in $\Bord^{d|\delta}(\X^{\times n}\sq \Sigma_n)$ and $\Bord^{d|\delta}(\X)$ cannot be made strictly compatible with~$\coP_n$. Geometrically, this is because the $n$-fold cover of a boundary is canonically isomorphic (but not equal) to the boundary of an $n$-fold cover. Hence, the power cooperation only determines a map between the 2-fibered products defining composition, not the strict ones as required in~\cite[Definition~2.18]{ST11}. For the same geometric reason,~$\coP_n$ cannot be made strictly compatible with source and target maps. However, we expect that a weakening of Stolz and Teichner's definition leads to a closely related geometric bordism category on which the power cooperation is defined. 

There is an unambiguous piece of this proposed extension of Stolz and Teichner's definition, namely the subcategory of closed bordisms, $\Bord^{d|\delta}_{\rm c}(\X)$. The subtleties described above disappear on this subcategory, e.g., composing along the empty bordism is the same as the disjoint union. The sketch of~$\coP_n$ in diagram~\eqref{eq:poweropFTdiag} restricted to $\Bord^{d|\delta}_{\rm c}(\X)$ is then equivalent to the definition of~$\coP_n$ given earlier in the section.

\section{Computing geometric power operations using an atlas} \label{sec:Section3}

In this section, we provide tools for calculating the geometric power operations. We will apply these in the subsequent sections to complexified equivariant $K$-theory and complexified equivariant elliptic cohomology. We specialize to a global quotient stack $\X=[X\sq G$]. The primary goal of this section is to describe an atlas for the stack $\mathcal{L}_{0}^{d|\delta}(X\sq G)$ and produce a map of atlases covering the geometric power cooperation
\[
\coP_n\colon \Sym(\mathcal{L}_{0}^{d|\delta}((X\sq G)^{\times n}\sq \Sigma_n))\to \Sym(\mathcal{L}_{0}^{d|\delta}(X\sq G)). 
\]
We will use this map of atlases to give explicit formulas for the geometric power operations in the cases of interest.

\subsection{Super tori over $X\sq G$} \label{sec:tori}

The goal of this section is to gain an understanding of the local structure of the stacks $\mathcal{L}^{d|\delta}(X\sq G)$ and $\mathcal{L}_{0}^{d|\delta}(X\sq G)$.

Let $S$ be a supermanifold. The 2-functor from the 2-category of Lie groupoids to the 2-category of stacks gives a morphism of groupoids
\beq
{\sf Grpd}((S\times\R^{d|\delta})\sq \uL,X\sq G) \to {\Stack}((S\times \R^{d|\delta})/\uL,[X\sq G]), \label{eq:grpdequiv}
\eeq
that sends the groupoid of functors between Lie groupoids to the groupoid of maps between their corresponding stacks, see \cref{ex:liestackcomparison}; here we have also used the equivalence of stacks $(S\times \R^{d|\delta})/\uL\simeq [(S\times \R^{d|\delta})\sq \uL]$ from \eqref{eq:representablequotient}.

\begin{lem} \label{lem:stackgrpd}
Locally in $S$, the functor~\eqref{eq:grpdequiv} is an equivalence of groupoids.
\end{lem}
\begin{proof}
The functor~\eqref{eq:grpdequiv} admits an explicit description. An object in the source groupoid is a pair of maps
\[
(\phi_{\rm ob}\colon S\times \R^{d|\delta}\to X, \phi_{\rm mor}\colon S\times \R^{d|\delta}\times \uL\to X \times G)
\] 
that fit together to define a functor between Lie groupoids. Similarly to \cref{ex:globalquotient}, the image under~\eqref{eq:grpdequiv} is the ``bundlization" of a functor between Lie groupoids, which in this case is a principal $G$-bundle over $(S\times \R^{d|\delta})/\uL$ with a $G$-equivariant map to~$X$:
\beq
(S\times \R^{d|\delta})/\uL \leftarrow (S\times \R^{d|\delta}\times G)/\uL \to X.\label{eq:Gbundconstr}
\eeq
The left arrow is the obvious projection, whereas the right arrow comes from the $\uL$-equivariant map
\[
\xymatrixcolsep{3.5pc}{
\xymatrix{S\times \R^{d|\delta}\times G \ar[r]^-{\phi_{\rm ob}\times \id_G} & X\times G \ar[r]^-{{\rm act}} & X}}
\]
for the trivial $\uL$-action on $X$ and the $\uL$-action on the trivial $G$-bundle from $S\times \R^{d|\delta}\times \uL\to X\times G\to G$ covering the $\uL$-action on $S\times \R^{d|\delta}$. 
See, for example,~\cite[Examples~16 and~17]{schommerpries_centralext} for details. 

We must show that the functor~\eqref{eq:grpdequiv} is locally fully faithful and essentially surjective. It follows from the definition of a map of bibundles (see \cite[Definition 20]{schommerpries_centralext}) that the map is fully faithful (for all $S$).

To see that it is essentially surjective we will need to work locally in $S$. Given an arbitrary principal $G$-bundle $P\to (S\times \R^{d|\delta})/\uL$, there exists an open cover $(S_i)$ of~$S$ such that $P|_{S_i}\simeq (S_i\times \R^{d|\delta}\times G)/\uL$, i.e., is of the form \eqref{eq:Gbundconstr}. Hence the asserted functor is essentially surjective locally in $S$. 
\end{proof}

Precomposition with the quotient map 
\[
(S\times\R^{d|\delta})\sq \uL \to ((S\times\R^{d|\delta})/\bE^d)\sq \uL \cong (S\times\R^{0|\delta})\sq \uL,
\] 
gives a functor
\[
{\sf Grpd}((S\times\R^{0|\delta})\sq \uL,X\sq G) \to {\sf Grpd}((S\times\R^{d|\delta})\sq \uL,X\sq G).
\]
By taking $\bE^d$-invariant maps (see \Cref{defn:edinvariantmaps} and \cref{rmk:edinvariantmaps}) in ~\eqref{eq:grpdequiv}, we have the following corollary of \cref{lem:stackgrpd}.

\begin{cor} \label{cor:stackgrpd}
Locally in $S$, the functor~\eqref{eq:grpdequiv} induces an equivalence of groupoids
\[
{\sf Grpd}((S\times\R^{0|\delta})\sq \uL,X\sq G)\stackrel{\sim}{\to} {\Stack}([(S\times \R^{0|\delta})\sq\uL],[X\sq G]).
\] 
\end{cor}

Below we will make use of the following equivalences of stacks (induced by equivalences of Lie groupoids)
\beq\label{eq:wrisos}
&&(X\sq G)^{\times n}\sq \Sigma_n\simeq X^{\times n}\sq G\wr \Sigma_n\qquad ((X\sq G)^{\times n}\times \underline{n})\sq \Sigma_n\simeq (X^{\times n}\times \underline{n})\sq G\wr\Sigma_n\label{eq:ofstacks}
\eeq
that identify a $G^{\times n}$-bundle over a $\Sigma_n$-bundle over $S$ with a $G\wr\Sigma_n$-bundle over $S$. Let $(x_1,\dots,x_n,i) \in X^n \times \underline{n}$ and write $(g_1,\dots,g_n,\sigma) \in G \wr \Sigma_n$ for $g_1, \ldots, g_n \in G$ and $\sigma\in \Sigma_n=\Aut(\underline{n})$. The right action of $G \wr \Sigma_n$ on $X^{\times n} \times \underline{n}$ is given by
\[
(x_1,\dots,x_n,i) \times (g_1,\dots,g_n,\sigma) \mapsto (x_{\sigma(1)}g_1,\ldots,x_{\sigma(n)}g_n,\sigma).
\]

\begin{lem} \label{lem:ev}
The evaluation map $(X^{\times n}\times \underline{n})\sq G\wr\Sigma_n\stackrel{\ev}{\to} X\sq G$
is the map of Lie groupoids
\beq
\begin{tikzpicture}
\node (A) at (0,0) {$X^{\times n}\times \underline{n}\times G \wr \Sigma_n$};
\node (B) at (5,0) {$X\times G$};
\node (D) at (0,-1.5) {$X^{\times n}\times \underline{n}$};
\node (E) at (5,-1.5) {$X$};
\draw[->] (A) to node[above] {$\ev_{\rm{mor}}$} (B);
\draw[->] (A) to node[left] {$s,t$} (D);
\draw[->] (B) to node[right] {$s,t$} (E);
\draw[->] (D) to node[below] {$\ev_{\rm{ob}}$} (E);
\end{tikzpicture}\nonumber
\eeq 
given by the maps on objects and morphisms
\[
\ev_{\rm{ob}}\colon(x_1,\dots,x_n,i)\mapsto x_i\qquad \ev_{\rm{mor}}\colon(x_1,\dots,x_n,i,g_1,\dots,g_n,\sigma)\mapsto (x_i,g_{\sigma^{-1}(i)})
\]
for $(x_1,\dots,x_n)\in X^{\times n}$, $i\in \underline{n}$, $(g_1,\dots,g_n)\in G^{\times n}$ and $\sigma\in \Sigma_n$. 
\end{lem}
\bp
Using \eqref{eq:wrisos}, there is an equivalence
\[
{\sf Grpd}_{\Sigma_n}(\Map(\underline{n},X\sq G)\times \underline{n},X\sq G) \simeq {\sf Grpd}((X^{\times n}\times \underline{n})\sq G\wr\Sigma_n,X\sq G)
\]
when the $\Sigma_n$-action on $X\sq G$ is trivial. The adjunction between internal hom and product gives an equivalence of groupoids
\[
{\sf Grpd}_{\Sigma_n}(\Map(\underline{n},X\sq G),\Map(\underline{n},X\sq G))\simeq {\sf Grpd}_{\Sigma_n}(\Map(\underline{n},X\sq G)\times \underline{n},X\sq G),
\]
where ${\sf Grpd}_{\Sigma_n}$ is the category of Lie groupoid with a (strict) $\Sigma_n$-action. Under this equivalence, the identity map on the left is sent to the evaluation map. 

To verify that the formula is correct, note that removing the symmetric group action gives the well-known formula for the evaluation map. This determines the $\Sigma_n$-equivariant evaluation map, and it is easy to check that this gives a well-defined map of Lie groupoids.
\ep

It will be useful to have an explicit local formula for $\pi^!$. Consider a map of Lie groupoids
\[
T \in {\sf Grpd}((S\times\R^{d|\delta})\sq \uL,X^{\times n} \sq G \wr \Sigma_n).
\]
We define $\pi^!T$ to be the map of groupoids induced by the pullback of Lie groupoids (the pullback of supermanifolds on the level of objects and morphisms): 
\beq\label{eq:liesquarepullbackformula}
\begin{tikzpicture}[baseline=(basepoint)];
\node (A) at (0,0) {$(S\times \R^{d|\delta}\times \underline{n})\sq \uL$};
\node (B) at (5,0) {$(X^{\times n}\times \underline{n})\sq G\wr\Sigma_n$};
\node (D) at (0,-1.5) {$(S\times \R^{d|\delta})\sq \uL$};
\node (E) at (5,-1.5) {$X^{\times n}\sq G\wr\Sigma_n$.};
\draw[->] (A) to node[above]{$\pi^!T$}  (B);
\draw[->] (A) to (D);
\draw[->] (B) to (E);
\draw[->] (D) to node[above]{$T$} (E);
\path (0,-.75) coordinate (basepoint);
\end{tikzpicture}
\eeq 
In the pullback, the $\uL$-action on~$\underline{n}$ is through the composition of group homomorphisms $S \times \uL\to G\wr\Sigma_n\to \Sigma_n$ coming from the map $T$. 

The pullback above decomposes according to the action of $\uL$.

\begin{lem} \label{lem:factor}
Let $\underline{n}=\coprod I_k$ be the decomposition of $\underline{n}$ into transitive $\uL$-sets and let $\uL_k \subset \uL$ the stabilizer of an element in $I_k$. Given choices of elements $i_k\in I_k$ for all $k$, there is an equivalence of Lie groupoids
\[
\coprod_k (S\times \R^{d|\delta})\sq \uL_{k} \stackrel{\sim}{\to}  (S\times \R^{d|\delta}\times \underline{n})\sq \uL.
\]
\end{lem}

\bp The desired equivalence is a composite of two equivalences,
\[
\coprod_k (S\times \R^{d|\delta})\sq \uL_{k}\stackrel{\sim}{\dashrightarrow} \coprod_k (S\times \R^{d|\delta}\times I_k)\sq \uL \stackrel{\sim}{\to} (S\times \R^{d|\delta}\times \underline{n})\sq \uL.
\]
The solid arrow on the right is an isomorphism of Lie groupoids (a diffeomorphism on both objects and morphisms). The dashed arrow depends on the choice of elements $i_k\in I_k$ for all~$k$, defining inclusions on objects
\[
S\times \R^{d|\delta}\cong S\times \R^{d|\delta}\times \{i_k\}\hookrightarrow S\times \R^{d|\delta}\times I_k,
\]
and similar inclusions on morphisms using that for each $k$, $\L_k\subset \L$ is a sublattice. Since the $\uL$-action is transitive on $I_k$, this dashed arrow is fully faithful and essentially surjective, verifying the lemma. 
\ep

As the notation suggests, $[\pi^!T]$ is a $2$-pullback, as we will verify in the next lemma:

\begin{lem}\label{lem:pull}
The map of stacks $[\pi^!T] \in {\Stack}([(S\times \R^{d|\delta}\times \underline{n})\sq \uL],[X^{\times n}\times \underline{n} \sq G \wr \Sigma_n])$ is a $2$-pullback of $\pi$ along $[T]$.
\end{lem}
\bp
After applying $[-]$ to \eqref{eq:liesquarepullbackformula}, both the top left and bottom left stacks are representable, so, after applying \cref{lem:factor}, we get a $2$-commuting diagram of stacks
\beq
\begin{tikzpicture}[baseline=(basepoint)];
\node (A) at (0,0) {$\coprod_k (S\times \R^{d|\delta}) / \uL_{k}$};
\node (B) at (5,0) {$[(X^{\times n}\times \underline{n})\sq G\wr\Sigma_n]$};
\node (D) at (0,-1.5) {$(S\times \R^{d|\delta}) / \uL$};
\node (E) at (5,-1.5) {$[X^{\times n}\sq G\wr\Sigma_n]$.};
\draw[->] (A) to node[above]{$[\pi^!T]$}  (B);
\draw[->] (A) to (D);
\draw[->] (B) to (E);
\draw[->] (D) to node[above]{$[T]$} (E);
\path (0,-.75) coordinate (basepoint);
\end{tikzpicture}
\eeq 
Since $\pi$ is a finite cover (\cref{defn:finitecover}), the $2$-pullback of $\pi$ along $[T]$ is representable by a finite cover $E$ of $(S\times \R^{d|\delta}) / \uL$. By the universal property of the $2$-pullback, there is a canonical map of $n$-fold covers 
\[
\coprod_k (S\times \R^{d|\delta}) / \uL_{k} \to E
\]
over $(S\times \R^{d|\delta}) / \uL$. It suffices to check that this is an isomorphism. But this can be seen by pulling back to the universal cover $S \times \R^{d|\delta} \to (S\times \R^{d|\delta}) / \uL$.
\ep

\subsection{An atlas for $\mathcal{L}^{d|\delta}_0(X\sq G)$}

Suppose we are given a geometric stack~$\mathcal{Z}$ and an atlas $U\to \mathcal{Z}$; see~\Cref{superstacks} for a quick review of geometric stacks. Then the set of functions on $\mathcal{Z}$ in the sense of \Cref{defn:funs} can be identified with the set of smooth functions $f\in C^\infty(U)$ such that $s^*f=t^*f$, where the atlas $U\to \mathcal{Z}$ determines a groupoid presentation with source and target maps $s$ and $t$, respectively; see~\eqref{diag:weak2} in \cref{superstacks}. With this in mind, we will construct an atlas for $\mathcal{L}^{d|\delta}_0(X\sq G)$ so that we can compute the effect of geometric power operations on functions.

Recall that $\uL = \Z^d$. We define a map
\beq
\mathcal{U}(X\sq G):=\Lat\times \coprod_{h \in \Hom(\uL, G)} \Map(\R^{0|\delta},X^{\im h})\to \mathcal{L}^{d|\delta}_0(X\sq G),\label{eq:generalatlas}
\eeq
where $\Lat$ is the manifold of based oriented lattices (a lattice with a choice of ordered basis that is positively oriented) in $\R^d\subset \R^{d|\delta}$, $h\colon \uL \to G$ is a group homomorphism, $X^{\im h}\subset X$ is the fixed point set of the image of ${h}$, and $\Map$ is the mapping supermanifold (see \Cref{superstacks}). To an $S$-point of $\mathcal{U}(X\sq G)$, we associate the $S$-family of super tori $(S\times \R^{d|\delta})/\uL$ with lattices coming from the corresponding $S$-point of~$\Lat$, and the map of stacks $(S\times \R^{d|\delta})/\uL \to [X\sq G]$ associated to the homomorphism of Lie groupoids, 
\[
(S\times \R^{d|\delta})\sq \uL\to (S\times \R^{0|\delta})\sq \uL \to X\sq G
\]
that on objects is the composite map $S\times \R^{d|\delta} \to S\times \R^{0|\delta}\to X^{\im h} \hookrightarrow X$, and on morphisms is determined by~$S\times \uL \to \uL \stackrel{h}{\to} G$ where the first map is the projection. The construction of the image of this assignment can thus be displayed as follows:
\[
\xymatrix{& S \times \R^{d|\delta} \times \uL \ar@<0.5ex>[d] \ar@<-0.5ex>[d] \ar[r] & S \times X \times \uL \ar[r] & X\times \uL \ar[r] & X \times G \ar@<0.5ex>[d] \ar@<-0.5ex>[d] \\
& S \times \R^{d|\delta} \ar[d] \ar[r] & S\times \R^{0|\delta} \ar[r] & X^{\im h} \ar[r] & X \ar[d] \\
S & T_{\Lambda}^{d|\delta} \ar[l] \ar[rrr] & & & X \sq G.}
\]

\begin{prop}\label{prop:atlas}
The map \eqref{eq:generalatlas} is an atlas for $\mathcal{L}^{d|\delta}_0(X\sq G)$.
\end{prop}
\begin{proof}
We will verify the conditions for an atlas given in \Cref{prop:eqatlas}.

First we show that $\mathcal{U}(X\sq G)\to \mathcal{L}^{d|\delta}_0(X\sq G)$ is an epimorphism. Fix an $S$-point of $\mathcal{L}^{d|\delta}_0(X\sq G)$. By \Cref{cor:stackgrpd}, there exists an open cover $(S_i)$ of $S$ (which determines a surjective submersion $\coprod_i S_i\to S$) over which the given $S$-point is determined by the data of an $S_i$-family of based oriented lattices $\Lambda_i\colon S_i\times \uL\to S_i\times \bE^{d|\delta}$ and functors between Lie groupoids $S_i\times \R^{0|\delta}\sq \uL \to X\sq G$ for each $i$. Fixing $S_i$, the functor is then given by a pair $(\phi_{\rm ob},\phi_{\rm mor})$, where $\phi_{\rm ob}\colon S_i\times \R^{0|\delta}\to X$ and $\phi_{\rm mor}\colon S_i\times\R^{0|\delta} \times \uL \to G$. Because $G$ is discrete, the map $\phi_{\rm mor}$ necessarily factors through $S_i\times \uL$, and we use the same notation $\phi_{\rm mor}\colon S_i\times \uL \to G$ for this map. For the data $\phi_{\rm mor}$ and $\phi_{\rm ob}$ to determine a functor, $\phi_{\rm mor}$ must be an $S_i$-family of homomorphisms and $(\phi_{\rm ob},\phi_{\rm mor})$ must be determined by an $S_i$-point of $\coprod_h \Map(\R^{0|\delta},X^{\im h})$. We have produced a surjective submersion $\coprod_i S_i\to S$ that factors through the proposed atlas, and hence the map $\mathcal{U}(X\sq G)\to \mathcal{L}^{d|\delta}_0(X\sq G)$ is an epimorphism. 

It remains to show that $\mathcal{U}(X\sq G)\times_{\mathcal{L}^{d|\delta}_0(X\sq G)}\mathcal{U}(X\sq G)$ is representable and that the projection maps are submersions. An $S$-point of the 2-pullback $\mathcal{U}(X\sq G)\times_{\mathcal{L}^{d|\delta}_0(X\sq G)} \mathcal{U}(X\sq G)$ consists of a pair of $S$-points of $\mathcal{U}(X\sq G)$ and an isomorphism between the corresponding objects over $S$ in the stack $\mathcal{L}^{d|\delta}_0(X\sq G)$. Hence, an $S$-point of the 2-pullback is a pair of $S$-points related by an isomorphism determined by $(f,g)$ in the diagram 
\beq
\begin{tikzpicture}[baseline=(basepoint)];
\node (A) at (-3,0) {$S\times \R^{d|\delta}$};
\node (B) at (3,0) {$[(S\times \R^{0|\delta})\sq\uL]$};
\node (AA) at (0,0) {$S\times \R^{d|\delta}/\uL$};
\node (BB) at (0,-1.5) {$S\times \R^{d|\delta}/\uL$};
\node (D) at (-3,-1.5) {$S\times \R^{d|\delta}$};
\node (E) at (3,-1.5) {$[(S\times \R^{0|\delta})\sq\uL]$};
\node (F) at (6,-.75) {$[X\sq G]$,};
\draw[->] (A) to  (AA);
\draw[->] (AA) to (B);
\draw[->] (BB) to (E);
\draw[->] (AA) to node [left] {$f$} (BB);
\draw[->] (AA) to node [right] {$\cong$} (BB);
\draw[->,dashed] (A) to node [left] 
{$\tilde{f}$} (D);
\draw[->] (B) to  (E);
\draw[->] (D) to (BB);
\draw[->] (B) to node [above] {$[\phi]$} (F);
\draw[->] (E) to node [below] {$[\phi']$} (F);
\node (D) at (4,-.75) {$g\twocommute$};
\path (0,-.75) coordinate (basepoint);
\end{tikzpicture}\label{isooncov}
\eeq 
where the objects are specified in terms of functors $\phi,\phi'$ as in the previous paragraph. We can lift the isometry $f\colon (S\times \R^{d|\delta})/\uL \to (S\times \R^{d|\delta})/\uL$ to a map between universal covers which determines $\tilde{f}\in \Iso(\R^{d|\delta})(S)$ that commutes with the $\uL$ action relative to a family of homomorphisms $\gamma\colon S\to \SL_d(\Z)$. The datum $g$ pulls back to an isomorphism between trivial $G$-bundles, i.e., a map $S\times \R^{d|\delta}\to S\times \R^{0|\delta}\to G$, which (because $G$ is discrete) is determined by a map $\tilde{g}\colon S\to G$. Finally, triples $(\tilde{f},\gamma,\tilde{g})$ and $(\tilde{f}',\gamma',\tilde{g}')$ determine the same isomorphism if they differ by the action of $S\times \uL$ for the subgroup
\[
S\times \uL \to S\times \Iso(\R^{d|\delta})\times G\times \SL_d(\Z),
\]
that includes $S \times \uL$ along $(\Lambda,\phi_{\rm mor})\colon S\times \uL\to S\times \Iso(\R^{d|\delta})\times G$. This identifies an $S$-point of the 2-pullback with the quotient 
\[
\mathcal{U}(X\sq G)\times_{\mathcal{L}^{d|\delta}_0(X\sq G)}\mathcal{U}(X\sq G)\simeq (\mathcal{U}(X\sq G)\times \Iso(\R^{d|\delta})\times G\times \SL_d(\Z))/\uL
\] 
of a representable stack by a free $\uL$-action, hence the 2-pullback is representable. The maps to $\mathcal{U}(X\sq G)$ are determined by the projection and action of the bundle of groups on the atlas, both of which are submersions. Hence,~\eqref{eq:generalatlas} determines an atlas. \ep

\subsection{Computing power operations using an atlas}\label{eq:computecoop}

Recall from \eqref{eq:wrisos} that we have natural equivalences $(X\sq G)^{\times n}\sq\Sigma_n \simeq X^{\times n}\sq G\wr \Sigma_n$. We will construct a map $\widetilde{\coP}_n$ between atlases covering the geometric power cooperation, as displayed in the following diagram:
\beq\label{eq:atlascoop}
\xymatrix{\mathcal{U}(X^{\times n}\sq G\wr \Sigma_n) \ar[r]^-{\widetilde{\coP}_n} \ar[d] & \coprod_{t \le n}\mathcal{U}(X\sq G)^{\times t} \ar[d] \\ 
\mathcal{L}^{d|\delta}_0(X^{\times n}\sq G\wr \Sigma_n) \ar[r]^-{\coP_n} & \Sym^{\le n}(\mathcal{L}^{d|\delta}_0(X\sq G)).}
\eeq
Here, the bottom horizontal map is the restriction of the $n^{\rm th}$ geometric power cooperation as exhibited in \Cref{lem:cooprestriction}, while the atlas $\mathcal{U}(X^{\times n}\sq G\wr \Sigma_n)$ of $\mathcal{L}^{d|\delta}_0(X^{\times n}\sq G\wr \Sigma_n)$ was constructed in \Cref{prop:atlas}. It is a consequence of \Cref{lem:symatlas} that the right vertical map provides an atlas for $\Sym^{\le n}(\mathcal{L}^{d|\delta}_0(X\sq G))$.

Our goal in this section is to construct the map $\widetilde{\coP}_n$ (see \Cref{def:atlascoop}) and to then show in \Cref{thm:mainthm} that it makes Diagram \eqref{eq:atlascoop} 2-commute. In later sections, we will use this explicit description of $\widetilde{\coP}_n$ to provide a formula for the geometric power operation in special cases. 

We recall some notation from the previous sections. Recall that $\uL = \Z^d$ and assume without loss of generality that $S$ is connected. An $S$-point of the source supermanifold of $\widetilde{\coP}_n$ is a triple
\beq
(\Lambda\colon S \times \Z^d \to \R^{d}, h \colon \uL \to G \wr \Sigma_n, \xi \colon S \times \R^{0|\delta} \rightarrow (X^{\times n})^{\im h}),\label{eq:triple}
\eeq
where the first term is an $S$-family of oriented lattices, the second term is a group homomorphism, and the third term is a map to the fixed points. Let  $\underline{n}=\coprod_{k \in K} I_k$ be the decomposition of $\underline{n}$ into transitive $\L$-sets for the action given by the composition $\L\to G\wr\Sigma_n\to \Sigma_n=\Aut(\underline{n})$. We use the notation $\Sigma_{I_k}<\Sigma_n$ for the subset of bijections $\underline{n}\to \underline{n}$ that are the identity on the complement of~$I_k\subset \underline{n}$. The homomorphism $h$ can then be factored as 
\[
\xymatrix{ & \prod_{k \in K} G \wr \Sigma_{I_k} \ar[d] \\ \L \ar@{-->}[ur] \ar[r]^{h} & G \wr \Sigma_n.}
\]
Let $\uL_k\subset \uL = \Z^d$ denote the sublattice that is the kernel of the composition 
\[
\uL\to G\wr \Sigma_{I_k}\to \Sigma_{I_k},
\]
i.e., the stabilizer of any element of $I_k\subset \underline{n}$. In light of the exact sequence $G^{\times |I_k|}\to G\wr \Sigma_{I_k}\to \Sigma_{I_k}$, we get a commutative diagram (in which the diagonal arrow is not necessarily $h$)
\[
\xymatrix{\uL_k \ar[r] \ar[d] & G^{\times |I_k|} \ar[d] \\ \uL \ar[r] \ar[dr] & G \wr \Sigma_{I_k} \ar[d] \\ & G\wr \Sigma_n.}
\]

To construct $\widetilde{\coP}_n$, we make a number of choices. For each $k \in K$ fix an element $i_k\in I_k$ and for each sublattice $\uL' \subseteq \uL$, choose a basis $\uL' \cong \uL$ with the property that the composite $\uL \cong \uL' \subseteq \uL$ is orientation preserving (i.e., has positive determinant). In particular, we have chosen a basis $\uL_k \cong \uL$ and the composite
\beq
M_{\uL_k} \colon \uL \cong \uL_k \subset \uL\cong \uL\label{eq:matrix}
\eeq
is a $d\times d$ matrix with integer entries and positive determinant, $M_{\uL_k}\in M^{\det>0}_{d\times d}(\Z)$. Further, define $ h_k \colon \uL \to G$ to be the composite
\beq
h_k\colon \uL \cong \L_k \to G^{\times |I_k|} \lra{\pi_{i_k}} G,\label{eq:mapproj1}
\eeq
where $\pi_{i_k}$ is the projection onto the $i_k^{\rm th}$ factor. Note that different choices of projection lead to $G$-conjugate maps $h_k\colon \uL \to G$. The proof of the next lemma is deferred to \Cref{app:b}.

\begin{lem}\label{lem:technical}
There is a commutative diagram 
\[
\xymatrix{(X^{\times n})^{\im h} \ar[rr] \ar[d]_{\cong} & & X^{\times n} \ar[d]^{\pi_{i_k}} \\
\prod_{k\in K} X^{\im h_k} \ar[r]_-{\pi_{i_k}} & X^{\im h_k} \ar[r] & X,}
\]
where the maps $h_k$ from~\eqref{eq:mapproj1} and the isomorphism $(X^{\times n})^{\im h} \xrightarrow{\cong} \prod_{k\in K} X^{\im h_k}$ depend on the choice of $i_k\in I_k$ for each $k \in K$. 
\end{lem}

Making use of this lemma, we obtain the map
\begin{equation}\label{eq:betak}
\xi_k\colon S\times \R^{0|\delta} \xrightarrow{\xi} (X^{\times n})^{\im h} \cong \prod_{k \in K} X^{\im h_k} \xrightarrow{\pi_{i_k}} X^{\im h_k}.
\end{equation}

\begin{rmk}
Assume that $i_{k}' \in I_k$ differs from $i_k$ and that $h_{k}'$ and $\xi_{k}'$ are the maps described above built from the choice of $i_{k}'$. Then there exists $g \in G$ that conjugates $h_k$ to $h_{k}'$ and the action of $g$ on $X$ sends $\xi_k$ to $\xi_{k}'$.
\end{rmk}

\begin{defn}\label{def:atlascoop}
Assume without loss generality that $S$ is connected. Define the map $\widetilde{\coP}_n\colon \mathcal{U}(X^{\times n}\sq G\wr \Sigma_n) \to \coprod_{t \le n}\mathcal{U}(X\sq G)^{\times t}$ on an $S$-point of the source~\eqref{eq:triple} to be the $S$-point of the target given by the tuple
\begin{equation}\label{eq:tildeP}
(\Lambda\circ M_{\uL_k} \colon S \times \uL \to \R^d, h_k \colon \uL \to G, \xi_k\colon S \times \R^{0|\delta} \rightarrow X^{\im h_k})_{k \in K} \in \mathcal{U}(X\sq G)^{\times |K|}
\end{equation}
using equations~\eqref{eq:matrix},~\eqref{eq:mapproj1} and~\eqref{eq:betak}. 
\end{defn}

\begin{thm} \label{thm:mainthm}
There is a 2-commutative diagram of stacks
\beq
\xymatrix{\mathcal{U}(X^{\times n}\sq G\wr \Sigma_n) \ar[r]^-{\widetilde{\coP}_n} \ar[d]_{p} & \coprod_{t \le n}\mathcal{U}(X\sq G)^{\times t} \ar[d]^{q} \\ \mathcal{L}^{d|\delta}_0(X^{\times n}\sq G\wr \Sigma_n) \ar[r]^-{\coP_n} & \Sym^{\le n}(\mathcal{L}^{d|\delta}_0(X\sq G)).}\label{eq:powercov}
\eeq 
\end{thm}
\bp
By construction, an $S$-point of an atlas as in \eqref{eq:generalatlas} gives rise to an $S$-point of $\mathcal{L}^{d|\delta}_0(X\sq G)$ coming from a map of Lie groupoids. Since $\coP_n$ sends an $S$-point coming from a map of Lie groupoids to an $S$-point coming from a map of Lie groupoids, it suffices to work with Lie groupoids in this proof.

We start with the $S$-point of $\mathcal{U}(X^{\times n}\sq G\wr \Sigma_n)$ given by the triple
 \beq\label{eq:tildePsource}
(\Lambda, h, \xi) = (\Lambda \colon S \times \uL \to \R^{d}, h \colon \uL \to G \wr \Sigma_n, \xi \colon S \times \R^{0|\delta} \rightarrow (X^{\times n})^{\im h}).
\eeq
Its image under $p$ is the $S$-point of $\mathcal{L}^{d|\delta}_0(X^{\times n}\sq G\wr \Sigma_n)$ determined by the map of Lie groupoids $\phi = p(\Lambda, h, \xi) \colon S \times \R^{d|\delta} \sq \uL \to X^{\times n} \sq G \wr \Sigma_n$. Making use of \eqref{eq:tildeP}, we have 
\[
\widetilde{\coP}_n(\Lambda, h, \xi) = (\Lambda\circ M_{\uL_k}, h_k, \xi_k)_{k \in K}.
\]
Let $\coprod_{k} \phi_k = \coprod_k q(\Lambda\circ M_{\uL_k}, h_k, \xi_k) \colon \Coprod{k} (S\times \R^{d|\delta})\sq \uL \to X \sq G$.

We wish to produce an equivalence of Lie groupoids making the following diagram commute:
\begin{equation} \label{eq:maindiagram}
\xymatrix{\Coprod{k} (S\times \R^{d|\delta})\sq \uL \ar[dr]_-{\coprod_{k} \phi_k} \ar@{-->}[rr]^-{\sim} & & (S \times \R^{d|\delta} \times \underline{n}) \sq \uL \ar[dl]^-{\coP_n(\phi)}  \\ & X \sq G. &  }
\end{equation}
The equivalence in question is given by the string of equivalences of Lie groupoids
\begin{equation} \label{eq:equivalencein}
\Coprod{k} (S\times \R^{d|\delta})\sq \uL \xrightarrow{\sim} \Coprod{k} (S\times \R^{d|\delta})\sq \uL_{k} \xrightarrow{\sim}  (S\times \R^{d|\delta} \times \underline{n}) \sq \uL.
\end{equation}
The first equivalence makes use of the chosen isomorphisms $\uL_k \cong \uL$ and the second equivalence is the equivalence of \cref{lem:factor}. 

We have explicit formulas for each map of Lie groupoids in \eqref{eq:maindiagram}. \cref{lem:factor} gives a formula for the equivalence. The definition of the atlas \eqref{eq:generalatlas} gives a formula for $\phi$ and $\phi_k$. \Cref{lem:ev} and \cref{lem:pull} together with \eqref{eq:liesquarepullbackformula} give a formula for $\coP_n(\phi)$.

To check that \eqref{eq:maindiagram} commutes on the objects of the Lie groupoids, fix $k \in K$ and consider the composite
\[
S\times \R^{d|\delta}\hookrightarrow S\times \R^{d|\delta}\times \underline{n} \to X^{\times n} \times \underline{n}\stackrel{\ev}{\to} X,
\]
where the first arrow is the inclusion at $i_k\in I_k\subset \underline{n}$, the second arrow comes from the pullback in the definition of $\coP_n$, and the third arrow is the evaluation map. The formula for this evaluation map in \Cref{lem:ev} shows that the composition is equal to 
\[
S\times \R^{d|\delta}\to S\times \R^{0|\delta} \lra{\xi} X^{\times n}\lra{\pi_{i_k}} X,
\]
where the first map is the projection to $\R^{0|\delta}$, and $\pi_{i_k}$ is the projection to the $i_k^{\rm th}$-factor.

To check that \eqref{eq:maindiagram} commutes on morphisms, note that the composite
\[
\uL \lra{M_{\uL_k}} \uL \lra{h} G \wr \Sigma_n\stackrel{p_{i_k}}{\to} G,\label{eq:compgrpd2}
\]
where $p_{i_k}$ is the projection onto the $i_k$th factor of $G$ ($p_{i_k}$ is not a group homomorphism unless $n=1$), is identical to $h_k$.

Finally, note that the construction of the equivalence in \eqref{eq:equivalencein} is natural in $S$.
\ep

\subsection{Adams operations on stacks of super tori}\label{sec:geoadams}

Super tori have canonical connected coverings associated with the sublattices of the form $n\Lambda\colon S\times \L\to S\times \bE^d$ for $n\in \N$, corresponding to muliplication of $\Lambda\colon S\times \L\to S\times \bE^d$ by~$n$. For an $S$-family of super tori over $\X$, we get a canonical $n^d$-fold covering space,
\beq
T^{d|\delta}_{n\Lambda} \to T^{d|\delta}_\Lambda\to \X\label{eq:adamsform}
\eeq
and the composition gives a new (connected) super torus over $\X$. 

\begin{defn} \label{defn:geometricAdams}The \emph{$n^{\rm th}$ geometric Adams cooperation} on super tori is the functor
\beq
\Psi_n\colon \mathcal{L}_{0}^{d|\delta}(\X)\to \mathcal{L}_{0}^{d|\delta}(\X)\label{eq:geoadams}
\eeq
that associates to an $S$-point of the source its canonical $n^d$-fold covering space~\eqref{eq:adamsform}. To a morphism given by a triangle as on the right in the diagram below, the Adams operation assigns the canonical map between covers
\beq
\begin{tikzpicture}[baseline=(basepoint)];
\node (A) at (0,0) {$T^{d|\delta}_{n\Lambda}$};
\node (B) at (3,0) {$T^{d|\delta}_\Lambda$};
\node (D) at (0,-1.5) {$T^{d|\delta}_{n\Lambda'}$};
\node (E) at (3,-1.5) {$T^{d|\delta}_{\Lambda'}$};
\node (F) at (6,-.75) {$\X$.};
\draw[->] (A) to  (B);
\draw[->,dashed] (A) to node [left] {$\tilde{f}$} (D);
\draw[->] (B) to node [left] {$f$} (E);
\draw[->] (D) to (E);
\draw[->] (B) to node [above] {$\phi$} (F);
\draw[->] (E) to node [below] {$\phi'$} (F);
\node (D) at (4,-.75) {$\twocommute$};
\path (0,-.75) coordinate (basepoint);
\end{tikzpicture}\nonumber
\eeq 

\end{defn}

The geometric Adams cooperation can be recovered from the geometric power cooperation. We will construct a map
\[
\tau_{nd} \colon \mathcal{L}_{0}^{d|\delta}(X \sq G) \to \mathcal{L}_{0}^{d|\delta}(X^{\times nd} \sq G \wr \Sigma_{nd}).
\]
Take an $S$-point of $\mathcal{L}^{d|\delta}(X \sq G)$, 
\[
(S \times \R^{d|\delta}) \sq \uL \to (S \times \R^{0|\delta}) \sq \uL \to X \sq G,
\]
determined by the map of supermanifolds $s \colon S \times \R^{0|\delta} \to X$ and the map of groups $h \colon \uL \to G$. The image of this $S$-point under $\tau_{nd}$ is given by the map of supermanifold obtained by composing with the diagonal,
\[
S \times \R^{0|\delta} \to X \to X^{\times n},
\]
together with the map of groups 
\[
\uL \lra{(h, t)} G \times \Sigma_{nd} \hookrightarrow G \wr \Sigma_{nd},
\]
where $t \colon \uL \to \uL/n\uL \hookrightarrow \Sigma_{nd}$ is the composite of the quotient map with the Cayley embedding, following \cite[proof of Proposition 2.9]{charpo}. By construction the kernel of the composite with the quotient map $G \wr \Sigma_{nd} \to \Sigma_{nd}$ is $n\uL$. It is easy to verify the following proposition:

\begin{prop}
There is an isomorphism of maps of stacks
\[
\coP_{nd} \circ \tau_{nd} \cong \Psi_n.
\]
\end{prop}

It is also possible to construct a map of atlases
\[
\widetilde{\Psi}_n \colon \Lat\times \coprod_{h \colon \uL\to G} \Map(\R^{0|\delta},X^{\im h}) \to \Lat\times \coprod_{h \colon \uL\to G} \Map(\R^{0|\delta},X^{\im h})
\]
covering the Adams operation. We will produce this as a product of maps. The map $\Lat \to \Lat$ sends an oriented based lattice $\uL \to \R^d$ to the composite $\uL \lra{\times n} \uL \to \R^d$. We define a map
\[
\coprod_{h \colon \uL\to G} \Map(\R^{0|\delta},X^{\im h}) \to \coprod_{h \colon \uL \to G} \Map(\R^{0|\delta},X^{\im h})
\]
by sending the component corresponding to $h$ to the component corresponding $h \circ n \colon \uL \lra{\times n} \uL \lra{h} G$ and using the map
\[
\Map(\R^{0|\delta},X^{\im h}) \to \Map(\R^{0|\delta},X^{\im (h \circ n)})
\]
induced by the inclusion $X^{\im h} \hookrightarrow X^{\im (h \circ n)}$. The proof of the following proposition is then similar to the proof of \Cref{thm:mainthm}.
\begin{prop}\label{prop:adamsmain}
There is a $2$-commutative diagram of stacks
\[
\xymatrix{\mathcal{U}(X\sq G) \ar[r]^{\widetilde{\Psi}_n} \ar[d] & \mathcal{U}(X\sq G) \ar[d] \\ \mathcal{L}_{0}^{d|\delta}(X \sq G) \ar[r]^{\Psi_n} & \mathcal{L}_{0}^{d|\delta}(X \sq G).}
\]
\end{prop}

\section{Power operations in complexified K-theory}

In this section we study $\mathcal{L}^{1|1}_0(X\sq G)$, the moduli stack of constant super loops in $X\sq G$. Functions on $\mathcal{L}^{1|1}_0(X\sq G)$ give a cocycle model for the complexified $G$-equivariant $\K$-theory of~$X$, see \cite{DBE_EquivTMF}. This allows us to compare the classical power operations on $\K_G(X)$ with the geometric power cooperation on $\mathcal{L}^{1|1}_0(X\sq G)$ introduced in \cref{sec:geompowerops}.

We will assume~$X$ is a smooth compact $G$-manifold in this section. This guarantees two things. First, $\K_G(X)$ can be computed as the Grothendieck group of finite dimendsional $G$-equivariant complex vector bundles on~$X$. Second, the complexification~$\C\otimes \K(X)$ is canonically isomorphic to the 2-periodic de Rham cohomology of~$X$ with complex coefficients via the Chern character. There is a similar description, due to Atiyah and Segal in \cite{Atiyah_Segal}, of the complexification of the $G$-equivariant K-theory of~$X$.

\subsection{Power operations in K-theory}\label{sec41}

We start with a quick review of the classical power operations; the standard reference is~\cite{Atiyah}. 

Let $\Vect_G(X)$ denote the groupoid whose objects are finite dimensional $G$-equivariant complex vector bundles on~$X$ and whose morphisms are isomorphisms of $G$-equivariant vector bundles. The Grothendieck group of $\Vect_G(X)$ is the $G$-equivariant complex $\K$-theory of $X$, denoted $\K_G(X)$. The classical power operations in $\K$-theory come from the $n^{\rm th}$ external power of a $G$-equivariant vector bundle $V\to X$, 
\beq
\Vect_G(X)\to \Vect_{G\wr \Sigma_n}(X^{\times n}),\qquad (V\to X)\mapsto (V^{\boxtimes n}\to X^{\times n}). \label{eq:vectpower}
\eeq
Atiyah and Segal produce an isomorphism~\cite[Theorem~2]{Atiyah_Segal}
\beq
\C\otimes \K_G(X)\cong \Big(\prod_{g\in G} {\rm H}^{\ev}(X^g;\C)\Big)^G,\label{eq:Chernchar}
\eeq
where $X^g\subset X$ denotes the $g$-fixed points. The isomorphism is induced by the equivariant Chern character
\[
{\rm Ch}_G\colon \K_G(X)\to  \Big(\prod_{g\in G} {\rm H}^{\ev}(X^g;\C)\Big)^G.
\]
We will use the notation 
\[
{\rm Ch}_g\colon \K_G(X)\to  {\rm H}^{\ev}(X^g;\C)
\]
to denote the component of this map corresponding to $g$ and $[\omega_g]\in {\rm H}^{\ev}(X^g;\C)$ the component of a class~$[\omega]\in \C\otimes \K_G(X)$ under the decomposition in \eqref{eq:Chernchar}.

We can refine the equivariant Chern character to take values in $G$-invariant closed differential forms. Let $\Vect_G^\nabla(X)$ denote the groupoid of $G$-equivariant vector bundles with $G$-invariant connection $\nabla$ and connection preserving equivariant vector bundle isomorphisms. Let $\Omega^\bullet(X)$ denote the superalgebra of differential forms on $X$, $\Omega^\bullet_{\cl}(X)$ denote the closed forms, and $\Omega^{\ev}_{\cl}(X)$ denote the closed forms of even degree. Then the refinement of the equivariant Chern character to the level of differential forms is a map 
\[
{\rm Ch}_G\colon \Vect_G^\nabla(X)/{\sim}\to \Big(\prod_{g\in G} \Omega^\ev_\cl(X^g)\Big)^G
\]
from isomorphism classes of equivariant vector bundles with $G$-invariant connection to equivariant differential forms. It is given by the formula,
\beq
{\rm Ch}(V,\nabla)_g={\rm Tr}(\exp(\nabla^2)\circ \rho(g))\in \Omega^{\ev}_\cl(X^g),\label{equivChernformula}
\eeq
where $\rho(g)\in \Gamma(X^g,\End(V))$ is the family of endomorphisms determined by $g\in G$ and the equivariant structure on $V$. 

\begin{notation}\label{notation:omegag} Above and throughout, $\omega_g\in \Omega^\bullet(X^g)$ is the $g$th component of a differential form in the product~$\omega\in \prod_{g\in G} \Omega^\bullet(X^g)$.
\end{notation}
The power operation~\eqref{eq:vectpower} on equivariant vector bundles has an obvious lift to vector bundles with connection
\[
\P_{n}^{\mathrm{classical}}\colon \Vect_G^\nabla(X)\to \Vect_{G\wr \Sigma_n}^\nabla(X^{\times n}),\qquad (V,\nabla)\mapsto (V^{\boxtimes n},\nabla^{\boxtimes n}),
\]
and we can study the interaction of this power operation with the equivariant Chern character. 

For any $G$-manifold~$X$, there is an isomorphism of algebras
\beq
\Big(\prod_{g\in G} \Omega^\ev_\cl(X^g)\Big)^G\cong C^\infty(\mathcal{L}^{1|1}_0(X\sq G)).\label{eq:Kiso}
\eeq
This was proved in~\cite[\S2]{DBE_EquivTMF} and is reviewed in~\Cref{sec:Kgeo} below. We may use the geometric power operation $\P_n$ and \eqref{eq:Kiso} to construct a power operation on forms via the following square:
\[
\xymatrix{\displaystyle\Big(\prod_{g\in G} \Omega^\ev_\cl(X^g)\Big)^G \ar@{-->}[r]^-{\P_n} \ar[d]_{\cong} & \displaystyle\Big(\prod_{({\bf g},\sigma)\in G\wr \Sigma_n} \Omega^\ev_\cl((X^{\times n})^{({\bf g},\sigma)})\Big)^{G \wr \Sigma_n} \ar[d]^{\cong} \\ C^\infty(\mathcal{L}^{1|1}_0(X\sq G)) \ar[r]_-{\P_n} & C^\infty(\mathcal{L}^{1|1}_0(X^n\sq G \wr \Sigma_n)).}
\]
The main result of this section is the following. 

\begin{thm}\label{thm:main11}
Let $X$ be a smooth and compact $G$-manifold and $n \ge 1$. There is a commutative diagram
\[
\xymatrix{\mathrm{Vect}^{\nabla}_{G}(X) \ar[r]^-{\P_{n}^{\mathrm{classical}}} \ar[d]_{{\rm Ch}_G} & \mathrm{Vect}^{\nabla}_{G \wr \Sigma_n}(X^{\times n}) \ar[d]^{{\rm Ch}_{G\wr\Sigma_n}} \\  \displaystyle \Big(\prod_{g\in G} \Omega^\ev_\cl(X^g)\Big)^G \ar[r]_-{\P_n} & \displaystyle\Big(\prod_{({\bf g},\sigma)\in G\wr \Sigma_n} \Omega^\ev_\cl((X^{\times n})^{({\bf g},\sigma)})\Big)^{G \wr \Sigma_n},}
\]
where $\P_n$ is induced by the geometric power operation under the isomorphism~\eqref{eq:Kiso}. 
\end{thm}

We prove this directly by computing the geometric power operation~$\P_n$ using the tools of the previous section and then comparing it with the formula for the power operation in equivariant K-theory. 

We recall that the equivalence relation of concordance (see \Cref{concordance2}) mediates between cocycles and cohomology classes. From \Cref{concordance}, the geometric power cooperation is compatible with taking concordance classes. Applying Stokes theorem, concordance classes of  closed differential forms are precisely de~Rham cohomology classes. Together with the isomorphism~\eqref{eq:Chernchar} we obtain the following corollary. 

\begin{cor}
With notation as in the previous theorem, taking concordance classes and forming the Grothendieck group of isomorphism classes of vector bundles in \Cref{thm:main11}, we obtain the commutative square
\[
\xymatrix{\K_G(X) \ar[r]^-{\P_{n}^{\mathrm{classical}}} \ar[d] & \K_{G \wr \Sigma_n}(X^{\times n}) \ar[d] \\ \C \otimes \K_G(X) \ar[r]_-{\P_n} & \C \otimes \K_{G \wr \Sigma_n}(X^{\times n}).} 
\]
\end{cor}

Making use of the isomorphism \eqref{eq:Kiso}, the $n^{\text{th}}$ geometric Adams cooperation $\Psi_n \colon \mathcal{L}^{1|1}(X \sq G) \to \mathcal{L}^{1|1}(X \sq G)$, induces the $n^{\rm th}$ Adams operation 
\[
\Psi^n\colon \Big(\prod_{g\in G} \Omega^\ev_\cl(X^g)\Big)^G \to \Big(\prod_{g\in G} \Omega^\ev_\cl(X^g)\Big)^G.
\]
\begin{notation} \label{notation:deg} Define the $C^\infty(X)$-linear derivation,
\[
\deg\colon \Omega^\bullet(X)\to \Omega^\bullet(X),
\]
to be the derivation that on a degree $j$-form $\omega \in \Omega^j(X)$ is multiplication by $j$, $\deg(\omega)=j\omega$. For a real number $r\in \R$, define the algebra endomorphism $r^{\deg} \colon \Omega^\bullet(X)\to \Omega^\bullet(X)$ by $r^{\deg}(\omega)=r^j\omega$ for $\omega\in \Omega^j(X)$. If $r$ is positive, we similarly define $r^{\deg/2}$ by $r^{\deg/2}(\omega)=r^{j/2}\omega$. 
\end{notation}

\begin{cor}\label{thm:KAdams}
The linear map $\Psi^n$ is determined by 
\beq
\Psi^n(\omega)_g= n^{\deg} \omega_{g^n}, \qquad \omega\in \prod_{g\in G} \Omega^{\ev}_{\rm cl}(X^g),\label{eq:Adams}
\eeq
where on the right side above, we regard $\omega_{g^n}\in \Omega^{\ev}_\cl(X^g)$ by restriction along the inclusion~$X^g\subset X^{g^n}$. Hence, the $n^{\rm th}$ geometric Adams operation $\Psi^n$ is a cocycle refinement of the classical $n^{\rm th}$ Adams operation on $\C\otimes \K_G(X)$. 
\end{cor}

\subsection{The geometric model}\label{sec:Kgeo}

In this subsection we review results relating moduli spaces of $1|1$-dimensional super circles to complexified equivariant $\K$-theory. The main goal is to spell out \Cref{defn:edinvariantmaps} in the case of the $1|1$-dimensional rigid conformal model geometry (compare with~\cite[\S4.2]{ST11}) and to characterize functions on $\mathcal{L}^{1|1}_0(X\sq G)$ in terms of differential forms on fixed point sets. We start with the relevant model geometry. 

\begin{defn}\label{defn:11model} Define the super Lie group $\bE^{1|1}$ as having underlying supermanifold $\R^{1|1}$ with group structure
\[
(t,\theta)\cdot (s,\eta)=(t+s+\theta\eta,\theta+\eta),\quad (t,\theta),(s,\eta)\in \R^{1|1}(S). 
\]
Consider the action of $\R^\times$ on $\bE^{1|1}$ through homomorphisms given by
\[
\mu\cdot (t,\theta)=(\mu^2t,\mu\theta),\quad \mu\in \R^\times(S),\ (t,\theta)\in \bE^{1|1}(S).
\]
Define the \emph{rigid conformal model geometry} to be the model geometry with model space~$\R^{1|1}$ together with its (left) action by~$\bE^{1|1}\rtimes \R^\times$. The composition 
\[
S\times \R^{1|1}\to S\times (\bE^{1|1}\rtimes \R^\times)\times \R^{1|1}\stackrel{\rm act}{\to} S\times \R^{1|1}
\]
is \emph{an $S$-family of rigid conformal maps}, where the first arrow is defined by an $S$-point~$(t,\theta,\mu)\in (\bE^{1|1}\rtimes \R^\times)(S)$ of the rigid conformal transformation group, and the second arrow by the action of $\bE^{1|1}\rtimes \R^\times$ on $\R^{1|1}$. 
\end{defn}

We will use the notation $\bE<\bE^{1|1}$ to denote the reduced subgroup of $\bE^{1|1}$, i.e., $\R$ with its usual additive structure. The following is a specialization of \Cref{defn:supertori} relative to the rigid conformal model geometry above that we provide both for the reader's convenience and also to establish some notation. 

\begin{defn} Let $\ell\in \R_{>0}(S)$. An \emph{$S$-family of super circles} is a quotient
\[
T^{1|1}_\ell:=(S\times \R^{1|1})/\Z
\]
for the $\Z$-action on $S\times \R^{1|1}$
\[
n\colon (t,\theta)\mapsto (t+n\ell,\theta),\quad (t,\theta)\in \R^{1|1}(S), \quad n \in \Z.
\]
The projection $T^{1|1}_\ell\to S$ allows us to view $T^{1|1}_\ell$ as a bundle over $S$. The canonical cover
\[
S\times \R^{1|1}\to (S\times \R^{1|1})/\Z=T^{1|1}_\ell
\]
endows the $S$-family of super circles $T^{1|1}_\ell$ with a fiberwise rigid conformal structure.
\end{defn}

The following pair of definitions are special cases of \Cref{defn:stori} and \Cref{defn:edinvariantmaps}. 

\begin{defn} For a stack $\X$, the \emph{super loop stack of~$\X$}, denoted $\mathcal{L}^{1|1}(\X)$, is the stack associated to the prestack whose objects over $S$ are pairs $(\ell,\phi)$, where $\ell\in \R_{>0}(S)$ determines a family of super circles $T^{1|1}_\ell$ and $\phi\colon T^{1|1}_\ell\to \X$ is a map. Morphisms between such objects over~$S$ consist of triangles
\beq
\begin{tikzpicture}[baseline=(basepoint)];
\node (A) at (0,0) {$T^{1|1}_\ell$};
\node (B) at (3,0) {$T^{1|1}_{\ell'}$};
\node (C) at (1.5,-1.5) {$\X$};
\node (D) at (1.5,-.6) {$\twocommute$};
\draw[->] (A) to node [above=1pt] {$\cong$} (B);
\draw[->] (A) to node [left=1pt]{$\phi$} (C);
\draw[->] (B) to node [right=1pt]{$\phi'$} (C);
\path (0,-.75) coordinate (basepoint);
\end{tikzpicture}\label{11triangle}
\eeq
that 2-commute, where the horizontal arrow is an isomorphism of $S$-families of super rigid conformal $1|1$-manifolds.

\end{defn}

\begin{defn}
The \emph{stack of ghost super loops}, denoted $\mathcal{L}_0^{1|1}(\X)$, is the full substack of $\mathcal{L}^{1|1}(\X)$ generated by pairs $(\ell,\phi)$ where $\phi\colon T^{1|1}_\ell\to \X$ is a map given by the composition 
\[
T^{1|1}_\ell \stackrel{\sim}{\leftarrow} [S\times \R^{1|1}\sq \Z] \stackrel{\pi}{\to} [S\times\R^{0|1}\sq \Z]\to \X,
\]
where $\pi$ is induced by the projection $\R^{1|1}\to \R^{0|1}$.
\end{defn}

Since $\uL = \Z$ in this case, we have $\Hom(\uL,G) \cong G$ (where $h \mapsto h(1)$) and an oriented lattice $\Lambda\colon S\times \Z\to \R$ can be identified with $\ell:=\Lambda(1)\in \R_{>0}(S)$. The atlas from \Cref{prop:atlas} specializes to a map
\beq
\coprod_{g\in G} \R_{>0}\times \Map(\R^{0|1},X^g)\to \mathcal{L}^{1|1}_0(X\sq G).\label{eq:atlas11}
\eeq
This map takes an $S$-point $\ell\in \R_{>0}(S)$ and a map $\xi\colon S\times \R^{0|1}\to X^g$ to the map of stacks $T^{1|1}_\ell\to [X\sq G]$ obtained from the maps of Lie groupoids
\[
T^{1|1}_\ell=(S\times \R^{1|1})/\Z \leftarrow S\times \R^{0|1}\sq \Z\stackrel{\xi\sq \Z}{\to} X^g\sq \Z\to X\sq G.
\]
Here, the $\Z$-action on $X^g$ is trivial, the arrow labeled by $\xi\sq \Z$ is determined by~$\xi$ (using that the $\Z$-actions are trivial), and the map $X^g\sq \Z\to X\sq G$ comes from the map $\Z \to G$ picking out $g\in G$ along with the inclusion $X^g\subset X$. By virtue of~\eqref{eq:atlas11} being an atlas, functions on the stack $\mathcal{L}^{1|1}_0(X\sq G)$ form a subalgebra 
\beq
&&C^\infty(\cL_0^{1|1}(X\sq G))\hookrightarrow C^\infty\Big(\coprod_{g\in G} \R_{>0}\times \Map(\R^{0|1},X^g)\Big) \cong \Omega^\bullet\Big(\coprod_{g\in G} X^g;C^\infty(\R_{>0})\Big), \label{eq:inclusion}
\eeq
where the inclusion is the pullback of functions along~\eqref{eq:atlas11}. The isomorphism on the right uses two facts. The first fact is that, for any pair of supermanifolds $M$ and $N$, there is a canonical isomorphism $C^\infty(M\times N)\cong C^\infty(M; C^\infty(N))$ (see \cite[Example 49]{HST}) . The second fact is that there is a canonical isomorphism  $C^\infty(\Map(\R^{0|1},X))\cong \Omega^\bullet(X)$ (see \cite{HKST}). 

\begin{rmk}\label{eq:rmkonell}
 In a slight abuse of notation, we will use $\ell=\id\in \R_{>0}(\R_{>0})\hookrightarrow \R(\R_{>0})\subset C^\infty(\R_{>0})$ to denote the identity map, identified with the standard coordinate on~$\R_{>0}$. This is justified because $\ell\colon S=\R_{>0}\to \R_{>0}$ is the universal $S$-point of $\R_{>0}$. 
\end{rmk}

Our next goal is to give a more explicit description of~\eqref{eq:inclusion}. Define a map 
\beq\label{eq:inclimage}
&&\left(\prod_{g\in G} \Omega^\ev_{\cl}(X^g)\right)^G\hookrightarrow \prod_{g\in G} \Omega^\bullet( X^g;C^\infty(\R_{>0})), \qquad \omega\mapsto \ell^{\deg/2}\omega,
\eeq
where $\deg$ is defined in \cref{notation:deg}. This connects with complexified K-theory as follows.

\begin{prop}[{\cite[\S2]{DBE_EquivTMF}}] \label{prop:Kthycocycle} 
The functions on the atlas that descend to the stack are determined by the commutative square
\beq
\begin{tikzpicture}[baseline=(basepoint)];
\node (A) at (0,0) {$\left(\prod_{g\in G} \Omega^\ev_{\cl}(X^g)\right)^G$};
\node (B) at (6,0) {$C^\infty(\mathcal{L}^{1|1}_0(X\sq G))$};
\node (C) at (0,-1.5) {$\prod_{g\in G} \Omega^\bullet( X^g;C^\infty(\R_{>0}))$};
\node (D) at (6,-1.5) {$C^\infty\left(\coprod_{g\in G} \R_{>0}\times \Map(\R^{0|1},X^g)\right)$,};
\draw[->,dashed] (A) to node [above] {$\cong$} (B);
\draw[->,right hook-latex] (A) to (C);
\draw[->] (C) to node [above] {$\cong$} (D);
\draw[->,right hook-latex] (B) to (D);
\path (0,-.75) coordinate (basepoint);
\end{tikzpicture}\label{eq:Kthycocycle}
\eeq
where the inclusion on the left is the map~\eqref{eq:inclimage}. The isomorphism~\eqref{eq:Chernchar} implies that the ring $C^\infty(\mathcal{L}^{1|1}_0(X\sq G))$ is a cocycle model for complexified equivariant K-theory. 
\label{prop:11present}
\end{prop}

\subsection{Computing the total geometric power operation}\label{sec:42} The goal of this section is to give a formula for the total geometric power operation for the model geometry and super circles described in the previous sections. As per \Cref{thm:mainthm}, we compute the geometric power operation in terms of the atlas for the stack $\mathcal{L}^{1|1}_0(X\sq G)$ described above. 

From \cref{def:atlascoop}, the general formula for the map $\widetilde{\coP}_n$ on an $S$-point of the atlas is
\beq
\widetilde{\coP}_n(\Lambda, h, \xi) = (\Lambda\circ M_{\uL_k}, h_k, \xi_k)_{k \in K},\label{eq:genform}
\eeq
where $\Lambda\colon S\times \L\to \R^d$ is a based lattice, $h\colon \L\to G\wr \Sigma_n$ is a homomorphism, and $\xi\colon S\times \R^{0|1}\to (X^{\times n})^{\im(h)}$ is a map of supermanifolds. We specialize this to our setting, where $\uL = \Z$. As mentioned before, a homomorphism $\Z \to G$ may be identified with the image of $1\in \Z$, i.e., an element of $G$. An oriented lattice $\Lambda\colon S\times \Z\to \R$ can be identified with $\ell:=\Lambda(1)\in \R_{>0}(S)$. So an $S$-point of the atlas for $\cL_0^{1|1}(X^{\times n}\sq G\wr\Sigma_n)$, given by a triple $(\Lambda, h, \xi)$, is equivalent to the data of the triple
\[
(\ell,({\bf g},\sigma),\xi),
\]
where $({\bf g},\sigma) = h(1) = (g_1, \ldots, g_n, \sigma)\in G\wr\Sigma_n$. Via this identification, the value of $\widetilde{\coP}_n$ on $(\ell,({\bf g},\sigma),\xi)$ can be understood as follows. Factor $\sigma = \sigma_1\cdots \sigma_{|K|}$, where $\sigma_k$ is a cycle, according to the decomposition $\underline{n} = \Coprod{k} I_k$ of $\underline{n}$ into transitive $\Z$-sets. It follows that $M_{\uL_k}$ is multiplication by the order of $\sigma_k$ (this is also $|I_k|)$ and $h_k$, which depends on our choice of $i_k \in I_k$, corresponds to 
\beq
{\bf g}_k = g_{i_k} g_{\sigma_k(i_k)} \cdots g_{(\sigma_k)^{|I_k|-1}(i_k)} \in G.\label{eq:hk}
\eeq
Thus $X^{\im h_k} = X^{{\bf g}_k}$. Using this notation, we see that $(\Lambda\circ M_{\uL_k}, h_k, \xi_k)$ is the same data as the triple
\beq
(|I_k|\ell, {\bf g}_k, \xi_k)\in (\R_{>0}\times G\times \Map(\R^{0|1},X^{{\bf g}_k}))(S).\label{eq:copntarget11}
\eeq
Using the map~\eqref{eq:inclimage} and \cref{prop:Kthycocycle}, the geometric power operation determines the map on differential forms
\beq
\big(\prod_{g\in G} \Omega^\ev_\cl(X^g)\big)^G&\cong& C^\infty(\mathcal{L}^{1|1}_0(X\sq G))\nonumber \\
&& \xrightarrow{\mathbb{P}_n} C^\infty(\mathcal{L}^{1|1}_0(X^{\times n}\sq G\wr\Sigma_n))\cong \left(\prod_{({\bf g},\sigma) \in G\wr \Sigma_n} \Omega^{\ev}_\cl((X^{\times n})^{({\bf g},\sigma)})\right)^{G\wr\Sigma_n}.\label{eq:diffpower}
\eeq
In the following proposition, for $\omega\in (\prod_{g\in G} \Omega^\ev_\cl(X^g))^G$, we denote the image under the above composition by $\P_n(\omega)$, with the isomorphisms to differential forms implicitly understood. We will also use \cref{notation:omegag} for~$\omega_g\in \Omega^\bullet(X^g)$.

\begin{prop}\label{prop:11formula}
The geometric power operation satisfies the formula
\begin{align*}
\P_n(\omega)_{({\bf g},\sigma)} & =(|I_1|^{\deg/2} \omega_{{\bf g}_{1}})(|I_2|^{\deg/2} \omega_{{\bf g}_{2}})\cdots (|I_{|K|}|^{\deg/2} \omega_{{\bf g}_{|K|}}) \\
& \in   \Omega^{\ev}_\cl(\prod_k X^{{\bf g}_{k}}) \cong \Omega^{\ev}_\cl((X^{\times n})^{({\bf g},\sigma)}),
\end{align*}
for $\omega\in (\prod_{g\in G} \Omega^\ev_\cl(X^g))^G$, $\P_n(\omega)\in(\prod_{({\bf g},\sigma) \in G\wr \Sigma_n} \Omega^{\ev}_\cl((X^{\times n})^{({\bf g},\sigma)}))^{G\wr\Sigma_n}$, $({\bf g},\sigma)\in G\wr\Sigma_n$, and ${\bf g}_{k}\in G$ defined in~\eqref{eq:hk}.
\end{prop}

\begin{proof}
By \cref{prop:Kthycocycle} functions on $\cL_0^{1|1}(X\sq G)$ are in bijection with functions on the atlas of the form $(\ell^{\deg/2}\omega_g)_{g\in G}$ where $\omega_g\in \Omega^\ev_\cl(X^{g})$. The restriction of the power cooperation to the component of the atlas \eqref{eq:atlas11} indexed by $({\bf g},\sigma)\in G\wr \Sigma_n$ is
\beq
\R_{>0}\times \Map(\R^{0|1},(X^{\times n})^{({\bf g},\sigma)})\stackrel{\widetilde{\coP}_n}{\to} \coprod_{k\in K} \R_{>0}\times \Map(\R^{0|1},X^{{\bf g}_k})
\label{eq:coPnmap11}
\eeq
\beq
&&\Omega^\bullet(X^{{\bf g}_k};C^\infty(\R_{>0}))\cong C^{\infty}(\R_{>0}\times \Map(\R^{0|1},X^{{\bf g}_k}))\nonumber\\ 
&&\to C^{\infty}(\R_{>0}\times \Map(\R^{0|1},(X^{\times n})^{({\bf g},\sigma)}))\cong \Omega^\bullet((X^{\times n})^{({\bf g},\sigma)};C^\infty(\R_{>0})),\nonumber
\eeq
where we have pre- and post-composed with the isomorphism between $C^\infty(\R_{>0})$-valued differential forms and functions on the atlas from~\eqref{eq:inclusion}. Next, we apply each of these maps to $\ell^{\deg/2}\omega_{{\bf g}_k}$ and take their product over $k$, obtaining
\begin{align*}
    \widetilde{\coP}_n^*(\ell^{\deg/2}\omega_{{\bf g}_k})_{k\in K} & =\prod_k ((|I_k|\ell)^{\deg/2}\omega_{{\bf g}_k}) \\
    & \in \Omega^{\ev}_\cl(\prod_k X^{{\bf g}_{k}};C^\infty(\R_{>0})) \cong \Omega^{\ev}_\cl((X^{\times n})^{({\bf g},\sigma)};C^\infty(\R_{>0})).
\end{align*}
By \cref{thm:mainthm}, the function $\widetilde{\coP}_n^*(\ell^{\deg/2}\omega_{{\bf g}_k})_{k\in K}$ on the atlas descends to a function on the stack $\cL_0^{1|1}((X^{\times n})\sq G\wr \Sigma_n)$ and computes the geometric power operation. Finally, we identify this map between functions on stacks with the map of differential forms~\eqref{eq:diffpower}. Using~\eqref{eq:inclimage} this has the effect of removing the dependence on $\ell$, obtaining the formula for the power operation in the statement of the proposition. 
\ep

Next we compare with the total power operation in K-theory. 
\begin{proof}[Proof of \Cref{thm:main11}] 
From the definition of the power operation in K-theory and the equivariant Chern character, for ${({\bf g},\sigma)\in G\wr \Sigma_n}$ we have
\beq
{\rm Ch}(\P_n^{\rm classical}(V,\nabla))_{({\bf g},\sigma)}&=&{\rm Ch}(V^{\boxtimes n},\nabla^{\boxtimes n})_{({\bf g},\sigma)}\nonumber\\
&=& {\rm Tr}(\exp((\nabla^{\boxtimes n})^2)\circ \rho({\bf g},\sigma))\in \Omega^{\ev}_\cl((X^{\times n})^{({\bf g},\sigma)}),\nonumber
\eeq
where we have used the notation for the Chern character from~\eqref{equivChernformula}. We can express 
\[
\exp\Big((\nabla^{\boxtimes n})^2\circ \rho({\bf g},\sigma)\Big)\in \Gamma\Big((X^{\times n})^{({\bf g},\sigma)};\End(V)\Big)
\] 
as an external tensor product of endomorphisms, one for each subset $I_k \subset \underline{n}$ on which $\sigma$ acts transitively. Indeed, we have
\[
\exp((\nabla^{\boxtimes n})^2\circ \rho({\bf g},\sigma))= \boxtimes_{k \in K} \exp((\nabla^{\boxtimes |I_k|})^2\circ \rho(({\bf g},\sigma)_k)),
\]
and so the trace can be written as a product of the factors,
\beq
{\rm Tr}(\exp((\nabla^{\boxtimes n})^2\circ \rho({\bf g},\sigma)))=\prod_{k\in K} {\rm Tr}( \exp((\nabla^{\boxtimes |I_k|})^2\circ \rho(({\bf g},\sigma)_k))).\label{eq:classicalfactors}
\eeq
In the above, $({\bf g},\sigma)_k$ denotes the restriction of the action of $({\bf g},\sigma) \in G \wr \Sigma_n$ on $V^{\boxtimes n}$ to $V^{\boxtimes |I_k|}$. We observe that
\[
{\rm Tr}(\exp((\nabla^{\boxtimes |I_k|})^2\circ \rho(({\bf g},\sigma)_k)=|I_k|^{\deg/2}{\rm Tr}(\exp(\nabla^2)\circ \rho({\bf g}_k))
\]
and so ${\rm Ch}(\P_n^{\rm classical}(V,\nabla))_{({\bf g},\sigma)}$ is the product over $k\in K$ of the factors on the right.

On the other hand, we can apply the geometric power operation $\P_n$ as computed in \Cref{prop:11formula} to the differential form $\omega={\rm Ch}_G(V,\nabla)\in (\prod_{g\in G} \Omega^\ev_\cl(X^g))^G$.
From the definition of the equivariant Chern character~\eqref{equivChernformula}, we have the equality
\beq
{\rm Tr}(\exp(\nabla^2)\circ \rho({\bf g}_k))={\rm Ch}_{{\bf g}_k}(V,\nabla)=\omega_{{\bf g}_k}.\label{eq:equalityofPn}
\eeq
Multiplying by the factor $|I_k|^{\deg/2}$ and taking the product over~$k$, the left hand side of~\eqref{eq:equalityofPn} gives the formula~\eqref{eq:classicalfactors} for ${\rm Ch}_G(\P_n^{\rm classical}(V,\nabla))$ and the right hand side is the formula for $\P_n(\omega)$. Hence $\P_n({\rm Ch}_G(V,\nabla))={\rm Ch}_G(\P_n^{\rm classical}(V,\nabla))$, proving the theorem. 
\ep

\begin{proof}[Proof of \Cref{thm:KAdams}] 
We compute using the atlas from~\Cref{prop:11present} and consider maps
\beq
\widetilde{\Psi}_n\colon \R_{>0}\times \Map(\R^{0|1},X^g)\to \R_{>0}\times \Map(\R^{0|1},X^{g^n})\label{eq:11adamscomp}
\eeq
from~\Cref{prop:adamsmain} for each $g\in G$. Then
\[
\widetilde{\Psi}_n(\ell,g,\xi)= (n\ell,g^n,i\circ \xi),\qquad \ell\in \R_{>0}(S), \ g\in G, \xi\in \Map(\R^{0|1},X^g)(S),
\] 
where $i\circ \xi$ is the composition $S\times \R^{0|1}\xrightarrow{\xi} X^g\stackrel{i}{\hookrightarrow} X^{g^n}$. Using \Cref{prop:Kthycocycle} to identify functions on the atlas that descend to the stack, the pullback of such functions along~\eqref{eq:11adamscomp} is given by 
\[
\ell^{\deg/2}\omega_g\mapsto (n\ell)^{\deg/2}\omega_{g^n}=n^{\deg/2}(\ell^{\deg/2}\omega_{g^n})\in \Omega^\ev(X;C^\infty(\R_{>0})),\qquad \omega_g\in \Omega^{\ev}_{\rm cl}(X^g). 
\]
This gives precisely the claimed formula for the Adams operations on cocycles. 
\ep

\begin{rmk} 
The factor of $n^{\deg/2}$ in \eqref{eq:Adams} comes about precisely because of the dependence of the function on the length of the super circle, together with how function on the stack $\cL_0^{1|1}(X\sq G)$ behaves under rescalings of this length. This points to a salient difference between geometric and topological field theories: without a length parameter on loops in~$X$, it is difficult to see how the Adams operations could emerge from the geometry. 
\end{rmk}

\section{Power operations in equivariant elliptic cohomology} \label{ellsection}

Next we turn our attention to the stack $\mathcal{L}^{2|1}_0(X\sq G)$ of maps from $2|1$-dimensional super tori to~$[X\sq G]$. Functions on $\mathcal{L}^{2|1}_0(X\sq G)$ give a cocycle model for a version of equivariant elliptic cohomology over~$\C$, reviewed in \Cref{sec:Devoto} below. The geometric power operation can then be used to construct power operations in complexified equivariant elliptic cohomology. 

The main results in this section are \Cref{prop:21formula} and \Cref{thm:TMFAdams}. The former gives an explicit formula for the total geometric power operation in equivariant elliptic cohomology, and the latter gives a formula for the effect of the geometric Adams operation in equivariant elliptic cohomology. As far as the authors are aware, these are the first constructions of these operations. So in contrast to the situation for K-theory, there is nothing explicit to which we can compare these formulas. However, the essence of these power operations conforms to the expected structures in chromatic homotopy theory, e.g., see~\cite{Baker,Andopower,Tamanoi2,Tamanoi1,Ganterpower, charpo}. We comment on this in more detail in \Cref{Sec:Ethy} by comparing with the character of the total power operation in Morava $\E$-theory.

\subsection{Equivariant elliptic cohomology over $\C$}\label{sec:Devoto}

For $G$ a finite group, Devoto~\cite{Devoto} used equivariant Thom spectra to define a $G$-equivariant refinement of the elliptic cohomology of Landweber, Ravenel, and Stong~\cite{LandweberRavenelStong}. Devoto's construction relies on an equivariant elliptic genus and elliptic cohomology with level structure for the congruence subgroup~$\Gamma_0(2)<\SL_2(\Z)$. In their study of moonshine and elliptic cohomology, Baker and Thomas~\cite{BakerThomas,Thomas} hypothesized an extension of Devoto's theory for the full modular group~$\SL_2(\Z)$ using the equivariant Witten genus. Over~$\C$, the construction of such an extension is straightforward. It is this version of equivariant elliptic cohomology that appears in the work of Ganter~\cite{GanterHecke} and Morava~\cite{Moravamoon}. As the literature can be somewhat diffuse, in this subsection we give a self-contained definition of this version of equivariant elliptic cohomology over~$\C$. 

\begin{defn}
Let $\Lat\subset \C\times \C$ denote the complex manifold of based, oriented lattices consisting of monomorphisms $\Lambda \colon \Z^2\to \C$ such that the ratio of the generators $\ell = \Lambda(1,0)$ and $\ell' = \Lambda(0,1)$ lies in the upper-half plane, $\ell'/\ell \in \mathfrak{h}\subset \C$. We observe that there is a diffeomorphism 
\[
\Lat\stackrel{\sim}{\to} \C^\times\times \mathfrak{h},\qquad (\ell,\ell')\mapsto (\ell,\ell'/\ell).
\]
\end{defn}

\begin{rmk} 
The upper half plane condition on oriented lattices above is equivalent to $(\ell,\ell')$ giving an oriented basis for $\R^2$. 
\end{rmk} 

There is an action of $\C^\times\times \SL_2(\Z)$ on $\Lat$ given by
\beq
&&(\ell,\ell')\mapsto \Big(\mu(a\ell+b\ell'),\mu(c\ell+d\ell')\Big),\qquad \mu \in \C^\times,\qquad \left[\begin{array}{cc} a & b \\ c & d \end{array}\right]\in \SL_2(\Z). \label{eq:Latact}
\eeq

\begin{defn}\label{defn:MF}
Let $j \in \Z$ and let $C^\infty_j(\Lat)\subset C^\infty(\Lat)$ denote the subspace of holomorphic functions satisfying $f(\mu\cdot \ell,\mu\cdot \ell')=\mu^{-j}f(\ell,\ell')$ for $\mu \in \C^\times$ acting as in~\eqref{eq:Latact}. Let $\mathcal{O}(\Lat)$ denote holomorphic functions on $\Lat$ that are meromorphic as $\ell\to \infty$ (or $\ell'\to \infty$) while keeping $\ell'$ (or $\ell$) fixed. Let $\mathcal{O}_j(\Lat) \subset C^\infty_j(\Lat)$ similarly denote the subspace of holomorphic functions that are meromorphic as $\ell\to \infty$ (or $\ell'\to \infty$) while keeping $\ell'$ (or $\ell$) fixed. 
\end{defn}

\begin{rmk} 
\emph{Modular forms} of weight $j$ are the elements of $\mathcal{O}_{j}(\Lat)$ that are invariant under the ${\rm SL}_2(\Z)$-action on $\Lat$. 
\end{rmk}

\begin{rmk}
We have the (strict) inclusions of algebras
\[
\bigoplus_{j\in \Z} C^\infty_j(\Lat)\hookrightarrow C^\infty(\Lat),\qquad \bigoplus_{j \in \Z} \mathcal{O}_j(\Lat)\hookrightarrow \mathcal{O}(\Lat),
\]
where the sources have a grading. The targets, however, do not carry a compatible grading. 
\end{rmk}

For $G$ a finite group, define $\CG$ to be the set of pairs of commuting elements of~$G$ so that $\CG \cong \Hom(\Z^2,G)$. There is an action of $G$ on $\CG$ by conjugation, $\zeta\colon (g,g')\mapsto (\zeta g_1\zeta^{-1},\zeta g_2\zeta^{-1})$. The set $\CG$ carries a left action by ${\rm SL}_2(\Z)$ given by 
\beq
(g,g')\mapsto (g^dg'^{-b},g^{-c}g'^a),\quad \left[\begin{array}{cc} a & b \\ c& d\end{array}\right]\in{\rm SL}_2(\Z).\label{eq:Gact}
\eeq

\begin{defn} \label{elldef}
For $X$ a manifold with an action of a finite group~$G$, the \emph{complexified equivariant elliptic cohomology of $X$} is 
\beq 
\Ell_G(X):= \bigoplus_{j \in \Z} \left(\prod_{(g,g')\in \CG}\left({\rm H}_{\rm dR}^{2j}(X^{\langle g,g'\rangle};\mathcal{O}_{j}(\Lat))\right)\right)^{G\times\SL_2(\Z)},\label{eq:Elldef}
\eeq
where $\SL_2(\Z)$ acts on $\CG$ as in~\eqref{eq:Gact}, $\langle g,g'\rangle$ is the subgroup of $G$ generated by $g$ and $g'$, $\Lat$ is defined in~\eqref{eq:Latact}, and $\zeta \in G$ acts by the component-wise diffeomorphism $X^{\langle g,g'\rangle}\to X^{\langle \zeta g\zeta^{-1},\zeta g'\zeta^{-1}\rangle}.$ 
\end{defn}

\begin{rmk} When $X=\pt$, $\Ell_G(\pt)$ can be interpreted as functions on the moduli stack of $G$-bundles on elliptic curves; see~\cite[\S2]{GanterHecke}. In fact, $\Ell_G(\pt)$ is the zeroth equivariant elliptic cohomology group of the point. Geometrically, the higher degree cohomology groups $\Ell^k_G(\pt)$ for $k\in \Z$ correspond to tensoring with powers of the Hodge bundle on the moduli stack of elliptic curves~\cite[Proposition~3.4]{BET19}. For $k\in \Z$, $\Ell_G^k(X)$ is 
\[
\Ell_G^k(X):= \bigoplus_{i-2j=k} \left(\prod_{(g,g')\in \CG}\left({\rm H}_{\rm dR}^i(X^{\langle g,g'\rangle};\mathcal{O}_{j}(\Lat))\right)\right)^{G\times\SL_2(\Z)}.
\]
\end{rmk}

\subsection{The geometric model}

In this subsection we review results relating moduli spaces of $2|1$-dimensional super tori to equivariant elliptic cohomology over~$\C$. This amounts to spelling out \Cref{defn:edinvariantmaps} in the case of the $2|1$-dimensional rigid conformal model geometry (compare to~\cite[\S4.2]{ST11}) and characterizing functions on $\mathcal{L}^{2|1}_0(X\sq G)$ in terms of elliptic cocycles.

We briefly review a standard description of the $S$-points of $\R^2\cong \C$. We have 
\[
\R^2(S):={\sf SMfld}(S,\R^2)\cong \{x,y\in C^\infty(S)^\ev \mid (x)_{\rm red}=\overline{(x)}_{\rm red},(y)_{\rm red}=\overline{(y)}_{\rm red}\},
\]
where the condition on functions is imposed on restriction to the reduced manifold of~$S$. Indeed, $C^\infty(S)$ does not have a real structure (see \cref{rmk:real}), and so this condition only makes sense on the restriction to $S_{\rm red}$. The diffeomorphism of manifolds~$\C\cong \R^2$ determines an isomorphism of supermanifolds. Setting $z=x+iy$ and $w=z-iy$, we find the description of $S$-points
\[
{\sf SMfld}(S,\R^2)\cong {\sf SMfld}(S,\C)\cong \{z,w\in C^\infty(S)^\ev\mid (z)_{\rm red}=\overline{(w)}_{\rm red}\}.
\]
In a standard abuse of notation, write $(z,\bar z)=(z,w)\in \C(S)$ to denote an $S$-point of $\C$, though we emphasize that (as there is no complex conjugation in $C^\infty(S)$) $z\in C^\infty(S)$ is only the conjugate of $\bar z\in C^\infty(S)$ on restriction to $C^\infty(S_{\rm red})$. Similarly, we use the notation
\[
(\mu,\bar\mu)\in \C^\times(S)\quad (\ell,\bar\ell,\ell',\bar\ell')\in \Lat(S)\subset (\C\times \C)(S). 
\]

\begin{defn}\label{defn:21model} Define the super Lie group $\bE^{2|1}$ with underlying supermanifold $\R^{2|1}$ and group structure given by the $S$-point formula
\[
(z,\bar z,\theta)\cdot (w,\bar w,\eta)=(z+w,\bar z+\bar w+\theta\eta,\theta+\eta),\quad (z,\bar z,\theta),(w, \bar w,\eta)\in \R^{2|1}(S). 
\]
Consider the action of $\C^\times$ on $\bE^{2|1}$ through homomorphisms given by
\[
(\mu,\bar\mu)\cdot (z,\bar z,\theta)=(\mu^2z,\bar \mu^2 \bar z,\bar\mu\theta),\quad (\mu,\bar\mu)\in \C^\times(S),\ (z,\bar z,\theta)\in \bE^{2|1}(S).
\]
Consider the semidirect product $\bE^{2|1}\rtimes \C^\times$. Define the \emph{rigid conformal model geometry} as the model space~$\R^{2|1}$ together with its (left) action by the rigid conformal transformation group~$\bE^{2|1}\rtimes \C^\times$. The composition 
\[
S\times \R^{2|1}\to S\times (\bE^{2|1}\rtimes \C^\times) \times \R^{2|1}\stackrel{\rm act}{\to} S\times \R^{2|1}
\]
is \emph{an $S$-family of rigid conformal maps}, where the first arrow is defined by an $S$-point of the rigid conformal transformation group,~$(z,\bar z,\theta,\mu,\bar\mu)\in (\bE^{2|1}\rtimes \C^\times)(S)$ and the second arrow is the action of $\bE^{2|1}\rtimes \C^\times$ on $\R^{2|1}$. 
\end{defn}

We will use the notation $\bE^2<\bE^{2|1}$ to denote the reduced subgroup of $\bE^{2|1}$, i.e., $\R^2$ with its usual additive structure. The following is a specialization of \Cref{defn:supertori}.

\begin{defn} For $(\ell,\bar\ell,\ell',\bar\ell')\in \Lat(S)$ an $S$-family of based lattices, an \emph{$S$-family of super tori} is a quotient of the form 
\[
T^{2|1}_{\ell,\ell'}:=(S\times \R^{2|1})/\Z^2
\]
for the action by $S\times \Z^2\subset S\times \bE^2\subset S\times (\bE^{2|1}\rtimes \C^\times)$ through fiberwise rigid conformal maps given by
\[
(n,m)\colon (z,\bar z,\theta)\mapsto (z+n\ell+m\ell',\bar z +n\bar\ell+m\bar\ell',\theta), \quad (n,m) \in \Z^2(S), \quad (z,\bar z,\theta)\in \R^{2|1}(S),
\]
determined by an $S$-family of homomorphisms $\langle (\ell,\bar\ell),(\ell',\bar\ell')\rangle \colon S\times \Z^2\to \bE^2\cong \C$ with generators $(\ell,\bar\ell)$ and $(\ell',\bar\ell')$ in $\C(S)\cong \bE^2(S)$ using the inclusion $\Lat\hookrightarrow \C\times \C$. The canonical cover
\[
S\times \R^{2|1}\to (S\times \R^{2|1})/\Z^2=T^{2|1}_{\ell,\ell'}
\]
endows the $S$-family of super tori $T^{2|1}_{\ell,\ell'}$ with a fiberwise rigid conformal structure.
\end{defn}

The following pair of definitions are special cases of \Cref{defn:stori} and \Cref{defn:edinvariantmaps}.

\begin{defn} For a stack $\X$, the \emph{super double loop stack of~$\X$}, denoted $\mathcal{L}^{2|1}(\X)$, is the stackification of the prestack whose objects over $S$ are given by pairs $(T^{2|1}_{\ell,\ell'},\phi)$, where $T^{2|1}_{\ell,\ell'}$ is a family of super tori and $\phi\colon T^{2|1}_{\ell,\ell'} \to \X$ is a map. Morphisms between these objects over~$S$ consist of triangles
\beq
\begin{tikzpicture}[baseline=(basepoint)];
\node (A) at (0,0) {$T^{2|1}_{\ell_1,\ell_1'}$};
\node (B) at (3,0) {$T^{2|1}_{\ell_2,\ell_2'}$};
\node (C) at (1.5,-1.5) {$\X$};
\node (D) at (1.5,-.6) {$\twocommute$};
\draw[->] (A) to node [above=1pt] {$\cong$} (B);
\draw[->] (A) to node [left=1pt]{$\phi$} (C);
\draw[->] (B) to node [right=1pt]{$\phi'$} (C);
\path (0,-.75) coordinate (basepoint);
\end{tikzpicture}\label{21triangle}
\eeq
that commute up to isomorphism, where the horizontal arrow is a fiberwise rigid conformal map between families of super tori.
\end{defn}

\begin{defn}
The \emph{stack of super double ghost loops}, denoted $\mathcal{L}_0^{2|1}(\X)$, is the full substack of $\mathcal{L}^{2|1}(\X)$ containing the objects $(T^{2|1}_{\ell,\ell'},\phi)$, where $\phi\colon T^{2|1}_{\ell,\ell'}\to \X$ is a map given by the composition 
\[
T^{2|1}_{\ell,\ell'} \stackrel{\sim}{\leftarrow} [S\times \R^{2|1}\sq \Z^2] \stackrel{\pi}{\to} [S\times\R^{0|1}\sq \Z^2]\to \X,
\] 
where $\pi$ is induced by the projection $\R^{2|1}\to \R^{0|1}$. 
\end{defn}

In this case, the atlas from \Cref{prop:atlas} has the form
\beq
\coprod_{(g,g')\in \CG} \Lat\times \Map(\R^{0|1},X^{\langle g,g'\rangle})\to \mathcal{L}^{2|1}_0(X\sq G),\label{eq:21atlas}
\eeq
and hence there is an injection of algebras
\beq\label{eq:inclusion2}
C^\infty(\cL^{2|1}_0(X\sq G)\hookrightarrow C^\infty\Big(\coprod_{(g,g')\in \CG} \Lat\times \Map(\R^{0|1},X^{\langle g,g'\rangle})\Big)
\eeq
We observe the isomorphism of Fr\'echet spaces (e.g., see \cite[Example 49]{HST})
\beq
&&C^\infty(\Lat\times \Map(\R^{0|1},X^{\langle g,g'\rangle}))\cong \Omega^\bullet(X^{\langle g,g'\rangle};C^\infty(\Lat)).\label{eq:algisos}
\eeq

\begin{rmk}
 Similar to the abuse of notation in \Cref{eq:rmkonell}, we will use $(\ell,\bar\ell,\ell',\bar\ell')=\id\in \Lat(\Lat)\hookrightarrow (\C\times \C)(\Lat)\subset C^\infty(\Lat)^{\times 4}$ to denote the identity map, which we identify with the functions on~$\Lat$ that are the restriction of the standard (complex) coordinate functions under the inclusion $\Lat\subset \C\times \C$. This is justified because $\id\colon S=\Lat\to \Lat$ is the universal $S$-point of~$\Lat$. 
\end{rmk}

Let $\vol \in C^\infty(\Lat)$ denote the real-valued function
\[
\vol(\ell,\bar\ell,\ell',\bar\ell')=\frac{\ell'\bar\ell-\ell\bar\ell'}{2i}
\] 
that reads off the volume of the (ordinary) torus associated with a based lattice. The above formula determines a natural transformation $\Lat\to C^\infty(-)$ from $S$-points of $\Lat$ to functions on $S$, thereby defining a function on $\Lat$. Define a map 
\beq
\Big( \bigoplus_{j \in 2\Z} \prod_{(g,g')\in \CG} \Omega^{j}_{\cl}(X^{\langle g,g'\rangle};C^\infty_{j/2}(\Lat))\Big)^{G\times \SL_2(\Z)}&\hookrightarrow&\bigoplus_{j\in \Z} \prod_{(g,g')\in G^{(2)}} \Omega^j( X^{\langle g,g'\rangle};C^\infty(\Lat))\nonumber\\
\omega&\mapsto& \vol^{\deg/2}\omega.\label{eq:inclimage2}
\eeq
\begin{prop}[{\cite[\S3]{DBE_EquivTMF}}] \label{prop:21present} The functions on the atlas that descend to the stack are determined by the commuting square
\[
\resizebox{\textwidth}{!}{
\xymatrix{
\displaystyle\Big( \bigoplus_{j \in 2\Z} \prod_{(g,g')\in \CG} \Omega^{j}_{\cl}(X^{\langle g,g'\rangle};C^\infty_{j/2}(\Lat))\Big)^{G\times \SL_2(\Z)} \ar@{-->}[r]^-{\cong} \ar@{^{(}->}[d] & C^\infty(\mathcal{L}^{2|1}_0(X\sq G)) \ar@{^{(}->}[d] \\
\displaystyle\bigoplus_{j\in \Z} \prod_{(g,g')\in G^{(2)}} \Omega^j( X^{\langle g,g'\rangle};C^\infty(\Lat))\ar[r]^-{\cong} & C^\infty\displaystyle\Big(\coprod_{(g,g')\in G^{(2)}} \Lat\times \Map(\R^{0|1},X^{\langle g,g'\rangle})\Big),
}
}
\]
where the inclusion on the left is the map~\eqref{eq:inclimage2}. 
\end{prop}

The cocycle model for equivariant elliptic cohomology requires that we restrict to a subalgebra $\mathcal{O}(\mathcal{L}^{2|1}_0(X\sq G))\subset C^\infty(\mathcal{L}^{2|1}_0(X\sq G))$ of appropriately holomorphic functions. Physically, this holomorphy is an expected consequence of the chiral supersymmetry; when $M=\pt$ and $G=\{e\}$, Stolz and Teichner prove that functions in the image of the restriction map~\eqref{eq:FTres} from $2|1$-dimensional field theories are indeed holomorphic~\cite[Theorem 1.15]{ST11}.  It turns out (though it is not obvious at this stage) that the geometric power operations also restricts to the subalgebra of holomorphic functions.

\begin{notation}
Below we use the notation 
\[
\prod_{(g,g')\in \CG} \Omega^{2\bullet}_{\cl}(X^{\langle g,g'\rangle};\mathcal{O}_{\bullet}(\Lat))=\bigoplus_{j\in \Z}\prod_{(g,g')\in \CG} \Omega^{2j}_{\cl}(X^{\langle g,g'\rangle};\mathcal{O}_{j}(\Lat))
\]
and similarly for coefficients in $C^{\infty}_\bullet(\Lat)$.
\end{notation}

\begin{defn} \label{defn:holo}
The algebra of \emph{holomorphic functions} on $\mathcal{L}^{2|1}_0(X\sq G)$, denoted $\mathcal{O}(\mathcal{L}^{2|1}_0(X\sq G))$, is the subalgebra of smooth functions that under the identification from~\Cref{prop:21present} lie in the subalgebra,
\[
\Big(\prod_{(g,g')\in \CG} \Omega^{2\bullet}_{\cl}(X^{\langle g,g'\rangle};\mathcal{O}_{\bullet}(\Lat))\Big)^{G\times \SL_2(\Z)}\subset \Big(\prod_{(g,g')\in \CG} \Omega^{2\bullet}_{\cl}(X^{\langle g,g'\rangle};C^\infty_{\bullet}(\Lat))\Big)^{G\times \SL_2(\Z)}
\]
of functions with holomorphic dependence on $\Lat$.
\end{defn}

This connects to equivariant elliptic cohomology by way of the following result. 

\begin{prop}[{\cite[\S3]{DBE_EquivTMF}}]
There is a natural isomorphism of algebras 
\beq
&&\mathcal{O}(\mathcal{L}^{2|1}_0(X\sq G))\cong \left(\prod_{(g,g')\in \CG} \Omega^{2\bullet}_{\cl}(X^{\langle g,g'\rangle};\mathcal{O}_{\bullet}(\Lat))\right)^{G\times \SL_2(\Z)}.\label{eq:funonstacksmooth}
\eeq
\label{prop:21presentsmooth}
\end{prop}

The following is an immediate consequence. 

\begin{cor}
The algebra $\mathcal{O}(\mathcal{L}^{2|1}_0(X\sq G))$ is a cocycle model for equivariant elliptic cohomology over~$\C$. 
\end{cor}

\subsection{Computing the total geometric power operation}

We proceed similarly to~\cref{sec:42}, specializing the input and output of $\widetilde{\coP}_n$ (\cref{def:atlascoop}) to this setting. Recall that $\uL = \Z^2$. In this case, an $S$-point of the atlas for $\cL_0^{2|1}(X^{\times n}\sq G\wr\Sigma_n)$, which is given by a triple $(\Lambda, h, \xi)$, is equivalent to the data of the triple
\[
((\ell,\ell'),(({\bf g},\sigma),({\bf g}',\sigma')),\xi),
\]
for $(\ell,\ell')\in \Lat(S)$, $(({\bf g},\sigma),({\bf g}',\sigma'))\in (G \wr \Sigma_n)^{(2)}$, $\xi\colon S\times \R^{0|1}\to X^{\langle({\bf g},\sigma),({\bf g}',\sigma')\rangle}$. 
This is the input data for the source of $\widetilde{\coP}_n$. 

Now we reinterpret the target of $\widetilde{\coP}_n$. Recall that
\[
\widetilde{\coP}_n(\Lambda, h, \xi) = (\Lambda\circ M_{\uL_k}, h_k, \xi_k)_{k\in K}
\]
and that this formula depends on a number of fixed choices. The matrix $M_{\uL_k} \colon \uL \to \uL$ was chosen so that $\det(M_{\uL_k}) > 0$ and so that $\ker(M_{\uL_k}) = \uL_k \subseteq \uL$. Further, we fixed $i_k \in I_k \subseteq \underline{n}$ and we may define ${\bf g}_k$ and ${\bf g}_{k}'$ as the image of the basis elements of $\uL = \Z^2$ in $G$ from the following diagram:
\[
\xymatrix{ 
\null & G  \\
\Z^2  \ar[r] \ar[ur]^{({\bf g}_{k},{\bf g}_{k}')} \ar[d]_{M_{\uL_k}} & G^{I_k} \ar[u]_{\pi_{i_k}} \ar[d]  \\ 
\Z^2 \ar[r]^-{h} & G\wr\Sigma_{I_k}.}
\]
Making use of this notation, we have
\beq
\widetilde{\coP}_n((\ell,\ell'),(({\bf g},\sigma),({\bf g}',\sigma')),\xi) = (M_{\uL_k}(\ell, \ell'), ({\bf g}_{k}, {\bf g}_k'), \xi_k)_{k \in K},\label{eq:copntarget21}
\eeq
where $M_{\uL_k}$ acts on $(\ell, \ell')$ by the formula in \eqref{eq:Latact}.

For $M\in M_{2\times 2}^{\mathrm{det} > 0}(\Z)$, we obtain a map $M\colon \Lat\to \Lat$ by restricting the action of $M_{2\times 2}^{\mathrm{det} > 0}(\Z)$ on $\C^2$ to $\Lat\subset \C^2$. Furthermore, this map is $\C^\times$-equivariant (being linear) and so restricts to an action on the subspaces $\mathcal{O}_{\bullet}(\Lat)\subset \mathcal{O}(\Lat)$. For $\omega\in \prod \Omega^{2\bullet}(X^{\langle g,g'\rangle};\mathcal{O}_{\bullet}(\Lat))$, let $M^*\omega$ denote the pullback of $\omega$ along this induced action on coefficients.

\begin{notation}\label{notation:omegag2}
For $\omega\in \prod_{(g,g')\in G^{(2)}} \Omega^\bullet( X^{\langle g,g'\rangle};C^\infty(\Lat))$, let $\omega_{g,g'}\in \Omega^\bullet(X^{\langle g,g'\rangle};C^\infty(\Lat))$ be the component of $\omega$ indexed by the factor $(g,g')$.
\end{notation}

Using \cref{prop:21presentsmooth}, the geometric power operation determines the map on differential forms
\beq
\begin{array}{l}
\Big( \prod_{(g,g')\in \CG} \Omega^{2\bullet}_{\cl}(X^{\langle g,g'\rangle};\mathcal{O}_{\bullet}(\Lat))\Big)^{G\times \SL_2(\Z)} \\
\phantom{BLAH}\cong \mathcal{O}(\mathcal{L}^{2|1}_0(X\sq G))\stackrel{\P_n}{\longrightarrow}\mathcal{O}(\mathcal{L}^{1|1}_0(X^{\times n}\sq G\wr\Sigma_n))\\
\phantom{BLAH}\cong \Big( \prod_{(({\bf g},\sigma),({\bf g}',\sigma'))\in (G\wr \Sigma_n)^{(2)}} \Omega^{2\bullet}_{\cl}((X^{\times n})^{\langle ({\bf g},\sigma),({\bf g}',\sigma')\rangle};\mathcal{O}_{\bullet}(\Lat))\Big)^{G\times \SL_2(\Z)}.
\end{array}
\label{needanumber}
\eeq
In the following proposition, for $\omega\in ( \prod_{(g,g')\in \CG} \Omega^{2\bullet}_{\cl}(X^{\langle g,g'\rangle};\mathcal{O}_{\bullet}(\Lat)))^{G\times \SL_2(\Z)}$, we denote the image under the above composition by $\P_n(\omega)$, with the isomorphisms to differential forms implicitly understood. We refer to~\Cref{notation:deg} for the definition of the map $\deg$ used below.

\begin{thm}\label{prop:21formula}
Suppose $X$ is a smooth and compact $G$-manifold and let $n \ge 1$. The geometric power operation is characterized by the formula
\begin{align*}
\P_n(\omega)_{({\bf g},\sigma),({\bf g}',\sigma')} & =\prod_k  \det(M_{\uL_k})^{\deg/2} M_{\uL_k}^*(\omega_{{\bf g}_{k},{\bf g}_{k}'}) \\
& \in \Omega^{2\bullet}_\cl(\prod_k X^{\langle {\bf g}_k,{\bf g}_k'\rangle};\mathcal{O}_{\bullet}(\Lat))\cong \Omega^{2\bullet}_\cl((X^{\times n})^{\langle ({\bf g},\sigma),({\bf g}',\sigma')\rangle};\mathcal{O}_{\bullet}(\Lat)),
\end{align*}
for 
$(({\bf g},\sigma),({\bf g}',\sigma')) \in (G\wr\Sigma_n)^{(2)}$, and $({\bf g}_{k},{\bf g}_{k}')\in \CG$ using the notation described above. 
\end{thm}

\begin{proof}
By \Cref{prop:21present}, functions on $\cL_0^{2|1}(X\sq G)$ are in bijection with functions on the atlas of the form $(\vol^{\deg/2}\omega_{g,g'})_{(g,g')\in \CG}$, where $\omega_{g,g'}\in\Omega^{2\bullet}(X^{\langle g,g'\rangle};\mathcal{O}_\bullet(\Lat))$. The restriction of the power cooperation to the component of the atlas~\eqref{eq:21atlas} indexed by $(({\bf g},\sigma),({\bf g}',\sigma'))\in (G\wr \Sigma_n)^{(2)}$ is 
\beq
&&\Lat\times \Map(\R^{0|1},(X^{\times n})^{\langle ({\bf g},\sigma),({\bf g}',\sigma')\rangle}) \xrightarrow{\widetilde{\coP}_n} \coprod_k \Lat\times \Map(\R^{0|1},X^{\langle {\bf g}_k,{\bf g}'_k\rangle}). \label{eq:coPnmap21}
\eeq
Restricting to the component indexed by $({\bf g}_k,{\bf g}_k')\in \CG$ and pulling back along $\widetilde{\coP}_n$ gives a map of algebras
\beq
&&\Omega^{\bullet}(X^{\langle {\bf g}_k,{\bf g}'_k\rangle};C^\infty(\Lat))\cong C^\infty(\Lat\times \Map(\R^{0|1},X^{\langle {\bf g}_k,{\bf g}'_k\rangle}))\nonumber\\
&&\to C^\infty(\Lat\times \Map(\R^{0|1},(X^n)^{(({\bf g},\sigma),({\bf g}',\sigma'))}))\cong \Omega^\bullet((X^n)^{(({\bf g},\sigma),({\bf g}',\sigma'))};C^\infty(\Lat)),\nonumber
\eeq 
where we have pre- and post-composed with the isomorphism between $C^\infty(\Lat)$-valued differential forms and functions on the atlas from~\eqref{eq:algisos}. Next, we apply each of these maps to $\vol^{\deg/2}\omega_{{\bf g}_k,{\bf g}_k'}$ and take their product over $k$, obtaining
\beq
\widetilde{\coP}_n^*(\vol^{\deg/2}\omega_{{\bf g}_k,{\bf g}'_k})_{k\in K}&=&\prod_k (\det(M_{\uL_k})\vol)^{\deg/2}M_{\uL_k}^*\omega_{{\bf g}_k,{\bf g}'_k}\nonumber\\
&&\in\Omega^{2\bullet}_\cl(\prod_k X^{\langle {\bf g}_k,{\bf g}_k'\rangle};C^\infty_\bullet(\Lat))\cong \Omega^{2\bullet}((X^{\times n})^{({\bf g},\sigma)};C^\infty_\bullet(\Lat)).\nonumber
\eeq
By \cref{thm:mainthm}, the function $\widetilde{\coP}_n^*(\vol^{\deg/2}\omega_{({\bf g}_k,\sigma),({\bf g}'_k,\sigma')})_{k\in K}$ on the atlas descends to a $C^\infty$-function on the stack $\cL_0^{2|1}((X^{\times n})\sq G\wr \Sigma_n)$ and computes the geometric power operation. We identify this map between functions on stacks with the map of differential forms~\eqref{needanumber}, though so far only for coefficients in $C^\infty_\bullet(\Lat)$ (rather than $\mathcal{O}_\bullet(\Lat)$). Using~\eqref{eq:inclimage2}, this has the effect of removing the factors of $\vol$, obtaining the claimed formula in the statement of the theorem. 

Finally, if the input differential form $\omega$ has coefficients in $\mathcal{O}_\bullet(\Lat)\subset C^\infty_\bullet(\Lat)$, we observe that the output $\P_n(\omega)$ also has holomorphic coeffcients. More explicitly, each factor~$\det(M_{\uL_k})^{\deg/2}M_{\uL_k}^*\omega_{{\bf g}_k,{\bf g}'_k}$ has coefficients in $\mathcal{O}_\bullet(\Lat)$ (since pulling back along the action of $2\times 2$ matrices and multiplication by a scalar $\det(M_{\uL_k})$ preserves the subalgebra of holomorphic functions), and so the product does as well. This completes the proof. \ep

\subsection{Adams operations} 

Fix a natural number $n \ge 1$.
We compute the effect of the $n$th geometric Adams operation on differential forms in terms of the composition,
\beq
\begin{array}{l}
\Big( \prod_{(g,g')\in \CG} \Omega^{2\bullet}_{\cl}(X^{\langle g,g'\rangle};\mathcal{O}_{\bullet}(\Lat))\Big)^{G\times \SL_2(\Z)} \\
\phantom{BLAH}\cong \mathcal{O}(\mathcal{L}^{2|1}_0(X\sq G))\stackrel{\Psi_n}{\longrightarrow}\mathcal{O}(\mathcal{L}^{1|1}_0(X\sq G))\label{eq:21Adams}\\
\phantom{BLAH}\cong \Big( \prod_{(g,g')\in \CG} \Omega^{2\bullet}_{\cl}(X^{\langle g,g'\rangle};\mathcal{O}_{\bullet}(\Lat))\Big)^{G\times \SL_2(\Z)}.
\end{array}
\eeq
Below we use the same notation $\Psi_n$ to denote the map on differential forms given by the composition above. We give a formula for $\Psi_n$ using \cref{notation:omegag2}.

\begin{thm}\label{thm:TMFAdams}
The $n^{\rm th}$ geometric Adams operation~\eqref{eq:21Adams}
is given by the formula
\beq
\omega_{g,g'}\mapsto n^{\deg/2} \omega_{g^n,g'^n}\qquad \omega_{g,g'}\in \Omega^{2\bullet}_{\rm cl}(X^{\langle g,g'\rangle};\mathcal{O}_{\bullet}(\Lat))\label{eq:EllAdams}.
\eeq
\end{thm}

\begin{proof}
By~\Cref{prop:adamsmain}, the geometric Adams operation is determined by the maps
\beq
\Lat \times \Map(\R^{0|1},X^{\langle g,g'\rangle})\stackrel{\widetilde{\Psi}_n}{\longrightarrow} \Lat \times \Map(\R^{0|1},X^{\langle g^n,g'^n\rangle})\label{eq:21adamscomp}
\eeq
for each pair $(g,g') \in G^{(2)}$. On $S$-points, we have
\[
\widetilde{\Psi}_n(\Lambda,(g,g'),\xi)=(n\Lambda,(g^n,g'^n),\xi'),
\]
where $\xi'$ is the composition
\[
S\times \R^{0|1}\stackrel{\xi}{\to} X^{\langle g,g'\rangle}\hookrightarrow X^{\langle g^n,g'^n\rangle }.
\]
We observe that 
\[
n\Lambda=M\Lambda,\qquad M=\left[\begin{array}{cc} n & 0 \\ 0 & n \end{array}\right].
\]
Hence, pulling back along~\eqref{eq:21adamscomp}, we find 
\beq
\widetilde{\Psi}_n^*(\vol^{\deg/2}\omega_{g,g'})&=&(n^2\vol)^{\deg/2}M^*\omega_{g^n,g'^n}\nonumber\\ 
&=& n^{\deg/2} \vol^{\deg/2}\omega_{g^n,g'^n}\in \Omega^{2\bullet}_\cl(X^{g^n,g'^n};C^\infty(\Lat)),\label{eq:emphasizethis}
\eeq
where we have used that $M^*\omega_{g^n,g'^n}=n^{-\deg/2}\omega_{g^n,g'^n}$ for $\omega_{g^n,g'^n} \in \Omega^{2\bullet}(X^{\langle g^n,g'^n\rangle};\mathcal{O}_{\bullet}(\Lat))$. This in turn follows from the fact that for $F\in \mathcal{O}_k(\Lat)$, $F(n\Lambda)=n^{-k}F(\Lambda)$. Using~\Cref{prop:21present} to identify functions on the atlas with differential forms, the resulting assignment is the claimed formula on cocycles for the Adams operation. 
\ep
\begin{rmk}
We emphasize a somewhat miraculous cancellation that occurs above: the factors $n^{\deg}$ and $n^{-\deg/2}$ in~\eqref{eq:emphasizethis} combine to give the correct total factor of~$n^{\deg/2}$ for the Adams operation. This cancellation depends critically on how volumes of certain covering spaces of tori are related to the volume of their base, as well as on the somewhat subtle definition of holomorphy from~\cite{DBE_EquivTMF}, see~\Cref{defn:holo} above.
\end{rmk}

\section{Comparison with operations in $\E$-theory} \label{Sec:Ethy}

Morava $\E$-theory is a fundamental object in chromatic homotopy theory. The essentially unique $\mathbb{E}_{\infty}$-ring structure on $\E$-theory, constructed by Goerss, Hopkins, and Miller in \cite{ghobstructiontheory}, gives rise to power operations and Adams operations \cite{andoisogenies}. Hopkins, Kuhn, and Ravenel \cite{hkr} constructed a character map for $\E$-theory, which approximates the Borel equivariant $\E$-cohomology of a finite $G$-CW complex $X$ by the rational cohomology of a $G$-CW complex built out of $X$. The relationship between power operations and character theory was explored by the first and third author in \cite{charpo}. The formulas in \cite{charpo} motivated several of the questions asked in this paper. In this section we review the relationship between power operations and character theory for Morava $\E$-theory and explain how this relates to the formulas of the previous sections.

\subsection{Character theory for Morava $\E$-theory}

In this subsection we give a brief introduction to character theory for Morava $\E$-theory. The Stolz--Teichner program suggests a close relationship between geometric $d$-dimensional field theories and certain height $d$ cohomology theories. Dimensional reduction for field theories approximates a $d$-dimensional field theory by a $0$-dimensional field theory. Character theory for height $d$ Morava $\E$-theory can be viewed as an analogue of dimensional reduction. For a more thorough introduction, see Appendix A.2 of \cite{peterson_book}, which was written by the third author.

Recall that \emph{Morava $\E$-theory} $\E_d$ of height $d$ is the Landweber exact ring spectrum associated to the universal deformation $\G$ of a height $d$ formal group law $\G_0$ over a perfect field $\kappa$ of characteristic $p$. For concreteness, we may choose $\G_0$ to be the Honda formal group law over $\kappa = \F_{p^d}$, which yields the coefficient ring 
\[
\E_d^* = \pi_{-*}\E_d \cong \W\F_{p^d}\llbracket u_1,\ldots,u_{d-1}\rrbracket[u^{\pm 1}],
\]
where $\W\F_{p^d}$ denotes the ring of Witt vectors on $\F_{p^d}$, the power series generators $u_i$ are in degree $0$, and $u$ is of degree $-2$. When $d=1$ and $\kappa = \F_p$, the universal deformation $\G$ is the multiplicative formal group law over $W(\F_p) \cong \Z_p$, $\E_{1}^{*} \cong \Z_p[u^{\pm 1}]$, and $\E_1 = \K^{\wedge}_{p}$ is $p$-adic $\K$-theory.

An analogue of dimensional reduction for Morava $\E$-theory is given by Hopkins--Kuhn--Ravenel character theory \cite{hkr}, which we shall now recall. Let $\L = \Z_p^d$. For any finite group $G$, the set $\Hom(\L,G)$ of commuting $p$-power order elements of $G$ admits a conjugation action by $G$ and we write $\Hom(\L,G)_{/\sim}$ for the corresponding quotient. 

Hopkins, Kuhn, and Ravenel construct a faithfully flat extension $C_0$ of $\Q \otimes \E_d^0$ that is used to construct a $2$-periodic rational cohomology theory $HC_0$ given by
\[
HC_0^*(Y) = C_0 \otimes_{\Q \otimes \E_d^*} (\Q \otimes \E_d)^*(Y) 
\]
for any space $Y$ equivalent to a finite CW-complex. Note that $C_0$ depends on $d$, although this is suppressed from the notation. This cohomology theory is just $2$-periodic singular cohomology with coefficients in $C_0$. Given a finite $G$-CW complex $X$, they then construct a \emph{(generalized) character map}
\[
\chi_d\colon \E_d^*(EG \times_G X) \longrightarrow \left (\Prod{[h] \in \Hom(\L,G)_{/\sim}}HC_0^*(X^{\im h}) \right )^G,
\]
generalizing the work of \cite{Atiyah_Segal} from complex $\K$-theory to Morava $\E$-theory.
There is an action of $\Aut(\L)$ on $C_0$ with fixed points given by $C_{0}^{\Aut(\L)} \cong \Q \otimes \E_{d}^0$. Combining this action with the action on $\Hom(\L,G)_{/\sim}$ by precomposition, they show that their character map lands in the fixed points
\[
\left (\Prod{[h] \in \Hom(\L,G)_{/\sim}}HC_0^*(X^{\im h}) \right )^{G \times \Aut(\L)}.
\]
This should be compared with \cref{prop:Kthycocycle} and \cref{elldef}.

\begin{thm}[Hopkins--Kuhn--Ravenel~\cite{hkr}] \label{charactermap}
For any finite $G$-CW complex $X$, the character map $\chi_d$ induces an isomorphism after base change to $C_0$:
\[
C_0 \otimes_{\E_{d}^{0}} \E_d^*(EG \times_G X) \overset{\simeq}{\longrightarrow} \left (\Prod{[h] \in \Hom(\L,G)_{/\sim}}HC_0^*(X^{\im h}) \right )^G.
\]
Moreover, $\chi_d$ is equivariant with respect to the natural $\Aut(\L)$-action on both sides (on the left hand side this action is just on $C_0$), and descends to an isomorphism of Borel equivariant cohomology theories
\[
\Q \otimes \E_d^*(EG \times_G X) \overset{\simeq}{\longrightarrow} \left (\Prod{[h] \in \Hom(\L,G)_{/\sim}}HC_0^*(X^{\im h}) \right )^{G \times \Aut(\L)}
\]
after taking $\Aut(\L)$-fixed points.
\end{thm}

\begin{ex}
We now specialize the Hopkins--Kuhn--Ravenel character map to the case $X = \pt$, the case of interest to us in the next subsection. In this case, the target of $\chi_d$ can be identified with
\[
\Cl_d(G,C_0) := \mathrm{Map}(\Hom(\L,G)_{/\sim},C_0),
\]
the ring of $C_0$-valued functions on the set $\Hom(\L,G)_{/\sim}$. We refer to the elements of $\Cl_d(G,C_0)$ as \emph{(generalized) class functions}.
\end{ex}

\subsection{The character of the total power operation for Morava $\E$-theory}

By the Goerss--Hopkins--Miller theorem \cite{ghobstructiontheory}, the ring spectrum $\E_d$ admits an $\mathbb{E}_{\infty}$-ring structure, which is unique up to a contractible space of choices. Consequently, there is a unique theory of power operations for $\E_d$. These are, for any space $X$, multiplicative non-additive maps
\[
\P_n\colon \E_d^0(X) \longrightarrow \E_d^0(E\Sigma_n \times_{\Sigma_n} X^{\times n}),
\]
defined by sending a class $[X \to \E_d]$ to the class $[E\Sigma_n \times_{\Sigma_n} X^{\times n} \to (E\Sigma_n)_{+} \wedge_{\Sigma_n} \E_d^{\wedge n} \to \E_d]$ using the $\mathbb{E}_{\infty}$-ring structure maps for $\E_d$. 

The $\Aut(\L)$-action on $C^0$ by $\E_d^0$-algebra maps that plays a role in \cref{charactermap} extends to an action by ring maps of the monoid $\End_{\text{fin}}(\L)$ of finite index endomorphism of $\L$. Let $\Lat_{\text{fin}}(\L)$ be the set of finite index sublattices of $\L$. There is a canonical surjection $\End_{\text{fin}}(\L) \twoheadrightarrow \Lat_{\text{fin}}(\L)$ sending a finite index endomorphism to its image. 

Given a conjugacy class $[h \colon \L \to G \wr \Sigma_n]$, we have an associated $\L$-set 
\[
\underbar{n} = \Coprod{k} I_k
\]
as in the discussion at the beginning of \Cref{eq:computecoop}. There we observe that we can extract well-defined conjugacy classes $[h_k \colon \L_k \to G]$ corresponding to the transitive components of $\underbar{n}$. Given a section $\phi$ of $\End_{\text{fin}}(\L) \twoheadrightarrow \Lat_{\text{fin}}(\L)$, the map $\phi_{\L_k}= \phi(\L_k) \colon \L \to \L$ has image $\L_k$. Thus there is an induced isomorphism $\psi_{\L_k} \colon \L \lra{\cong} \L_k$.

\begin{defn}
Let $\phi$ be a section to the canonical map $\End_{\text{fin}}(\L) \to \Lat_{\text{fin}}(\L)$. We define the \emph{pseudo-power operation} associated to $\phi$ to be the natural map
\[
\P_n^{\phi}\colon \Cl_d(G,C_0) \longrightarrow \Cl_d(G \wr \Sigma_n,C_0)
\]
that sends a class function $f \in \Cl_d(G,C_0)$ to the class function on $G \wr \Sigma_n$ given by 
\[
\P_{n}^{\phi}(f)([h \colon \L \to G \wr \Sigma_m]) = \prod_k \phi_{\L_k}f([\psi_{\L_k}^{*}h_k ]).
\]
\end{defn}
The formula above should be compared with \cref{prop:11formula} and \cref{prop:21formula}.

These operations satisfy a number of compatibility properties similar to the total power operations, justifying the terminology pseudo-power operation. The main result of \cite{charpo} is the following: 

\begin{thm}\label{thm:charpo}
For any section $\phi$ as above, there is a commutative diagram:
\[
\xymatrix{\E_d^0(BG) \ar[r]^-{\chi_d} \ar[d]_{\P_n} & \Cl_d(G,C_0) \ar[d]_{\P_n^{\phi}} & \Q\otimes \E_d^0(BG) \ar[d]^{\P_n^{\Q}} \ar[l] \\
\E_d^0(BG \wr \Sigma_n) \ar[r]_-{\chi_d} & \Cl_d(G \wr \Sigma_n, C_0) & \Q\otimes \E_d^0(BG\wr \Sigma_n). \ar[l]}
\]
The right square is induced by taking $\Aut(\L)$-fixed points of $\P_n^{\phi}$. The resulting map $\P_n^{\Q}$ is independent of $\phi$ and has the structure of a global power functor in the sense of \cite{globalgreen}. 
\end{thm}

This should be compared with \cref{thm:main11}.

\subsection{Adams operations for Morava $\E$-theory}
The Adams operations in $\E$-theory were first defined by Ando in \cite{andoisogenies}. To define them we must review an important result concerning the $\E$-cohomology of symmetric groups.

\begin{thm}[Strickland~\cite{etheorysym}] \label{subgroups}
Let $k \ge 0$. There is a canonical isomorphism 
\[
\E_{d}^0(B\Sigma_{p^k})/I_{tr} \cong \Gamma \Sub_{p^k}(\G),
\]
where $I_{tr}$ is the image of the transfer map in $\E$-cohomology along the inclusion $\Sigma_{p^{k-1}}^{\times p} \subset \Sigma_{p^k}$ and $\Gamma\Sub_{p^k}(\G)$ is the ring of functions on the scheme $\Sub_{p^k}(\G)$ that classifies subgroup schemes of order $p^k$ in $\G$.
\end{thm}

It turns out that $\E_{d}^0(B\Sigma_{p^k})$ is a free $\E_{d}^0$-module of finite rank. Thus 
\[
\E_{d}^0(B\Sigma_{p^k} \times X) \cong \E_{d}^0(B\Sigma_{p^k}) \otimes_{\E_{d}^0} \E_{d}^0(X)
\]
for all spaces $X$. Restriction along the diagonal $B\Sigma_{p^k} \times X \to E\Sigma_{p^k} \times_{\Sigma_{p^k}} X^{\times {p^k}}$, gives a power operation
\[
P_{p^k} \colon \E_{d}^0(X) \xrightarrow{\P_{p^k}} \E_d^0(E\Sigma_{p^k} \times_{\Sigma_{p^k}} X^{\times {p^k}}) \to \E_{d}^0(B\Sigma_{p^k}) \otimes_{\E_{d}^0} \E_{d}^0(X).
\] 

There are canonical subgroup schemes of $\G$ given by the $p^k$-torsion $\G[p^k]$. The order of $\G[p^k]$ as a group scheme is $p^{kd}$, thus \cref{subgroups} implies that there is a map of $\E^0$-algebras $\E^0(B\Sigma_{p^{kd}})/I_{tr} \to \E^0$ classifying the subgroup scheme $\G[p^k] \subset \G$. Applying \cref{subgroups}, we may form the ring endomorphism $\psi^{p^k}$ of $\E_{d}^0(X)$ given by the composite
\[
\psi^{p^k}\colon \E_{d}^0(X) \lra{P_{p^{kd}}} \E_{d}^0(B\Sigma_{p^{kd}}) \otimes_{\E_{d}^0} \E_{d}^0(X) \to \E_{d}^0(B\Sigma_{p^{kd}})/I_{tr} \otimes_{\E_{d}^0} \E_{d}^0(X) \lra{\G[p^k] \otimes 1} \E_{d}^0(X). 
\]
Applying \Cref{thm:charpo} to $\psi^{p^k}$ when $X = BG$ for $G$ a finite group gives a formula for $\psi^{p^k}$ on the level of class functions. Given $f \in \Cl_d(G,C_0)$ and a conjugacy class $[\L \rightarrow G]$ represented by a $d$-tuple of commuting elements $(g_1, \ldots, g_d)$, we find that
\[
\psi^{p^k}(f)((g_1,\ldots, g_d)) = f((g_{1}^{p^k},\ldots, g_{d}^{p^k})).
\]
This formula should be compared with \cref{thm:KAdams} and \cref{thm:TMFAdams}.

\appendix

\section{A brief review of super geometry and stacks}\label{append:A}

\subsection{Supermanifolds}\label{sec:superman}\label{appen:super} 

Recall that a super commutative algebra $A$ is a $\Z/2$-graded algebra such that $ab = (-1)^{|a||b|}ba$ for any two homogeneous elements $a,b \in A$.

\begin{defn} 
Define the supermanifold~$\R^{k|l}$ to be the locally ringed space with underlying topological space $\R^k$ and structure sheaf of super commutative $\C$-algebras defined by $U\mapsto C^\infty(U)\otimes_\C\Lambda^\bullet(\C^l)$, where $C^\infty(U)$ is $\C$-valued smooth functions on an open subset $U\subset \R^k$ and $\Lambda^\bullet(\C^l)$ is the $\Z/2$-graded exterior algebra on $\C^l$. 
\end{defn}

\begin{defn}
A \emph{$k|l$-dimensional supermanifold} is a locally ringed space whose underlying space is second countable, Hausdorff, and locally isomorphic to $\R^{k|l}$. Supermanifolds and maps between them (as locally ringed spaces) form a category denoted~${\sf SMfld}$. 
\end{defn}

\begin{rmk} 
The above flavor of supermanifolds are called \emph{$cs$-manifolds} in~\cite{DM}, and differ slightly from another common definition of supermanifolds with structure sheaves defined over~$\R$. 
\end{rmk}

There is a fully faithful embedding of the category of ordinary smooth manifolds and smooth maps into ${\sf SMfld}$ that on objects regards a manifold with its sheaf of complex-valued functions as a supermanifold. There is also a \emph{reduction} functor from ${\sf SMfld}$ to ordinary manifolds: let $N_{\rm red}$ denote the manifold built from~$N$ by taking the quotient of the structure sheaf by its nilpotent ideal. If we then regard the ordinary manifold $N_{\rm red}$ as a supermanifold, there is an evident monomorphism~$N_{\rm red}\hookrightarrow N$ in ${\sf SMfld}$.

 Following the usual notation, we write $C^\infty(N)$ for the global sections of the structure sheaf of a supermanifold~$N$. We observe (by a standard partition of unity argument) that supermanifolds are affine, meaning that a map of supermanifolds $N\to N'$ is determined by a map of super commutative algebras $C^\infty(N')\to C^\infty(N)$. We use the notation $C^\infty(N)^{\ev}\oplus C^\infty(N)^{\odd}\cong C^\infty(N)$ to denote the direct sum decomposition into even and odd functions on $N$.

\begin{ex}\label{ex:Batch} 
Let $E$ be a complex vector bundle over an ordinary manifold~$M$. Then $\Pi E:=(M,\Lambda^\bullet E^\vee)$ is a supermanifold with $(\Pi E)_{\rm red}=M$ and $C^\infty(\Pi E)=\Gamma(M,\Lambda^\bullet E^\vee)$. By Batchelor's Theorem~\cite{batchelor}, any supermanifold~$N$ is isomorphic to~$\Pi E$ for some complex vector bundle over an ordinary manifold. 
\end{ex}

\begin{rmk}\label{rmk:real}
Recall that a \emph{real structure} on a complex vector space is a $\C$-antilinear involution. The sheaf of complex-valued functions on a smooth manifold has a real structure given by complex conjugation of smooth functions. Functions on supermanifolds typically do not have a real structure: the involution on $C^\infty(S_{\rm red})$ need not have an extension to $C^\infty(S)$. Indeed, in the previous example such an extension would be the data of a real structure on the vector bundle~$E$. 
\end{rmk}

We will frequently use the functor of points to study the category of supermanifolds. This means that we will identify a supermanifold~$N$ with the presheaf on the category of supermanifolds given by $S\mapsto {\sf SMfld}(S,N)$. We call ${\sf SMfld}(S,N)$ the set of $S$-points of $N$, and often denote this set by $N(S)$. By the Yoneda lemma, maps between such presheaves are in bijection with maps between supermanifolds.

\begin{ex} 
We can describe $\R^{k|l}$ in terms of its $S$-points for a test supermanifold~$S$. We have
\beq
&&{\sf SMfld}(S,\R^{k|l})\cong \{x_1,\dots,x_k\in C^\infty(S)^\ev,\ \theta_1,\dots,\theta_l\in C^\infty(S)^\odd\mid (x_i)_{\rm red}=\overline{(x_i)}_{\rm red}\},\label{RnmSpot}
\eeq
where we emphasize that the condition $(x_i)_{\rm red}=\overline{(x_i)}_{\rm red}$ on the functions $x_i$ is only on restriction to the reduced manifold, $S_{\rm red}\hookrightarrow S$.
\end{ex}

\subsection{(Super) Lie groupoids and (super) stacks}\label{superstacks}

The purpose of this subsection is to give a brief introduction to super Lie groupoids and super stacks. Our main reference is~\cite[\S7]{HKST}. For the correspondence between differentiable stacks and Lie groupoids, we refer to \cite{blottiere}, \cite[Section 2.6]{behrendxu_stacks}, and~\cite[Section 2]{schommerpries_centralext} for further details. 

\begin{defn}
A \emph{super Lie groupoid} $\mathcal{G}=\{\mathcal{G}_1\rightrightarrows \mathcal{G}_0\}$ consists of a supermanifold of objects $\mathcal{G}_0$, a supermanifold of morphisms $\mathcal{G}_1$, source and target maps $s,t\colon\mathcal{G}_1\to \mathcal{G}_0$ that are required to be submersions, a unit map $\mathcal{G}_0\to \mathcal{G}_1$ and a composition map $\mathcal{G}_1\times_{\mathcal{G}_0}\mathcal{G}_1\to \mathcal{G}_1$. These data are required to satisfy the axioms of a groupoid object. A \emph{functor} $\mathcal{G}\to \mathcal{H}$ is the data of maps of supermanifolds $\mathcal{G}_i\to \mathcal{H}_i$ for $i=0,1$ satisfying the axioms of a functor. A \emph{natural transformation} is the data of a map of supermanifolds $\mathcal{G}_0\to\mathcal{H}_1$ satisfying the axioms of a natural transformation. We will often drop the modifier ``super" when discussing super Lie groupoids. Let ${\sf Grpd}$ denote the 2-category whose objects are Lie groupoids, 1-morphisms are functors between Lie groupoids, and 2-morphisms are natural transformations between functors. 
\end{defn}

\begin{ex}\label{ex:actgrpd} Given an action of a super Lie group $G$ on a supermanifold $N$, the \emph{quotient groupoid} $N\sq G$ has objects~$N$ and morphisms $G\times N$. The source map is the projection and the target map is the action map. The unit is the inclusion along the identity of~$G$ and composition is determined by multiplication in~$G$. 
\end{ex}

\begin{defn}
A \emph{super stack} is a category fibered in groupoids over supermanifolds satisfying descent with respect to surjective submersions of supermanifolds. We will often drop the modifier ``super" when discussing super stacks. Let ${\Stack}$ denote the 2-category of super stacks, fibered functors, and fibered natural transformations. 
\end{defn}

We recall that any stack $\X$ defines a lax 2-functor from the opposite category of supermanifolds to (the 2-category of) groupoids. On objects, this functor assigns to a supermanifold~$S$ the fiber of~$\X$ at~$S$. The Grothendieck construction gives an equivalence between the 2-category of stacks viewed as fibered categories and the 2-category of stacks viewed as lax 2-functors. We will freely pass between these equivalent points of view on stacks.

\begin{defn}[{\cite[Definition~7.21]{HKST}}]\label{defn:symmonstack}
A \emph{symmetric monoidal} category fibered over supermanifolds is a fibered category ${\sf C}\to {\sf SMfld}$ together with fibered functors $\otimes\colon {\sf C}\times_{\sf SMfld} {\sf C}\to {\sf C}$ and $1\colon {\sf SMfld}\to{\sf C}$ and the obvious fibered natural transformations as per the standard definition of a symmetric monoidal category. A \emph{symmetric monoidal stack} is a fibered symmetric monoidal category satisfying descent in the symmetric monoidal sense: the groupoids in the descent diagram have symmetric monoidal structures, and we require the equivalence to be an equivalence of symmetric monoidal categories.
\end{defn}

Recall (e.g.,~\cite[\S7]{HKST}) that there is a \emph{stackification} functor from the 2-category of groupoid-valued presheaves on supermanifolds to super stacks, which is characterized as the left adjoint to the canonical inclusion.

\begin{ex}\label{ex:liestackcomparison} 
Given a Lie groupoid $\mathcal{G}=\{\mathcal{G}_1\rightrightarrows \mathcal{G}_0\}$, we obtain a presheaf of groupoids whose value on a supermanifold $S$ is $\mathcal{G}(S)=\{\mathcal{G}_1(S)\rightrightarrows \mathcal{G}_0(S)\}$. In fact, this extends to a 2-functor from the 2-category ${\sf Grpd}$ of Lie groupoids to the 2-category of (lax) presheaves of groupoids on supermanifolds. Postcomposition with stackification then induces a functor
\beq
[-]\colon {\sf Grpd}\to {\Stack}.\label{eq:bundlization1}
\eeq
We use the notation~$\mathcal{G}\mapsto [\mathcal{G}]$ and $\{f\colon \mathcal{G}\to\mathcal{G}'\}\mapsto \{[f]\colon [\mathcal{G}]\to[\mathcal{G}']\}$ to denote the images of objects and 1-morphisms under this functor. If a stack $\X$ is equivalent to $[\mathcal{G}]$, we say that $\mathcal{G}$ is a \emph{groupoid presentation} of the stack~$\X$. We observe that the $2$-functor~\eqref{eq:bundlization1} in particular gives a map of groupoids
\beq
{\sf Grpd}(\mathcal{G},\mathcal{G}') \to {\Stack}([\mathcal{G}],[\mathcal{G}']). \label{eq:bundlization}
\eeq
The 2-functor $[-]$ map can be understood geometrically using the language of bibundles, e.g., see~\cite[Section 2]{schommerpries_centralext}.
\end{ex}

\begin{ex} The stackification of the action groupoid from \Cref{ex:actgrpd} has as objects over~$S$ pairs $(P,\phi)$ for $P\to S$ a principal $G$-bundle and $\phi\colon P\to N$ a $G$-equivariant map. In particular, $[\pt\sq G]$ is the stack that classifies principal $G$-bundles on the site of supermanifolds. 
\end{ex}

\begin{ex}\label{ex:bibundles} 
Let $G$ be a group acting on a stack $\X$. Then the (stack) quotient $\X\sq G$ is the stack whose $S$-points are principal $G$-bundles $P\to S$ with a $G$-equivariant map $P\to \X$. Morphisms over $S$ are 2-commuting diagrams
\beq
\begin{tikzpicture}[baseline=(basepoint)];
\node (A) at (0,-.75) {$S$};
\node (B) at (3,0) {$P$};
\node (E) at (3,-1.5) {$P'$};
\node (G) at (6,-.75) {$\X$.};
\node (H) at (4,-.75) {$\twocommute$};
\draw[->] (B) to  (A);
\draw[->] (B) to node[left] {$\cong$} (E);
\draw[->] (E) to (A);
\draw[->] (B) to  (G);
\draw[->] (E) to (G);
\path (0,-.75) coordinate (basepoint);
\end{tikzpicture}\nonumber
\eeq
Principal bundles and equivariant maps pull back along base changes $S\to S'$. For more details on group actions on stacks, we refer the reader to Appendix A in \cite{stoffel_eqchernchar}.
\end{ex}

\begin{defn} \label{defn:atlas}
An \emph{atlas} for a stack $\X$ is a supermanifold $U$ and a map $U\to \X$ such that for any map $S\to \X$ from a supermanifold $S$, the 2-pullback $S\times_\X U$ is representable (as a supermanifold) and the canonical map $S\times_\X U\to S$ is a surjective submersion. 
\end{defn} 

There is an equivalent description of an atlas for a stack, given by Behrend and Xu in \cite[Proposition~2.2]{behrendxu_stacks}, that we will use in the proof of \Cref{prop:atlas}: Let $\mathcal{U}$ and~$\mathcal{Y}$ be stacks. A morphism $\mathcal{U}\to \mathcal{Y}$ is an \emph{epimorphism} if for every map $S\to \mathcal{Y}$ there exists a surjective submersion $\tilde{S}\to S$ and a 2-commuting square 
\beq
\begin{tikzpicture}[baseline=(basepoint)];
\node (A) at (0,0) {$\tilde{S}$};
\node (B) at (3,0) {$\mathcal{U}$};
\node (D) at (0,-1.5) {$S$};
\node (E) at (3,-1.5) {$\mathcal{Y}.$};
\draw[->] (A) to  (B);
\draw[->] (A) to (D);
\draw[->] (B) to (E);
\draw[->] (D) to (E);
\node (D) at (1.5,-.75) {$\twocommute$};
\path (0,-.75) coordinate (basepoint);
\end{tikzpicture}\nonumber
\eeq 

\begin{prop}[\cite{behrendxu_stacks}]\label{prop:eqatlas}
An epimorphism $\mathcal{U}\to \mathcal{Y}$ is an atlas for $\mathcal{Y}$ if and only if it satisfies the following conditions:
    \begin{enumerate}
        \item $\mathcal{U}$ is representable,
        \item $\mathcal{U}\times_\mathcal{Y}\mathcal{U}$ is representable, and 
        \item the two canonical maps $\mathcal{U}\times_\mathcal{Y}\mathcal{U}\to \mathcal{U}$ are submersions.
    \end{enumerate} 
\end{prop}

A stack is \emph{geometric} if it admits an atlas. An atlas also determines a groupoid presentation of a stack whose object supermanifold is~$U$ and whose morphisms are the 2-pullback,
\beq
\begin{tikzpicture}
\node (A) at (0,0) {$U\times_\X U$};
\node (B) at (3,0) {$U$};
\node (D) at (0,-1.5) {$U$};
\node (E) at (3,-1.5) {$\X$,};
\draw[->] (A) to node[above] {$s$}  (B);
\draw[->] (A) to node[left] {$t$} (D);
\draw[->] (B) to (E);
\draw[->] (D) to (E);
\end{tikzpicture}\label{diag:weak2}
\eeq 
where $s$ and $t$ denote the source and target maps in the groupoid. 

\begin{defn}\label{defn:finitecover} 
A morphism of stacks $\X\to \Y$ is a \emph{finite cover} if for any map $S\to \Y$ where $S$ is an ordinary supermanifold, the 2-pullback
\beq
\begin{tikzpicture}[baseline=(basepoint)];
\path (0,-.75) coordinate (basepoint);
\node (A) at (0,0) {$\widetilde{S}$};
\node (B) at (3,0) {$\X$};
\node (D) at (0,-1.5) {$S$};
\node (E) at (3,-1.5) {$\Y$,};
\draw[->] (A) to  (B);
\draw[->] (A) to (D);
\draw[->] (B) to (E);
\draw[->] (D) to (E);
\end{tikzpicture}\label{diag:ncovering}
\eeq 
is representable and the map $\widetilde{S}\to S$ is a $n$-sheeted cover. 
\end{defn}

We now turn to a discussion of symmetric powers of stacks, the free symmetric monoidal stack on a given stack, and their atlases. 

\begin{defn}\label{def:symstack}
Let $\X$ be a stack. For any $n\ge 0$, we define the $n^{\rm th}$ symmetric power of $\cX$ to be the quotient stack $\Sym^{n}(\cX) = \cX^{\times n} \sq \Sigma_n$ and write $\Sym^{\le n}(\cX)$ for the disjoint union $\coprod_{i=0}^n\Sym^{i}(\cX)$. 
\end{defn}

The free symmetric monoidal stack on $\cX$ is given by the coproduct
\[
\Sym(\cX) = \coprod_{n \ge 0}\Sym^{n}(\cX) = \coprod_{n \ge 0} (\cX^{\times n} \sq \Sigma_n).
\]
Equivalently, $\Sym(\cX)$ is obtained by taking the free symmetric monoidal groupoid on each $S$-point of $\X$ and then stackifying.

\begin{lem}\label{lem:symatlas}
Let $p\colon U \to \X$ be an atlas for a stack $\X$ and $n \ge 1$. Then the canonical map $U^{\times n} \to \Sym^{n}(\X)$ is an atlas for $\Sym^{n}(\X)$.
\end{lem}
\begin{proof}
Let $S \to \Sym^{n}(\X)$ be a map from a supermanifold $S$ and consider the following diagram of 2-pullbacks:
\[
\begin{tikzpicture}
\node (A) at (0,0) {$S \times_{\Sym^{n}(\X)}U^{\times n}$};
\node (B) at (3,0) {$U^{\times n}$};
\node (C) at (0,-1.5) {$\widetilde{S}$};
\node (D) at (3,-1.5) {$\X^{\times n}$};
\node (E) at (6,-1.5) {$\pt$};
\node (F) at (0,-3) {$S$};
\node (G) at (3,-3) {$\Sym^{n}(\X)$};
\node (H) at (6,-3) {$\pt\sq \Sigma_n$.};
\draw[->] (A) to (B);
\draw[->] (A) to (C);
\draw[->] (B) to node[right] {$p^{\times n}$} (D);
\draw[->] (C) to (D);
\draw[->] (D) to (E);
\draw[->] (C) to (F);
\draw[->] (D) to node[right] {$\pi$} (G);
\draw[->] (E) to (H);
\draw[->] (F) to (G);
\draw[->] (G) to (H);
\end{tikzpicture}
\]
By construction, $\widetilde{S} \to S$ is a $\Sigma_n$-cover of $S$, hence it is a surjective submersion of supermanifolds. Since $p^{\times n}\colon U^{\times n} \to \cX^{\times n}$ is an atlas, the canonical map $S \times_{\Sym^{n}(\X)}U^{\times n} \to \widetilde{S}$ is a surjective submersion of supermanifolds as well. It follows that the composite $S \times_{\Sym^{n}(\X)}U^{\times n} \to S$ is a surjective submersion, so $U^{\times n} \to \Sym^{n}(\X)$ is an atlas as claimed.
\end{proof}

\subsection{Recap of Stolz--Teichner model geometries}\label{appen:model}

\begin{defn}
An \emph{open cover} of a supermanifold~$S$ is a collection of maps $(U_i\to S)$ with the property that
\begin{enumerate}
    \item $((U_i)_{\rm red}\to S_{\rm red})$ is an open cover of the manifold $S_{\rm red}$ in the usual sense and
    \item each $U_i\to S$ is a local isomorphism of supermanifolds.
\end{enumerate} 
\end{defn}

\begin{defn} 
A \emph{model geometry} is the data of 
    \begin{enumerate}
        \item a supermanifold $\M$ and
        \item a super Lie group $\Iso(\M)$ acting on $\M$.
    \end{enumerate} 
We call $\M$ the \emph{model space} and $\Iso(\M)$ the \emph{(group of) isometries}.
\end{defn} 

The following is an abridged version of~\cite[Definitions~2.33 and~4.4]{ST11}; more details can be found there. 

\begin{defn}\label{defn:rigid} An $S$-family of supermanifolds with $\M$-structure is the data of 
    \begin{enumerate}
        \item a smooth submersion $p\colon T\to S$;
        \item a maximal atlas $(U_i)$ of $T$ with charts equipped with isomorphisms over $S$, $\varphi_i \colon U_i \xrightarrow{\sim} V_i \subset S \times \M$ where $V_i$ is open;
        \item transition data $g_{ij}\colon p(U_i\bigcap U_j)\to \Iso(\M)$. 
    \end{enumerate}
The isomorphisms $\varphi_i$ are required to be compatible with transition data $g_{ij}$ and the transition data must satisfy a further cocycle condition. \emph{Isomorphisms} between $S$-families of supermanifolds with $\M$-structure are maps $T\to T'$ over $S$ that on an open cover are determined by the action of~$\Iso(\M)$ on open sub supermanifolds of $\M$. 
\end{defn}

\begin{rmk}\label{rmk:covermodel}
We observe that for an $S$-family $T\to S$ of supermanifolds with $\M$-structure and a covering space $\widetilde{T}\to T$, we obtain a new $S$-family of supermanifolds with $\M$-structure from $\widetilde{T}\to S$: we simply pull back the open cover and transition data that defines the $\M$-structure on $T$. Furthermore, for finite covers~$\widetilde{T}_1,\widetilde{T}_2\to T$ where $T$ is an $S$-family of manifolds with $\M$-structure, an isomorphism of covers $\widetilde{T}_1\to \widetilde{T}_2$ over $T$ is automatically an isomorphism of $\M$-manifolds where the covers are endowed with their canonical $\M$-structure coming from $T$.
\end{rmk}

\begin{rmk} \label{rmk:itsastack}
The category of super manifolds with $\M$-structure determines a category fibered over supermanifolds. This fibered category is in fact a stack (which is implicit in~\cite[\S2.8]{ST11}). Indeed, supermanifolds with $\M$-structure can be pulled back along base changes $S\to S'$. Descent comes from observing that a maximal atlas $(U_j)$ of $T$ (as in \cref{defn:rigid}) together with an open cover $(S_i)$ of $S$ can be refined to a maximal atlas~$(T_k)$ of~$T$ via the pullback square
\beq
\begin{tikzpicture}[baseline=(basepoint)];
\path (0,-.75) coordinate (basepoint);
\node (A) at (0,0) {$\coprod T_k$};
\node (B) at (3,0) {$\coprod U_j$};
\node (D) at (0,-1.5) {$\coprod p^{-1}S_i$};
\node (E) at (3,-1.5) {$T$};
\draw[->] (A) to  (B);
\draw[->] (A) to (D);
\draw[->] (B) to (E);
\draw[->] (D) to (E);
\end{tikzpicture}\nonumber 
\eeq 
that restricts to a maximal atlas on each $T|_{S_i}$. From the existence of this open cover $(T_k)$ that mutually refines $(U_j)$ and $(p^{-1}S_i)$, standard arguments for fiber bundles show that descent is satisfied. Hence, supermanifolds with $\M$ structure form a stack. This type of argument also shows that $\M(\X)$ (defined in~\cref{sec:21}) is a stack, using descent for the stack $\X$ and the fact that a cover $(S_i)$ of $S$ gives the cover $(p^{-1}S_i)$ of $T$.
\end{rmk}

\begin{ex}[Super Euclidean geometries, {\cite[\S4.2]{ST11}}]\label{ex:superEuc} The super Euclidean model spaces arise from the data of:
    \begin{enumerate}
        \item a real vector space~$V$ with inner product;
        \item a complex spinor representation~$\Delta$ of $\Spin(V)$;
        \item and a $\Spin(V)$-equivariant symmetric pairing $\Gamma\colon \Delta\otimes \Delta\to V_\C$
    \end{enumerate} 
The pairing $\Gamma$ endows $V\times \Pi \Delta$ with a group structure given in terms of the maps on $S$-points,
\[
(V\times \Pi \Delta)\times (V\times \Pi \Delta)\to (V\times \Pi \Delta),\quad (v,\sigma)\cdot (v',\sigma')=(v+v'+\Gamma(\sigma,\sigma'),\sigma+\sigma'),
\]
for $(v,\sigma),(v',\sigma')\in (V\times \Pi \Delta)(S)$. When ${\rm dim}_\R(V)=d$ and ${\rm dim}_\C(\Delta)=\delta$, we employ the notation $\bE^{d|\delta}:=V\times \Pi \Delta$. We call this the group of \emph{super Euclidean translations} when viewing it as a group and the \emph{super Euclidean space}, denoted $\R^{d|\delta}$, when viewing it as a supermanifold on which this group acts. Note that $\Gamma$ is part of the data of the group of super Euclidean translations but (as per the standard convention) omitted from the notation $\bE^{d|\delta}$. 

We also have an action of $\Spin(V)$ on $V\times \Pi \Delta$ via the spinor representation on $\Delta$ and through the homomorphism $\Spin(V)\to \SO(V)$ on $V$. This defines a super group $(V\times \Pi \Delta)\rtimes \Spin(V)$, the \emph{super Euclidean isometry group.} The pair $\M=\R^{d|\delta}$ and $\Iso(\M)=\bE^{d|\delta}\rtimes \Spin(\R^d)$ define a \emph{super Euclidean model geometry}. \label{ex:modelsuper}
\end{ex}

\begin{ex}[Super Euclidean tori] \label{ex:supertori}
A choice of lattice $\Z^d\subset V$ gives a subgroup of~$\bE^{d|\delta}$,
\[
\Lambda\colon \Z^d\hookrightarrow V\subset V\times \Pi \Delta\simeq \bE^{d|\delta},
\]
which in turn determines a $\Z^d$-action on $\R^{d|\delta}$. Similarly, an $S$-family of lattices $\Lambda\colon S\times \Z^d\hookrightarrow S\times V$ determines an $S$-family of $\Z^d$-actions on $S\times \R^{d|\delta}$ with quotient defined as 
\[
T^{d|\delta}_{\Lambda}:=(S\times \R^{d|\delta})/\Z^d. 
\]
We observe that the $S$-family $T^{d|\delta}_{\Lambda}\to S$ is uniquely equipped with the structure of a family of super Euclidean manifolds: for a sufficiently fine open cover of $T^{d|\delta}_{\Lambda}$, a component of the preimage of an open set $U_i\subset T^{d|\delta}_{\Lambda}$ along the quotient map
\[
q\colon S\times \R^{d|\delta}\to T^{d|\delta}_{\Lambda}
\]
determines an open subset $V_i\subset S\times \R^{d|\delta}=S\times \M$ for which $q$ restricts to an isomorphism $q|_{U_i}:=\varphi_i\colon U_i\to V_i$. Furthermore, when $U_i\bigcap U_j$ is nonempty, there is a unique map $U_i\bigcap U_j\to \Iso(\M)$ coming from the $\Z^d\subset \bE^{d|\delta}$-action on $\R^{d|\delta}$ that permutes the components of $p^{-1}(U_j)$. These data are compatible by construction. We call the super Euclidean manifolds $T^{d|\delta}_{\Lambda}$ \emph{super Euclidean tori}.
\end{ex}

The examples of model geometries relevant in this paper will come from extending super Euclidean model geometries to include global dilations of $\R^{d|\delta}$; see \Cref{defn:11model} and \Cref{defn:21model}. Since these model geometries come from enlarging the isometry group, the super Euclidean tori define tori in these model geometries as well.

\section{A proof of \Cref{lem:technical}}\label{app:b}

Recall that $\L = \Z^d$. Let $\Sigma_n = \Aut(\underline{n})$, let $\pi \colon G \wr \Sigma_n \to \Sigma_n$ be the canonical surjection, and let $\alpha \colon \L \to G \wr \Sigma_n$ be a group homomorphism. Assume that $\underline{n}=\coprod_{k\in K} I_k$ is a decomposition of $\underline{n}$ into transitive $\L$-sets for the action given by the composite $\pi h \colon \L \to \Sigma_n$ and let $\Sigma_{I_k} = \Aut(I_k)$.

For any $j_k \in I_k$, let
\[
\rho_{j_k} \colon \L \to G \wr \Sigma_n \to G
\]
be the map of sets given by projecting on the $j_{k}^{\text{th}}$-factor of $G^{\times n}$ in the wreath product. Let $\uL_k\subset \uL$ denote the kernel of the composite
\[
\uL\to G\wr \Sigma_{I_k}\to \Sigma_{I_k}.
\]
Thus $\L_k$ is a finite index sublattice of $\uL$. By construction the map $\uL_k \to G \wr \Sigma_{I_k}$ factors through $G^{\times |I_k|}$. Thus $\rho_{j_k} |_{\L_k}$ is a group homomorphism even though $\rho_{j_k}$ is not.

Given a $G$-space $X$, there is a $G \wr \Sigma_n$-action on $X^{\times n}$ and thus an $\L$-action on $X^{\times n}$ through $\alpha$. Write $\tau_{j_k}\colon X^{\times n} \to X$ for the projection onto the factor corresponding to $j_k \in \underline{n}$. Given a fixed point ${\bf x} \in (X^{\times n})^{\im h}$, it follows that $\tau_{j_k}({\bf x}) \in X^{\im \rho_{j_k}|_{\L_k}}$. \Cref{lem:technical} is a consequence of the following result:

\begin{lem}\label{appen:lemma} 
For each $k \in K$, fix $j_k \in I_k$. The map
\[
(X^{\times n})^{\im h} \to \prod_{k\in K} X^{\im \rho_{j_k}|_{\L_k}}
\]
sending ${\bf x}$ to $(\tau_{j_k}(\bf x))$ is a homeomorphism. When $X$ is a $G$-manifold, this map is a diffeomorphism. 
\end{lem}

\begin{proof}
The first reduction is to note that it suffice to prove this for each transitive component of $\pi \alpha$ as there is a factorization of $\alpha$
\[
\xymatrix{ & \prod_{k\in K} G \wr \Sigma_{I_k} \ar[d] \\ \L \ar[ur] \ar[r]^-{\al} & G \wr \Sigma_n.}
\]
Thus we may assume that $\pi \alpha$ is transitive and hence $|K|=1$.

Under this assumption, we will set up some conventions in order to be able to efficiently manipulate the image of $\alpha$. Since $\pi \alpha$ is transitive, $\im \pi \alpha$ is a transitive abelian subgroup of $\Sigma_n$. Let $A = \L/\L_1$, then the induced map $A \hookrightarrow \Sigma_n$ can be identified with the Cayley map $A \to \Sigma_A = \Aut_{\text{Set}}(A)$. As above, for each $a \in A$, we have a map of sets $\rho_a\colon \L \to G$ for $a \in A$. 

Let $l \in \L$ and assume that $\pi \alpha (l) = a$, then
\[
\alpha (l) = ({\bm \rho}(l), a) \in G \wr \Sigma_A,
\]
where ${\bm \rho}(l)$ is the $A$-tuple such that $[{\bm \rho}(l)]_{b} = \rho_b(l)$ for any $b \in A$. Here and below, we write $[-]_b$ for the $b^{\rm th}$-coordinate of an $A$-tuple of elements.  Let $a' \in A$ and pick an $l' \in \L$ with the property that $\pi \alpha(l') =a'$. Then
\[
\alpha(l+l') = ({\bm \rho}(l+l'), aa')
\]
and
\begin{equation}\label{eq:multiplicationformula}
\alpha(l)\alpha(l') = ({\bm \rho}(l), a)({\bm \rho}(l'), a') = ({\bm \rho}(l)(a \cdot {\bm \rho}(l')), aa'),
\end{equation}
where $[a \cdot {\bm \rho}(l')]_b = \rho_{a^{-1}b}(l')$ for each $b \in A$. 

Using this notation, we wish to show that the projection
\[
(X^{A})^{\im \al} \to X^{\im \rho_e|_{\L_1}}
\] 
is a homeomorphism, where $X^A$ denotes the space of functions from $A$ to $X$ and the projection map is induced by evaluation at the identity element $e$ in $A$. We will produce a two-sided inverse to the projection.

For each $a \in A$ fix an element $l_a \in \L$ such that $l_a$ maps to $a \in A = \L/\L_1$. Define the $A$-tuple ${\bm \rho}(l) \in G^{\times |A|}$ by
\[
{\bm \rho}(l) = (\rho_a(l_a))_{a \in A}.
\]
We may use the $G$-action on $X$ to define a map
\[
X^{\im \rho_e|_{\L_1}} \to X^A, \qquad x \mapsto {\bm \rho}(l)x = (\rho_a(l_a)x)_{a \in A}.
\]
This map is continuous as the group action is continuous. 

We will now show that this map lands in $(X^{A})^{\im \al}$. For $({\bf g},b) \in G \wr \Sigma_A$, the action of $G \wr \Sigma_A$ on an element ${\bf x}\in X^A$ is given by
\begin{equation}\label{eq:theaction}
({\bf g},b){\bf x} = (g_ax_{b^{-1}a})_{a\in A}.
\end{equation}

Let $l' \in \L$ be an element mapping to some $b \in A$. We want to show that $\alpha(l'){\bm \rho}(l)x = {\bm \rho}(l)x$. As above, we may write $\alpha(l') = ({\bm \rho}(l'),b) \in G \wr \Sigma_A$. Using \eqref{eq:theaction}, we thus have
\[
\alpha(l'){\bm \rho}(l)x = (\rho_a(l')\rho_{b^{-1}a}(l_{b^{-1}a})x)_{a \in A}.
\] 
It suffices to show that, for each $a \in A$:
\[
\rho_a(l')\rho_{b^{-1}a}(l_{b^{-1}a})x = \rho_a(l_a)x.
\]
Since $x \in X^{\im \rho_e|_{\L_1}}$, we must show that $(\rho_a(l_a))^{-1}\rho_a(l')\rho_{b^{-1}a}(l_{b^{-1}a}) \in \im \rho_e|_{\L_1}$. The key ingredient to seeing this is to note that 
\[
\alpha(-l_{a}) = \alpha(l_a)^{-1} = (a^{-1}\cdot({\bm \rho}(l_a))^{-1},a^{-1}),
\]
which can be verified using \eqref{eq:multiplicationformula}. Moreover, $-l_{a}+l'+l_{b^{-1}a} \mapsto e \in A$ under $\pi\alpha$ and is thus contained in $\L_1$. Therefore we may calculate that
\begin{align*}
\alpha(-l_{a}+l'+l_{b^{-1}a}) 
&= \alpha(-l_a)\alpha(l')\alpha(l_{b^{-1}a}) \\
&= (a^{-1}\cdot({\bm \rho}(l_a))^{-1},a^{-1})({\bm \rho}(l'),b)({\bm \rho}(l_{b^{-1}a}),b^{-1}a) \\ 
&= (a^{-1}\cdot({\bm \rho}(l_a))^{-1},a^{-1})({\bm \rho}(l')(b\cdot{\bm \rho}(l_{b^{-1}a})),a) \\
&=((a^{-1}\cdot({\bm \rho}(l_a))^{-1})(a^{-1}\cdot  ({\bm \rho}(l')(b\cdot{\bm \rho}(l_{b^{-1}a})))),e)
\end{align*}
and
\begin{align*}
[(a^{-1}\cdot({\bm \rho}(l_a))^{-1})(a^{-1}\cdot  ({\bm \rho}(l')(b\cdot{\bm \rho}(l_{b^{-1}a}))))]_e 
& = [a^{-1}\cdot({\bm \rho}(l_a))^{-1}]_e [a^{-1}\cdot  ({\bm \rho}(l')(b\cdot{\bm \rho}(l_{b^{-1}a})))]_e \\
& = (\rho_a(l_a))^{-1}[{\bm \rho}(l')(b\cdot{\bm \rho}(l_{b^{-1}a}))]_a \\
& = (\rho_a(l_a))^{-1}[{\bm \rho}(l')]_a[b\cdot{\bm \rho}(l_{b^{-1}a})]_a  \\
& = (\rho_a(l_a))^{-1}\rho_a(l')\rho_{b^{-1}a}(l_{b^{-1}a}).
\end{align*}
Thus we may conclude that $(\rho_a(l_a))^{-1}\rho_a(l')\rho_{b^{-1}a}(l_{b^{-1}a}) \in \im \rho_e|_{\L_1}$. Consequently, we have produced a map $X^{\im \rho_e|_{\L_1}} \to (X^{A})^{\im \al}$. By construction, this is a section to the projection. To check that $(X^{A})^{\im \al} \to X^{\im \rho_e|_{\L_1}} \to (X^{A})^{\im \al}$ is the identity, it suffices to notice that, for ${\bf x} \in (X^{A})^{\im \alpha}$,
\[
[{\bf x}]_a = [\alpha(l_a){\bf x}]_a = [({\bm\rho}(l_a),a){\bf x}]_a = \rho_a(l_a)[{\bf x}]_e,
\]
where the first equality uses that ${\bf x}$ is fixed by $\im \alpha$.
\end{proof}

\bibliographystyle{amsalpha}
\bibliography{references}

\end{document}